\documentclass[a4paper,oneside, reqno, 10pt]{amsart}
\usepackage{mathrsfs}
\usepackage{amsfonts}
\usepackage{amssymb}
\usepackage{amsxtra}
\usepackage{dsfont}
\usepackage{amsthm}

\usepackage[english,polish]{babel}

\newcommand{\pl}[1]{\foreignlanguage{polish}{#1}}

\usepackage{color}

\usepackage{hyperref} 

\usepackage{enumitem}
\usepackage{stmaryrd} 

\theoremstyle{plain}
\newtheorem{theorem}{Theorem}[section]
\newtheorem*{sq}{Stein's question}
\newtheorem{proposition}{Proposition}[section]

\newtheorem{lemma}[proposition]{Lemma}
\theoremstyle{definition}

\newtheorem{remark}{Remark}[section]
\newtheorem*{remark*}{Remark}

\numberwithin{equation}{section}

\newcommand{\RR}{\mathbb{R}}
\newcommand{\ZZ}{\mathbb{Z}}
\newcommand{\TT}{\mathbb{T}}
\newcommand{\CC}{\mathbb{C}}
\newcommand{\NN}{\mathbb{N}}
\newcommand{\QQ}{\mathbb{Q}}

\newcommand{\ind}[1]{{\mathds{1}_{{#1}}}}
\newcommand{\dist}{\operatorname{dist}}

\newcommand{\sprod}[2] {{#1 \cdot #2}}

\newcommand{\floor}[1]{{\lfloor {#1} \rfloor}}
\newcommand{\vfloor}[1]{{\llbracket {#1} \rrbracket}}

\newcommand{\dif}{\mathrm{d}}
\newcommand{\vo}{{\bf 1}}

\newcommand{\la}{\lambda}

\DeclareMathOperator{\supp}{supp}

\title[Dimension-free estimates for the discrete spherical maximal functions]
{Dimension-free estimates for the discrete \\ spherical maximal functions}

\author{Mariusz Mirek}

\address{Mariusz Mirek \\
  Department of Mathematics\\
  Rutgers University\\
Piscataway, NJ 08854\\ USA \&
	Instytut Matematyczny\\
	Uniwersytet \pl{Wroc{\lll}awski}\\
	Plac Grun\-waldzki 2/4\\
	50-384 \pl{Wroc{\lll}aw}\\
	Poland}
\email{mariusz.mirek@rutgers.edu}

\author{Tomasz Z. Szarek}
\address{ Tomasz Z. Szarek\\
BCAM - Basque Center for Applied Mathematics\\
48009 Bilbao\\
Spain \&
Instytut Matematyczny\\
Uniwersytet \pl{Wroc{\lll}awski}\\
Plac Grun\-waldzki 2/4\\
50-384 \pl{Wroc{\lll}aw}\\
Poland}
\email{tzszarek@bcamath.org}

\author{B{\l}a{\.z}ej Wr{\'o}bel}
\address{ B{\l}a{\.z}ej Wr{\'o}bel\\
	Instytut Matematyczny\\
	Uniwersytet \pl{Wroc{\lll}awski}\\
	Plac Grun\-waldzki 2/4\\
	50-384 \pl{Wroc{\lll}aw}\\
	Poland}
\email{blazej.wrobel@math.uni.wroc.pl}

\thanks{Mariusz Mirek was partially
supported by Department of Mathematics at Rutgers University. Mariusz
Mirek and B{\l}a{\.z}ej Wr{\'o}bel were supported by the National
Science Centre, Poland, grant Opus 2018/31/B/ST1/00204. 
Tomasz Z. Szarek was partially supported by the National
Science Centre, Poland, grant Opus 2017/27/B/ST1/01623.}

\begin{document}
 
\selectlanguage{english}

\begin{abstract}
We prove that the discrete spherical maximal functions (in the spirit
of Magyar, Stein and Wainger) corresponding to the Euclidean spheres
in $\mathbb Z^d$ with dyadic radii have $\ell^p(\mathbb Z^d)$ bounds
for all $p\in[2, \infty]$ independent of the dimensions $d\ge 5$.  An
important part of our argument is the asymptotic formula in the Waring
problem for the squares with a dimension-free multiplicative error
term. By considering new approximating multipliers we will show
how to absorb an exponential in dimension (like $C^d$ for some $C>1$)
growth in norms arising from the sampling principle of Magyar, Stein and
Wainger, and ultimately deduce dimension-free estimates for the
discrete spherical maximal functions.
\end{abstract}

\maketitle

\section{Introduction}

\subsection{Motivations and statement of the main results.}
For $t>0$ let $E_t:=\{y\in \RR^d\colon t^{-1}y\in E\}$ denote the
dilate of a set $E\subseteq \RR^d$.  If
$\mathbb I\subseteq \RR_+:=(0, \infty)$ is a non-empty index set such
that $E_t\cap\ZZ^d\neq\emptyset$ for every $t\in\mathbb I$, then for
every $x\in\ZZ^d$ we define an averaging operator by
\begin{align*}
\mathcal M_t^Ef(x):=\frac{1}{|E_t\cap \ZZ^d|}\sum_{y\in E_t\cap\ZZ^d}f(x-y),
\qquad
f\in\ell^1(\ZZ^d).
\end{align*}
For $\emptyset\neq\mathbb I\subseteq \RR_+$ and $E\subseteq \RR^d$ as
above, and $p\in[1, \infty]$ let
$0<\mathcal C(p,\mathbb I, E)\le \infty$ be the smallest constant in
the following maximal inequality
\begin{align*}
\big\|\sup_{t\in\mathbb I}|\mathcal M^E_tf|\big\|_{\ell^p(\ZZ^d)}
\le \mathcal C(p,\mathbb I, E)\|f\|_{\ell^p(\ZZ^d)},
\qquad
f\in\ell^p(\ZZ^d).
\end{align*}
Note that $\mathcal C(\infty,\mathbb I, E)=1$, since $\mathcal M^E_t$ is an averaging operator,  and
$\mathcal C(p,\mathbb I_1, E)\le \mathcal C(p,\mathbb I_2, E)$ if $\mathbb I_1\subseteq \mathbb I_2$.

In this paper we are mainly concerned with the discrete averaging
operators corresponding to the Euclidean unit spheres
$S:=S^{d-1}:=\{x\in\RR^d \colon |x|=1\}$ in $\RR^d$, where
$S_t:=S_t^{d-1}=\{x\in\RR^d \colon |x|=t\}$.
For $x\in\ZZ^d$ and  
$t \in \sqrt{\NN}:=\{r\in(0, \infty): r^2\in\NN\}$ we shall denote the discrete spherical average by
\begin{align}
\label{eq:27}
\mathcal A^d_{t} f(x):=\mathcal M_t^Sf(x)
=\frac{1}{|S_{t}\cap \ZZ^d|}\sum_{y\in S_{t} \cap \ZZ^d} f(x-y),
\qquad 
f\in\ell^1(\ZZ^d).
\end{align}
We shall also use the convention that $S_0:=\{0\}$, and $\mathcal A^d_{0}f(x):=f(x)$.

Maximal inequalities  corresponding to the discrete averaging
operators \eqref{eq:27}  were extensively investigated by
Magyar \cite{Ma}, and Magyar, Stein and Wainger \cite{MSW}. In the latter work a complete result
was established, which asserts that $\mathcal C(p,\sqrt{\NN}, S^{d-1})<\infty$ if and
only if $p>\frac{d}{d-2}$ and $d\ge5$. The restricted weak-type
endpoint result was also proved for
$\sup_{t\in\sqrt{\NN}}|\mathcal A^d_tf|$ by Ionescu \cite{I1}.

The main purpose of this work is to understand the asymptotic behavior of the  best constants in
maximal inequalities corresponding to discrete spherical averages
\eqref{eq:27} as $d\to\infty$. Namely, we prove dimension-free
estimates for the dyadic maximal function corresponding to
\eqref{eq:27}, where the time parameter $t$ runs over the dyadic set
$\mathbb I=\mathbb D_{\ge1}=\{2^n:n\in\NN_0\}$, where $\NN_0:=\NN\cup\{0\}$.  Our main result can be
formulated as follows:
\begin{theorem}
\label{thm:0}
For every $p\in[2, \infty]$ there exists a constant $C_p>0$ such that
\begin{align}
\label{eq:28}
\sup_{d\ge5}\mathcal C(p,\mathbb D_{\ge1}, S^{d-1})\le C_p.
\end{align}
\end{theorem}

Theorem \ref{thm:0} is motivated by a question of Eli Stein from the
mid 1990's about the dimension-free estimates for the discrete
Hardy--Littlewood maximal functions corresponding to the Euclidean
balls, which using our notation can be stated as follows:
\begin{sq}
Let $B^2:=B^2(d):=\{x\in\RR^d: |x|\le 1\}$ be the Euclidean unit ball centered at the origin. Is it 
true that there is a constant $C>0$ such that
\begin{align}
\label{eq:29}
\sup_{d\in\NN}\mathcal C(2,\RR_+, B^2(d))=\sup_{d\in\NN}\mathcal C(2,\sqrt{\NN_0}, B^2(d))\le C\, ?
\end{align}
\end{sq}

We now give some remarks about Theorem \ref{thm:0} and Stein's question.

\begin{enumerate}[label*={\arabic*}.]
\item In fact, Theorem \ref{thm:0} is a purely
$\ell^2(\ZZ^d)$ result, and by interpolation with
$\ell^{\infty}(\ZZ^d)$ it suffices to prove that there is a constant
$C>0$ such that
\begin{align}
\label{eq:24}
\sup_{d\ge5}\mathcal C(2,\mathbb D_{\ge1}, S^{d-1})\le C.
\end{align}
\item Recently, the first and the third author in a collaboration with
Bourgain and Stein \cite{BMSW2} made a first step towards establishing
\eqref{eq:29} and proved that there is
a constant $C>0$ such that
\begin{align}
\label{eq:31}
\sup_{d\in\NN}\mathcal C(2,\mathbb D_{\ge1}, B^2(d))\le C.
\end{align}
This is a dyadic version of Stein's question, which gives some evidence that 
\eqref{eq:29} might be true. 
\item
An initial goal of \cite{BMSW2} was motivated by the desire to establish \eqref{eq:29} 
by noting a simple inequality
\begin{align}
\label{eq:32}
\sup_{0<t\le T}|\mathcal M_{t}^{B^2}f(x)|\le \sup_{0\le t\le T}|\mathcal A^d_{t}f(x)|,
\qquad T>0,\quad x\in\ZZ^d, \quad f\in \ell^1(\ZZ^d),
\end{align}
where $\mathcal A^d_{0}f(x)=f(x)$.  Since  $S_0=\{0\}$ we easily obtain \eqref{eq:32} thanks to the 
identity
\begin{align*}
\mathcal M_{t}^{B^2}f(x)=\frac{1}{|B^2_t(d)\cap\ZZ^d|}
\sum_{\lambda\in \NN_0 : \lambda\le t^2 }| S_{\sqrt{\lambda}}\cap \ZZ^d|\,\mathcal A^d_{\sqrt{\lambda}} f(x),
\qquad f\in \ell^1(\ZZ^d),
\end{align*}
which is a consequence of the disjoint decomposition 
\begin{align}
\label{eq:48}
B^2_t(d)\cap \ZZ^d=\bigcup_{ \lambda\in \NN_0 : \lambda\le t^2 } S_{\sqrt{\lambda}}^{d-1}\cap \ZZ^d,
\qquad t>0.
\end{align}
Taking into account \eqref{eq:32} it is easy to see that if one could find a constant $C>0$ such that
\begin{align}
\label{eq:33}
\sup_{d\ge5}\mathcal C(2,\sqrt{\NN}, S^{d-1})\le C,
\end{align}
then \eqref{eq:33} would imply a positive answer to Stein's question. Inequality
\eqref{eq:33} is the primary motivation behind this project.
Unfortunately, the authors of \cite{BMSW2} were unable to prove
\eqref{eq:33}. In this paper we returned to this problem with a number of  new
ideas and in Theorem \ref{thm:0} (or more precisely in \eqref{eq:24})
we establish a dyadic variant of \eqref{eq:33}. Appealing to the following estimate
\begin{align}
\label{eq:44}
\qquad\quad\big\|\sup_{t\in\sqrt{\NN}}|\mathcal A^d_tf|\big\|_{\ell^2(\ZZ^d)}\le
\big\|\sup_{t\in\mathbb D_{\ge1}}|\mathcal A^d_tf|\big\|_{\ell^2(\ZZ^d)}
+\bigg\|\Big(\sum_{n\in\NN_0}\sup_{\substack{2^n\le t\le 2^{n+1}\\t\in\sqrt{\NN}}}|\mathcal A^d_tf-\mathcal A^d_{2^n}f|^2\Big)^{1/2}\bigg\|_{\ell^2(\ZZ^d)},
\end{align}
one sees that inequality \eqref{eq:24} can be thought of as the first
step towards proving \eqref{eq:33}, since the first norm in
\eqref{eq:44} is bounded thanks to \eqref{eq:24}. Now the proof of
\eqref{eq:33} is reduced to bounding the square function in
\eqref{eq:44}, which in turn may be reduced to understanding the difference 
\begin{align}
\label{eq:45}
\big\|\mathcal A^d_{\sqrt{\lambda+1}}f-\mathcal A^d_{\sqrt{\lambda}}f\big\|_{\ell^2(\ZZ^d)}.
\end{align}
It is expected that \eqref{eq:45} should be controlled from above by a
constant (independent of $d$) multiple of $\lambda^{-1}$.  A similar
problem is apparent in \cite{BMSW2} while $\mathcal M_t^{B^2}$ is
studied in the context of inequality \eqref{eq:29}, and arguing as
above the matter is reduced to understanding \eqref{eq:45} with
$\mathcal M_t^{B^2}$ in place of $\mathcal A^d_{t}$.  Interestingly,
due to \eqref{eq:48}, this question is also related
to controlling \eqref{eq:45}. Although our methods have severe
limitations and nothing can be said about \eqref{eq:45} at the moment,
we believe that \eqref{eq:32} is the correct idea that motivates this
work and it will be helpful in establishing inequality \eqref{eq:29} in
Stein's question.

\item The idea of using \eqref{eq:32} goes back to Stein
\cite{SteinMax}, see also Stein and Str\"omberg \cite{StStr}, where
the dimension-free estimates for the continuous Hardy--Littlewood
maximal functions corresponding to the Euclidean balls were proved,
see \eqref{eq:41} below.  We shall describe this method in a moment,
as it may give an alternative way to approach \eqref{eq:33}.

\item In the proof of Theorem \ref{thm:0} we recover the asymptotic
formula for the number of lattice points in the spheres
$S_{\sqrt{\lambda}}^{d-1}$ for $d\ge5$, see \cite[Theorem 20.2, p. 456]{IK} and \cite[Theorem 5.7, p. 146]{Nat}.  In fact, in Theorem
\ref{thm:asymptotic} we improved qualitatively and quantitatively the
asymptotic formula in the classical Waring problem for the squares and
obtained the multiplicative error term that satisfies certain
uniformities with respect to radii and dimensions. This part of our
paper may be of independent interest.  More precisely, see \eqref{eq:1}, we prove that
there exists a constant $C>0$ independent of the dimension such that
\begin{align}
\label{eq:37}
|S_{\sqrt{\lambda}}\cap \ZZ^d|= \frac{\pi^{d/2}}{\Gamma(d/2)}\lambda^{d/2-1}\mathfrak{S}_d(\lambda)\big(1+o(1)\big)
\quad \text{ as } \quad \la \ge C d^3 \text{ and } \,  d\to \infty,
\end{align}
and the singular series $\mathfrak{S}_d(\lambda)$ given by
\eqref{eq:40} satisfies
$\frac{1}{2}\le \mathfrak{S}_d(\lambda)\le \frac{3}{2}$ for all
$d\ge 16$ and $\lambda\in\NN$.
We have not found in the existing literature anything about the
uniformities with respect to radii and dimensions in the context of 
the asymptotic formula \eqref{eq:37}. However, we believe it is a
very natural problem, which is interesting in its own right, and might
have been studied in the past.  Another natural question of great
interest arises whether a similar formula holds in \eqref{eq:37} if
$1\le \la \le C d^3$ and $d\to \infty$.  Our method does not work in
this regime.  We hope to investigate this problem in future work, and
also in the context of \eqref{eq:33} as well as Stein's question.

\item We conclude with a brief mention that our method allows us to
verify \eqref{eq:33}  for large scales, see Remark
\ref{rem:100}. Namely, there are universal constants $C, C_3 >0$ such
that
\begin{align*}
\sup_{d\ge5}\mathcal C(2,\sqrt{\NN} \cap (C_3 d^{3/2}, \infty), S^{d-1})\le C.
\end{align*}
However, if $1\le \la \le C_3^2 d^3$
then (in view of \eqref{eq:44}) inequality \eqref{eq:33} is reduced to investigating
\eqref{eq:45}, for which a preliminary step is to estimate
\begin{align}
\label{eq:49}
\big||S_{\sqrt{\lambda+1}}\cap \ZZ^d|-|S_{\sqrt{\lambda}}\cap \ZZ^d|\big|.
\end{align}
At this moment it is not clear whether it is possible to gain any power of $\lambda^{-1}$ in
\eqref{eq:45}.
\end{enumerate}

During the work on article \cite{BMSW2} the authors had even been
thinking about counterexamples for \eqref{eq:33}, since essentially at
the same time it was shown \cite[Theorem~2]{BMSW3} that the
dimension-free phenomenon may fail in the discrete setup.
Surprisingly, one can prove that for every $p\in(1, \infty)$ there is
a constant $C_p>0$ such that for certain ellipsoids $E(d)\subset\RR^d$
and all $d\in \NN$ one has
$\mathcal C(p,\RR_+, E(d)) \ge C_p(\log d)^{1/p}$.  This stands
in sharp contrast to the situation that we know from the continuous
setup \cite{B1, B2, B3, BMSW1, Car1, Mul1, SteinMax, StStr}. We also refer to
the survey articles \cite{DGM1} and \cite{BMSW4} for more exhaustive
exposition of the dimension-free phenomena in the continuous setting.
On the other hand, for the cubes $B^\infty(d):=[-1, 1]^d$ it was also
shown in \cite[Theorem~3]{BMSW3} that for every $p\in(3/2, \infty]$
there is a constant $C_p>0$ such that
$\sup_{d\in\NN}\mathcal C(p,\RR_+, B^{\infty}(d))\le C_p$. For
$p\in(1, 3/2]$ it still remains open whether
$\sup_{d\in\NN}\mathcal C(p,\RR_+, B^{\infty}(d))$ is finite. However,
one can prove that for all $p\in(1, \infty]$ there is a constant
$C_p>0$ such that
$\sup_{d\in\NN}\mathcal C(p,\mathbb D_{\ge1}, B^{\infty}(d))\le C_p$.
All these circumstances were the turning point, which forced the
authors of \cite{BMSW2} to change the above-described strategy in
3. and find a different way to prove \eqref{eq:31}. We adapt the
scheme of the proof and some of these strategies from \cite{BMSW2} to prove Theorem
\ref{thm:0} or more precisely \eqref{eq:24}.  Although the methods
developed in \cite{BMSW2} are important in this paper, there are some
novel ideas in our approach that we now highlight:
\begin{enumerate}[label*={\arabic*}.]
\item As opposed to the situation considered in \cite{BMSW2}, here we
use a variant of the Hardy--Littlewood circle method to analyze the
Fourier multipliers \eqref{eq:116} corresponding to the spherical
averages $\mathcal A_t^d$. This is a consequence of a more singular
nature of averages $\mathcal A_t^d$, which is noticeable in the fact
that the family of spheres $(S^{d-1}_t)_{t\in\RR_+}$ fails to be
nested in contrast to the family of balls $(B^2_t(d))_{t\in\RR_+}$,
i.e. if $t_1<t_2$ then $S^{d-1}_{t_1}\cap S^{d-1}_{t_2}=\emptyset$,
whereas $B^2_{t_1}(d)\subset B^2_{t_2}(d)$. As a matter of fact, we follow the ideas of
Magyar, Stein and Wainger \cite{MSW} and adjust their approach to the
dimension-free problem. This is a delicate process described in
Theorem \ref{thm:asymptotic}, where it was essential to keep track
carefully of how the constants arising in the error terms of the
underlying circle method depend on $d$. As a result we have obtained
\eqref{eq:37} with a multiplicative dimension-free error term, which
allows us to circumvent the absence of a dilatation structure on
$\ZZ^d$. While working on \cite{BMSW2} the authors were not able to
detect the correct relationship between radii $\lambda$ and dimensions
$d$ that would guarantee \eqref{eq:37}.   This was one of the obstacles
why the strategy described in 3. had been given up in \cite{BMSW2}.

\item An important tool of the Magyar, Stein and Wainger paper
\cite{MSW} is the sampling principle, see \cite[Corollary 2.1]{MSW}
(or Proposition \ref{prop:MSW1} below), which is a general abstract
theorem that allows one to compare $L^p(\RR^d)$ boundedness of certain
convolution operators on $\RR^d$ with $\ell^p(\ZZ^d)$ boundedness of
analogous operators on $\ZZ^d$. Although very useful in many discrete
problems, a literal application of the sampling principle, as in
\cite{MSW}, to our problem makes no sense, since it produces an
exponential in dimension (like $C^d$) growth in norm which is
prohibited, see Proposition \ref{prop:MSW1}.  This also prevented the authors of \cite{BMSW2} to use
the spherical maximal function to prove \eqref{eq:29}.  In the remarks
after \cite[Proposition 2.1]{MSW} the authors ask whether the constant
in \cite[Corollary 2.1]{MSW} can be taken to be independent of $d$ or
even $C=1$. To the best of our knowledge, if the multiplier in the
sampling principle \cite{MSW} takes values in the space of bounded
linear operators between two finite-dimensional Banach spaces $B_1$
and $B_2$, the question about the dimension-free
$\ell^p(\ZZ^d; B_1)\to \ell^p(\ZZ^d; B_2)$ estimates in the sampling
principle \cite{MSW} for all $p\in[1, \infty]$ is still open. However, if $B_1$ and $B_2$ are
finite-dimensional Hilbert spaces Kovrizhkin \cite{Kov}, and also
recently Tao \cite{Tao} gave independently negative answers to this
question as long as $p$ is sufficiently close to $1$ or $\infty$. The
answer probably is negative for all $p\neq2$, but at this moment it is
open. In the Hilbert space setup if $p=2$ the $C^d$ factor from the
samplng principle in \cite{MSW} may be deleted by a simple application
of the Plancherel theorem. In our case the situation is more
complicated since we work with non-Hilbert spaces.  Even for $p=2$ it
is suspected that the constant in the sampling principle \cite{MSW}
(in the non-Hilbert setting) is no longer $1$, and in fact depends
exponentially on the dimension, but we do not have a proof of this.
Working on the current project, along the way, we come across some
perhaps unexpected property of approximating multipliers, which
essentially permits us to eat the exponential growth in dimension from
the sampling principle. Specifically, we modified the approximating
multiplier from \cite{MSW} by considering new multipliers
\eqref{eq:atpq} and \eqref{eq:btpq}, which produce acceptable error
terms \eqref{eq:mtxidec}, and absorb the exponential growth arising in
the sampling principle, see Theorem \ref{thm:1'_s2}.  We hope that our
approach will be also useful when the dimension-free estimates will be
discussed for $p\neq2$.

\item Finally, in Lemma \ref{lem:aux_4} and Lemma \ref{lem:aux_5} we
provided estimates of the Fourier transform corresponding to the
continuous spherical measure on $S^{d-1}$ that lead to the
dimension-free bounds in Lemma \ref{lem:PerSm} and Lemma
\ref{lem:PerSm'}.  Although the Fourier transform estimates of the
surface measures received considerable attention over the years, much
to our surprise it seems that no extensions delivered in Lemma
\ref{lem:aux_4} and Lemma \ref{lem:aux_5} have appeared in the
literature.  The estimates from Lemma \ref{lem:aux_5} are actually
very much in the spirit of Bourgain's result \cite[eq. (10), (11),
(12), p. 1473]{B1}, where the estimates of multipliers associated with
the Hardy--Littlewood averages over convex symmetric bodies in $\RR^d$
are provided in terms of the corresponding isotropic
constants. Interestingly, in contrast to Bourgain's proof \cite{B1}
the proofs of our results are based on elementary manipulations of the
Bessel functions, (like change the contour of integration, see Lemma
\ref{lem:aux_4}) and uniform estimates of the Bessel functions that
lead to the conclusion of Lemma \ref{lem:aux_5}.  Lemma
\ref{lem:aux_5} plays an essential role in the estimates of
multipliers \eqref{eq:atpq} and \eqref{eq:btpq}. This is a new
ingredient which was not apparent in \cite{MSW}.
\end{enumerate}

Since the Magyar, Stein and Wainger paper \cite{MSW} is critical in
this paper and has had a profound impact on the discrete harmonic
analysis we conclude with a brief mention about the current state of
the art in the related areas. Kesler, Lacey and Mena \cite{KLM},
\cite{KLM2} started to develop sparse estimates in the context of
discrete spherical averages. It is a very successful line of research,
which significantly enhanced the field of discrete harmonic
analysis. In \cite{KLM2} conjecturally sharp sparse bounds for
$\sup_{\lambda\in\sqrt{\NN}}|\mathcal A_{\lambda}^{d}f|$ were proved,
and used to give a new proof of Magyar, Stein and Wainger theorem
\cite{MSW} as well as the endpoint result of Ionescu \cite{I1} in a
fairly unified way. Another interesting line of research has been
initiated by Hughes \cite{H1}, who asked about the bounds for
$\mathcal C(p, \mathbb L, S^{d-1})$, where $\mathbb L\subset\NN$ is a
lacunary set. Hughes also observed (even though the
Magyar--Stein--Wainger theorem \cite{MSW} is sharp) that it also makes
sense to study $\mathcal C(p, 2\NN+1, S^{d-1})$ for $d=4$ upon
restricting the radii $\lambda$ to odd integers. Specifically, Hughes
\cite{H1} constructed a very sparse set $\mathbb L \subset\NN$ of
radii such that $\mathcal C(p, \mathbb L, S^{d-1})<\infty$ for
$\frac{d}{d-2}\le p\le \infty$ and $d\ge 4$. The latter result was
recently extended by Kesler, Lacey and Mena \cite{KLM}, where it was
shown that $\mathcal C(p, \mathbb L, S^{d-1})<\infty$ for any lacunary
sequence $\mathbb L\subset \NN$ and any $\frac{d-2}{d-3}< p\le \infty$
with $d\ge5$. The case $d=4$ was recently established by Anderson and
Madrid \cite{AM} for $\frac{d+1}{d-1}< p\le \infty$ and all lacunary
sequences $\mathbb L\subset \sqrt{\NN\setminus 4\NN}$.  Cook and Hughes
\cite{CH} studied similar problems in the context of Birch forms, and
specifically recovered the main result from \cite{KLM}.  They also
illustrated in \cite{CH}  that
for any $1<p< \frac{d}{d-1}$ there exists a set of lacunary radii
$\mathbb L\subset\NN$ such that
$\mathcal C(p, \mathbb L, S^{d-1})=\infty$. This negative result remains 
also true \cite{CH} for averages over more general forms in the spirit
of Birch. This is a remarkable phenomenon that exhibit some peculiar
features in the discrete world.  Finally, it is worth noting that the
negative result from \cite{CH} does not exclude positive results for
$1<p< \frac{d}{d-1}$. Namely, Cook showed that
$\mathcal C(p, \mathbb L, S^{d-1})<\infty$ for all $1<p\le \infty$ and
$d\ge 5$ by constructing a very sparse sequence of radii \cite{C3}, he
also showed a similar phenomenon \cite{C1} for the averages associated
to a certain class of homogeneous algebraic hypersurfaces.

\subsection{Dimension-free estimates in the continuous setting}
We now make a link between the strategy described above with the ideas
of the proof of dimension-free estimates for the Hardy--Littlewood
maximal functions associated with the continuous averaging operators
over the Euclidean balls in $\RR^d$.

For every $t>0$ and $x\in\RR^d$ we define the continuous Hardy--Littlewood averaging operator by
\begin{align*}
M_t^{B^2}f(x):=\frac{1}{|B^2_t|}\int_{B^2_t}f(x-y)dy, \qquad f\in L^1_{\rm loc}(\RR^d).
\end{align*}
For $\emptyset\neq\mathbb I\subseteq \RR_+$  and $p\in[1, \infty]$ let
$0<C(p,\mathbb I, B^2(d))\le \infty$ be the smallest constant in
the following maximal inequality
\begin{align}
\label{eq:34}
\big\|\sup_{t\in\mathbb I}|M_t^{B^2}f|\big\|_{L^p(\RR^d)}\le C(p,\mathbb I, B^2(d))\|f\|_{L^p(\RR^d)},
\qquad f\in L^p(\RR^d).
\end{align}
Using a standard covering argument for $p=1$ and interpolation with $p=\infty$  it is not hard
to see that $C(p,\RR_+, B^2(d))<\infty$ for every $p\in(1, \infty]$, since $C(\infty,\RR_+, B^2(d))=1$.
Stein \cite{SteinMax} (see
also Stein and Str\"omberg \cite{StStr}) proved that there
exists a constant $C_p>0$ depending only on $p\in(1, \infty]$ such
that
\begin{align}
\label{eq:41}
\sup_{d\in\NN}C(p,\RR_+, B^2(d))\le C_p.
\end{align}
The key idea from \cite{SteinMax, StStr} to establish \eqref{eq:41} is
to use the spherical averaging operator, defined for any $t>0$ and
$x\in\RR^d$ by
\begin{align*}
A_t^d f(x):=\int_{S^{d-1}} f(x-t\theta)d\mu^d(\theta), \qquad f\in C_c(\RR^d),
\end{align*}
where $\mu^d$ denotes the normalized surface measure on
$S^{d-1}$, see \eqref{eq:42} below. We now see that the spherical operator $\mathcal A_t^d$
from \eqref{eq:27} is a discrete analogue of the operator $A_t^d$.
Let $C(p, \mathbb I, S^{d-1})$ be the best constant in inequality
\eqref{eq:34} with $A_t^d$ in place of $M_t^{B^2}$.  It is very well
known from the results of Stein \cite{Ste0} for all $d\ge 3$, and
Bourgain \cite{Bou0} for $d=2$ that $C(p, \RR_+, S^{d-1})<\infty$ if
and only if $\frac{d}{d-1}<p\le \infty$. We also know that
$C(p, \mathbb D, S^{d-1})<\infty$ for all $1<p\le \infty$ as it was
shown by Calder{\'o}n \cite{Cal} and independently by Coifman and
Weiss \cite{CW}. Using polar coordinates one easily sees that
\begin{align}
\label{eq:68}
\sup_{t>0}|M_{t}^{B^2}f(x)|\le \sup_{t>0}|A_{t}^df(x)|.
\end{align}
Now the method of rotations enables one to view high-dimensional
spheres as an average of rotated low-dimensional ones, and
consequently one can conclude that for every $d\ge 2$ and
$p>\frac{d}{d-1}$ we have
\begin{align}
\label{eq:53}
C(p, \RR_+, S^{d})\le C(p, \RR_+, S^{d-1}) < \infty.
\end{align}
Hence, the sequence $(C(p, \RR_+, S^{d-1}))_{d\in\NN}$ is
non-increasing, and in particular bounded, in $d > \frac{p}{p-1}$.  Therefore,
to prove \eqref{eq:41} we fix $p\in(1, \infty)$ and pick the smallest integer $d_0\in\NN$ such that
$d_0>\frac{p}{p-1}$. We can assume, without loss of generality, that
$d>d_0$, then by  \eqref{eq:68},
\eqref{eq:53} and Stein's result we conclude
\begin{align*}
C(p,\RR_+, B^2(d))\le C(p, \RR_+, S^{d-1})\le  C(p, \RR_+, S^{d_0-1}) < \infty,
\end{align*}
and \eqref{eq:41} follows.

Thinking about Stein's question \eqref{eq:29} the authors of
\cite{BMSW2} tried to adapt the ideas of the proof of \eqref{eq:41} to
the discrete setting. Although in \eqref{eq:32} there is a discrete
analogue of \eqref{eq:68} it is completely unclear whether there is a
discrete analogue of \eqref{eq:53}. 
More precisely, it is interesting to know whether
for every $d\ge5$ and
$p>\frac{d}{d-2}$ the following is true
\begin{align}
\label{eq:69}
\mathcal C(p, \RR_+, S^{d})\le \mathcal C(p, \RR_+, S^{d-1}).
\end{align}
The estimate \eqref{eq:69} is not easy even for $p=2$, mainly due to
the lack of the dilatation structure and the method of rotation on
$\ZZ^d$, which were essential to establish \eqref{eq:41} in
\cite{SteinMax, StStr}.  A natural remedy for the first of these
obstacles is the asymptotic formula with a dimension-free
multiplicative error term as in \eqref{eq:37}, whereas for the second
one are symmetries of $S_{\sqrt{\lambda}}\cap \ZZ^d$:
\begin{enumerate}[label*={(\alph*)}] 
\item\label{sym:1} If $(x_1,\ldots, x_d)\in S_{\sqrt{\lambda}}\cap \ZZ^d$, then
$(\varepsilon_1x_1,\ldots, \varepsilon_dx_d)\in S_{\sqrt{\lambda}}\cap \ZZ^d$ for any  $(\varepsilon_1, \ldots, \varepsilon_d)\in\{-1, 1\}^d$.
\item\label{sym:2} If $(x_1,\ldots, x_d)\in S_{\sqrt{\lambda}}\cap \ZZ^d$, then  $(x_{\tau(1)},\ldots, x_{\tau(d)})\in S_{\sqrt{\lambda}}\cap \ZZ^d$ for any
permutation $\tau\in{\rm Sym}(d)$.
\end{enumerate}
These kind of symmetries were also strongly exploited in \cite{BMSW2} for the discrete Euclidean balls.

Let us emphasize that if we could prove \eqref{eq:69} then it would
imply \eqref{eq:33}, and moreover, in view of the Magyar, Stein and
Wainger theorem \cite{MSW} we would be able to give an affirmative
answer to Stein's question even for all $p\in(1, \infty]$ in place of
$2$ in \eqref{eq:29}, which would be a genuine discrete analogue of inequality
\eqref{eq:41}.

Finally, we remark that the method from \cite{SteinMax, StStr} is limited
to the Euclidean balls. The case of general convex symmetric bodies
$G\subset\RR^d$ (which means that $G$ is a convex, compact subset of
$\RR^d$ which is symmetric and has a non-empty interior), requires a
different approach. Let $C(p,\mathbb I, G)$ be the best constant in
\eqref{eq:34}, where $B^2$ is replaced with a general convex symmetric
bodies $G\subset\RR^d$.

Stein's work \cite{SteinMax} gave rise to the
famous conjecture, which asserts that for every $p\in(1, \infty]$
there is a constant $C_p>0$ such that
\begin{align}
\label{eq:36}
\sup_{d\in\NN}\sup_{G\in\mathfrak B(d)}C(p,\RR_+, G)\le C_p,
\end{align}
where $\mathfrak B(d)$ is the set of all convex symmetric bodies in $\RR^d$.
This problem has been studied for four decades by several authors.  We
now briefly list the current state of the art concerning
\eqref{eq:36} as well as its relations to dimension-free phenomena in the discrete setting.
\begin{enumerate}[label*={\arabic*}.]
\item Bourgain \cite{B1}, \cite{B2}, and independently  Carbery
\cite{Car1}, proved \eqref{eq:36} for $p\in(3/2, \infty]$. They also
showed a dyadic variant of \eqref{eq:36} for $p\in(1, \infty]$ with
$\mathbb D:=\{2^n:n\in\ZZ\}$ in place of $\RR_+$. Although inequality \eqref{eq:36}  for $p\in(1, 3/2]$
remains still open,  the case of $q$-balls is quite well understood.
\item For the $q$-balls $B^q(d)$ (see definition \eqref{eq:88}) we know that for every
$p\in(1, \infty]$ and for every $q\in[1, \infty]$ there is a constant
$C_{p, q}>0$ such that
$\sup_{d\in\NN}C(p,\RR_+, B^q(d))\le C_{p, q}$. This was established
by M\"uller in \cite{Mul1} (for $q\in [1, \infty)$) and by Bourgain in
\cite{B3} (for cubes $q=\infty$).
\item Dimension-free estimates in the discrete setting were initiated
by the first and the third authors in a collaboration with Bourgain
and Stein \cite{BMSW3}, \cite{BMSW4} and \cite{BMSW2}. Recently, the
first and the third authors in a collaboration with Kosz and Plewa
\cite{KMPW} proved that for every $G\in\mathfrak B(d)$ one has
\begin{align}
\label{eq:51}
C(p,\RR_+, G)\le \mathcal C(p,\RR_+, G),
\quad \text{ for all } \quad
p\in[1, \infty],
\end{align}
where the case $p=1$ corresponds to the optimal constants in the weak
type $(1, 1)$ inequalities respectively in $\RR^d$ and $\ZZ^d$. It was
also shown in \cite{KMPW} that
$C(1,\RR_+, B^{\infty}(d))= \mathcal C(1,\RR_+, B^{\infty}(d))$, which
in view of Aldaz's result \cite{Ald1} yields
$\mathcal C(1,\RR_+, B^{\infty}(d))\ _{\overrightarrow{d\to\infty}}\ \infty$. Thus
the boundedness of $\sup_{d\in\NN}\mathcal C(p,\RR_+, B^{\infty}(d))$
for $p\in(1, 3/2]$ cannot be deduced by interpolation from
\cite[Theorem~3]{BMSW3} for $p\in(3/2, \infty]$.
Inequality \eqref{eq:51} exhibits a well known phenomenon in harmonic
analysis, which states that it is harder to establish bounds for
discrete operators than the bounds for their continuous counterparts,
and this is the best what we could prove in this generality.
\end{enumerate}

\subsection{Overview of the methods and proofs}
We now give an overview of the proofs of our main results.  The proof
of Theorem \ref{thm:0} is similar in spirit to the proof of inequality
\eqref{eq:31} (we refer to \cite{BMSW2} for the details), but with
several new difficulties arising that require new ideas to overcome.
The proof of Theorem \ref{thm:0} will consist of three
steps. Specifically, we shall prove the following dimension-free
estimates, which will result in inequality \eqref{eq:28}. 

\begin{theorem}
\label{thm:00}
There exist absolute constants $C, C_0, C_1, C_2, C_3>0$ such that the following is true. 
\begin{enumerate}[label*={\arabic*}.]
\item The large-scale estimate holds
\begin{align}
\label{eq:ls}
\sup_{d\ge5}\mathcal C(2,\mathbb D_{C_3, \infty}, S^{d-1})\le C,
\end{align}
where $\mathbb D_{C_3, \infty}:=\{t\in\mathbb D_{\ge1}:t\ge C_3d^{3/2}\}$.
\item The intermediate-scale estimate holds
\begin{align}
\label{eq:is}
\sup_{d\ge5}\mathcal C(2,\mathbb D_{C_1, C_2}, S^{d-1})\le C,
\end{align}
where
$\mathbb D_{C_1, C_2}:=\{t\in\mathbb D_{\ge1}:C_1d^{1/2}\le t\le C_2d^{3/2}\}$.
\item The
small-scale estimate holds
\begin{align}
\label{eq:ss}
\sup_{d\ge5}\mathcal C(2,\mathbb D_{C_0}, S^{d-1})\le C,
\end{align}
where $\mathbb D_{C_0}:=\{t\in\mathbb D_{\ge1}:1\le t\le C_0d^{1/2}\}$.
\end{enumerate}
\end{theorem}
Since we are working with the dyadic numbers the exact values of
$C_0, C_1, C_2, C_3$ will never play a role as long as they are
absolute constants. Indeed, if we establish \eqref{eq:ls},
\eqref{eq:is} and \eqref{eq:ss} for some constants
$C, C_0, C_1, C_2, C_3>0$, then \eqref{eq:ls}, \eqref{eq:is} and
\eqref{eq:ss} remain true with any other set of constants
$C_0', C_1', C_2', C_3'>0$ in place of $C_0, C_1, C_2, C_3$ and some
constant $C'>0$, that may depend on $C_0, C_1, C_2, C_3$, in place of
$C$. Taking into account this remark we may always adjust $C_0=C_1$ and
$C_2=C_3$ in Theorem \ref{thm:00}, which immediately yields Theorem \ref{thm:0}.  In
what follows, the implied constants will be always allowed to depend
on $C_0, C_1, C_2, C_3$ and we shall be mainly concerned with proving Theorem \ref{thm:00}.

If we restrict ourselves to small dimensions $5\le d<16,$ then Theorem
\ref{thm:00} holds by \cite{MSW}. Therefore throughout the proof of
Theorem \ref{thm:00} it suffices to assume that $d\ge 16.$

The proof of Theorem \ref{thm:00} uses a variety of Fourier methods to
understand the multipliers $\mathfrak m_t$ (see \eqref{eq:116})
corresponding to the averages $\mathcal A^d _t$, with
$t^2=\la \in \NN$. In Section \ref{sec:3'} that handles the
large-scale case \eqref{eq:ls}, when $t\ge C_3 d^{3/2}$, we follow the
ideas of Magyar, Stein and Wainger \cite{MSW} and use a variant of the
Hardy--Littlewood circle method to establish the asymptotic formula
for $\mathfrak m_t$ for $t\ge C_3 d^{3/2}$, see Theorem
\ref{thm:asymptotic}. In the proof of Theorem \ref{thm:asymptotic} we
had to keep track carefully of how the constants in all error terms
depend on the dimension.  Only an exponential growth of the form $C^d$
is admissible, which consequently determines of how large the constant
$C_3> 0$ in \eqref{eq:ls} must be. This is the novelty of this paper,
which can be thought of as a dimension-free variant of the circle
method. To the best of our knowledge this aspect has not been
discussed in the literature.  An important consequence of Theorem
\ref{thm:asymptotic} is the asymptotic formula \eqref{eq:37} with a
dimension-free multiplicative error term, which permits us to overcome
the problem with the absence of dilatation structure on $\ZZ^d$.
Theorem \ref{thm:asymptotic} is a key building block of Proposition
\ref{pro:mtxidec}, where the approximating multipliers \eqref{eq:atpq}
and \eqref{eq:btpq} are defined. This is the place, where our approach
diverges from the method presented in \cite{MSW}. By introducing
decomposition \eqref{eq:mtxidec} in Proposition \ref{pro:mtxidec} and
manipulating the parameter $n$ in the definition of the second
multiplier \eqref{eq:btpq} we were able to absorb (using the decay
from the Gauss sums \eqref{eq:2}) the exponential dimension dependence
produced by the sampling principle of Magyar, Stein and Wainger
\cite[Corollary 2.1]{MSW}, see Theorem \ref{thm:1'_s2}.  Although the
maximal function corresponding to the second approximating multiplier
\eqref{eq:btpq} is handled in Theorem \ref{thm:1'_s2}, some price must
be paid, since the sampling principle cannot be used to bound the
maximal function corresponding to the first approximating multiplier
\eqref{eq:atpq}. Fortunately the latter maximal function is estimated
in Theorem \ref{thm:1'_s1}.  Here is the place, where the
 estimates from Lemma \ref{lem:aux_5} (in the spirit of
Bourgain's isotropic-constant-estimates from \cite{B1}) of the Fourier
transform of continuous spherical measures enter into play.  The
estimates from Lemma \ref{lem:aux_4} and Lemma \ref{lem:aux_5} may be
new and of independent interest.  We hope that these observations will
be useful in extending dimension-free estimate \eqref{eq:24} for
$p\neq2$. It is noteworthy that the large-scale case for the 
Hardy--Littlewood maximal function from \cite{BMSW2} is a simple
consequence of the comparison principle from \cite[Theorem 1]{BMSW3},
since the averages $\mathcal M_t^{B^2}$ are defined over the nested
family of discrete Euclidean balls $B_t^2(d)\cap\ZZ^d$.  This
contrasts sharply with the above-described situation of the maximal
spherical averages, which being more singular required a
dimension-free variant of the circle method that we developed to
circumvent various difficulties with the lack of nestedness for the
spheres.  We finally remark that Theorem \ref{thm:asymptotic} is also
critical in the proof of Proposition \ref{prop:mtinfla}, which in turn
plays a pivotal role in the intermediate-scale estimate \eqref{eq:is}.

In Section \ref{sec:me}, we handle the intermediate-scale case
\eqref{eq:is}, when $C_1d^{1/2}\le t\le C_2d^{3/2}$. Here we mainly follow the ideas developed  in \cite[Section
2]{BMSW2}. However, there are subtle adjustments necessary to fit the
arguments to the new situation. The multiplier \eqref{eq:116} has an additional symmetry
\begin{align}
\label{eq:74}
\mathfrak m_t (\xi+(1/2,\ldots, 1/2)) = (-1)^{t^2}\mathfrak m_t (\xi),
\qquad \xi \in \TT^d, \quad t \in \sqrt{\mathbb{N}},
\end{align}
which causes new complications.  Proposition \ref{prop:0} and
Proposition \ref{prop:2} reveal the importance of property
\eqref{eq:74}, which is reflected by the appearance of various
quantities depending on $\xi+(1/2,\ldots, 1/2)$ in \eqref{eq:22'} and
\eqref{eq:23}. The proofs of Proposition \ref{prop:0} and Proposition
\ref{prop:2} exploit the symmetries of spheres $S_t\cap\ZZ^d$
described in \ref{sym:1} and \ref{sym:2}. The invariance of the
spheres $S_t\cap\ZZ^d$ under the permutation group \ref{sym:2} brought
into play probabilistic tools, which are especially important in the
proof of inequality \eqref{eq:23} from Proposition \ref{prop:2} as well as in Lemma \ref{lem:13}.
There are three results: Lemma \ref{lem:5}, Lemma \ref{lem:8} and Lemma
\ref{lem:9}, necessary for the proof of inequality \eqref{eq:23} to
work.  Lemma \ref{lem:5} asserts that an essential amount of mass of
the sphere $S_t\cap\ZZ^d$ is concentrated in the region where the
coordinates are large. Its proof can be reduced to the corresponding
result for balls from \cite{BMSW2} upon proving Lemma \ref{lem:lfs},
which gives a simple  comparison between numbers of lattice points in
balls and spheres. Lemma \ref{lem:8}, as in \cite{BMSW2}, leads us to the
so-called decrease dimension trick in Lemma \ref{lem:9}. The decrease
dimension trick resemble to some extent the method of rotations
\cite{SteinMax, StStr} from the continuous setting that enables one to
view high-dimensional spheres as an average of rotated low-dimensional
ones. Here we use Lemma \ref{lem:9} to view the original multiplier
$\mathfrak m_t$ as an average of new multipliers $\mathfrak m_l^{(r)}$
(see \eqref{eq:21}) associated with spheres $S_l^{r-1}\cap\ZZ^r$ in
lower dimensional spaces whose radii satisfy the relation
$l\ge C_3r^{3/2}$ from the previous case with respect to the new
dimensions $r\in\NN$ such that $r\le d$. Now the machinery from the
previous section can be used. Specifically, Proposition
\ref{prop:mtinfla} can be applied to the $\mathfrak m_l^{(r)}$, and
this with the aid of a variant of the convexity lemma established in \cite{BMSW2} (see
Lemma \ref{lem:7}) yields Proposition \ref{prop:2}.

Finally, in Section \ref{sec:sm} we handle the small-scale case
\eqref{eq:ss}, when $1\le t\le C_0d^{1/2}$. We proceed much the same
way as in \cite[Section 3]{BMSW2}. We find suitable approximating
multipliers and reduce the matters to the square function estimates
using Proposition \ref{prop:4}. Then we establish Lemma \ref{lem:15},
which states that an essential percentage of the mass of
$S_t\cap\ZZ^d$ is concentrated on the Hamming cube
$\{-1,0, 1\}^d$. This lemma can be easily deduced from its ball
counterpart \cite[Lemma 3.2]{BMSW2} upon invoking Lemma \ref{lem:lfs}.
Lemma \ref{lem:15} also shows that $\mathfrak m_t$ is closely related to the
Krawtchouk polynomial \eqref{eq:38}, see \cite{HKS}.  Using a uniform
bound for the Krawtchouk polynomials (see Property \ref{item:5} in
Theorem \ref{thm:100}) we are able to deduce in Proposition
\ref{prop:5} a decay of the multipliers $\mathfrak m_t$ at infinity.
Proposition \ref{prop:5} is the main result of Section \ref{sec:sm} and its proof
follows very closely the proof of corresponding result for the balls from 
\cite[Section 3]{BMSW2}, therefore we refer to  \cite{BMSW2} for more details.

\section{Notation}
Here we set out some basic notation that will be used throughout the paper.
The letter $d\in\NN$ is reserved for the dimension throughout this paper.
\subsection{Basic notation} The sets $\ZZ$, $\RR$, $\CC$ and $\TT:=\RR/\ZZ$ have standard
meaning.
\begin{enumerate}[label*={\arabic*}.]
\item We use $\NN:=\{1,2,\ldots\}$ and $\NN_0 := \NN\cup\{0\}$ to
denote the sets of positive integers and non-negative integers,
respectively.  We also set $\NN_N := \{1, 2, \ldots, N\}$ for any
$N \in \NN$.

\item Let $\mathbb D:=\{2^n: n\in\ZZ\}$ denote the set of all dyadic
numbers in $\RR_+:=(0, \infty)$ and the set of all dyadic numbers in
$\NN$ will be denoted by $\mathbb D_{\ge1}:=\{2^n: n\in\NN_0\}$.

\item For any $x\in\RR$ we will use the floor function
$\lfloor x \rfloor: = \max\{ n \in \ZZ : n \le x \}$, the fractional
part $\{x\}:=x-\lfloor x\rfloor$ and the distance to the nearest
integer $\|x\|:={\rm dist}(x, \ZZ)$.

\item For further reference observe that for $\eta\in\TT$ one has
$2\|\eta\|\le|\sin(\pi\eta)|\le \pi\|\eta\|$,
 since 
$|\sin(\pi\eta)|=\sin(\pi\|\eta\|)$ and 
\begin{align}
\label{eq:103}
2|\eta|\le|\sin(\pi\eta)|\le \pi|\eta|,
\quad\text{ for } \quad
0\le|\eta|\le 1/2.
\end{align}

\item We use $\ind{A}$ to denote the indicator function of a set $A$. If $S$ 
is a statement we write $\ind{S}$ to denote its indicator, equal to $1$
if $S$ is true and $0$ if $S$ is false. For instance $\ind{A}(x)=\ind{x\in A}$.

\item For $d\in\NN$ let ${\rm Sym}(d)$ be the permutation group on $\NN_d$. We know
that $|{\rm Sym}(d)|=d!$.  For $A\subseteq{\rm Sym}(d)$ let
$\mathbb P[A]:={|A|}/{d!}$ be the uniform distribution on the symmetry
group ${\rm Sym}(d)$. The expectation $\mathbb E$ will be always taken
with respect to the uniform distribution $\mathbb P$ on the symmetry
group ${\rm Sym}(d)$.

\end{enumerate}

\subsection{Asymptotic notation and magnitudes}

\begin{enumerate}[label*={\arabic*}.]

\item  The letters  $C,c, C_0, C_1, \ldots>0$ will always denote
absolute constants which do not depend on the dimension, however their
values may vary from occurrence to occurrence.

\item For two nonnegative quantities
$A, B$ we write $A \lesssim_{\delta} B$ ($A \gtrsim_{\delta} B$) if
there is an absolute constant $C_{\delta}>0$ (which possibly depends
on $\delta>0$) such that $A\le C_{\delta}B$ ($A\ge C_{\delta}B$).  We
will write $A \simeq_{\delta} B$ when $A \lesssim_{\delta} B$ and
$A\gtrsim_{\delta} B$ hold simultaneously. We will omit the subscript
$\delta$ if irrelevant.

\item For two nonnegative quantities $A, B$ we will also use the
convention that $A \lesssim^{d} B$ ($A \gtrsim^{d} B$) to say that there
is an absolute constant $C>0$ such that
$A\le C^d\, B$ ($A\ge C^d \, B$). If $A \lesssim^d B$ and $A\gtrsim^d B$ hold
simultaneously, then we write 
$A \simeq^d B$.

\item For a function $f:X\to \CC$ and positive-valued
function $g:X\to (0, \infty)$, we write $f = O(g)$ if there exists a
constant $C>0$ such that $|f(x)| \le C g(x)$ for all $x\in X$. We will
also write $f = O_{\delta}(g)$ if the implicit constant depends on
$\delta$.  For two functions $f, g:X\to \CC$ such that $g(x)\neq0$ for
all $x\in X$ we write $f = o(g)$ if $\lim_{x\to\infty}f(x)/g(x)=0$.
\end{enumerate}

\subsection{Euclidean spaces}
Denote by ${\bf 1}$ the vector $(1,\ldots,1)\in\RR^d$.

\begin{enumerate}[label*={\arabic*}.]
\item The Euclidean space $\RR^d$ is endowed with the standard inner
product
\[
x\cdot\xi:=\langle x, \xi\rangle:=\sum_{k=1}^dx_k\xi_k
\]
for every $x=(x_1,\ldots, x_d)$ and
$\xi=(\xi_1, \ldots, \xi_d)\in\RR^d$, and the corresponding  Euclidean norm is
denoted by $|x|:=|x|_2:=\sqrt{x\cdot x}$ for any $x\in\RR^d$.

\item We will write $\tau\circ x=(x_{\tau(1)}, \ldots, x_{\tau(d)})$ for every $x\in\RR^d$ and $\tau\in{\rm Sym}(d)$.

\item We will identify the $d$-dimensional torus $\TT^d:=\RR^d/\ZZ^d$
with the unit cube $Q:=[-1/2, 1/2)^d.$

\item For $x\in \RR^d$ let $\vfloor{x}$ be defined as the unique
vector in $\ZZ^d$ such that $x-\vfloor{x}\in [-1/2,1/2)^d$. In
particular note that for $\xi \in Q$ we have $\vfloor{\xi}=0.$

\item For $x\in\RR^d$ we will write
$\|x\|^2:=\|x_1\|^2+\ldots +\|x_d\|^2,$ where $\|x_j\|=\dist(x_j,\ZZ)$
for $j\in\NN_d$. Note that for $\xi\in Q$ the norm $\|\xi\|$ coincides with the Euclidean norm $|\xi|.$

\item We define 
\begin{align}
\label{eq:88}
\begin{gathered}
B^q:=B^q(d):=\Big\{x\in\RR^d\colon |x|_q:=\Big(\sum_{1\le k\le
	d}|x_k|^q\Big)^{1/q}\le 1\Big\} \quad \text{for} \quad q\in[1, \infty),\\
B^{\infty}:=B^{\infty}(d):=\{x\in\RR^d\colon|x|_{\infty}:=\max_{1\le k\le d}|x_k|\le 1\}.     
\end{gathered}
\end{align}

\item Let $S_t:=S_{t}^{d-1}:=\{x\in\RR^d \colon |x|=t\}$ denote the sphere with radius $t>0$ centered at the origin.
If $t=1$ we abbreviate $S_t^{d-1}$ to $S^{d-1}$.

\item The symbol $\sigma^d$ denotes the canonical surface measure on
the unit sphere $S^{d-1}$, and let
\begin{align}
\label{eq:42}
\mu^d:=\frac{1}{\sigma^d(S^{d-1})}\sigma^d
\end{align}
be its normalization.

\end{enumerate}

\subsection{$L^p$ spaces}
 $(X, \mathcal B(X), \nu)$
denotes a measure space $X$ with $\sigma$-algebra $\mathcal B(X)$ and
$\sigma$-finite measure $\nu$.
\begin{enumerate}[label*={\arabic*}.]

\item
  The set of  $\nu$-measurable
complex-valued functions defined on $X$ will be denoted by $L^0(X)$.
\item The set of functions in $L^0(X)$ whose modulus is integrable
with $p$-th power is denoted by $L^p(X)$ for $p\in(0, \infty)$,
whereas $L^{\infty}(X)$ denotes the space of all essentially bounded
functions in $L^0(X)$.

\item Let $B$ be a Banach space and let $L^0(X; B)$ denote the space of all $B$-measurable functions.  In
our applications we can restrict our attention to the case when the
underlying Banach space is finite dimensional. Our estimates will be
independent of the Banach spaces in question and limiting arguments
will encompass the results in the desired generality.

\item  For $p\in(0, \infty]$ we define the $L^p$ spaces of
$B$-valued functions 
\begin{align*}
L^{p}(X;B):
=\{F\in L^0(X;B):\|F\|_{L^{p}(X;B)}=\big\|\|F\|_B\big\|_{L^{p}(X)}<\infty\}.
\end{align*}
\item In our case we will usually have $X=\RR^d$ or
$X=\TT^d$ equipped with the Lebesgue measure, and $X=\ZZ^d$ endowed with the
counting measure. If $X$ is endowed with counting measure we will
abbreviate $L^p(X)$ to $\ell^p(X)$, and $L^{p}(X; B)$ to $\ell^{p}(X; B)$.
\end{enumerate}

\subsection{Fourier transforms}

\begin{enumerate}[label*={\arabic*}.]
\item The Fourier transform of a function $f\in L^1(\RR^d)$ will be denoted by
\begin{align*}
\mathcal Ff(\xi) := \int_{\RR^d} f(x) e^{-2\pi i \sprod{\xi}{x}} dx \quad \text{for any}\quad \xi\in\RR^d.
\end{align*}
More generally, if $\nu$ is a complex Borel measure on $\RR^d$ then the Fourier transform of $\nu$ is defined by
\begin{align*}
\mathcal F\nu(\xi) := \int_{\RR^d} f(x) e^{-2\pi i \sprod{\xi}{x}} d\nu (x) \quad \text{for any}\quad \xi\in\RR^d.
\end{align*}
If $f \in \ell^1\big(\ZZ^d\big)$ we define the discrete Fourier
transform by setting
\begin{align*}
\hat{f}(\xi) := \sum_{x \in \ZZ^d} f(x) e^{-2\pi i \sprod{\xi}{x}} \quad \text{for any}\quad \xi\in\TT^d,
\end{align*}
To simplify notation we will denote by $\mathcal F^{-1}$ the inverse
Fourier transform on $\RR^d$ or the inverse Fourier transform (Fourier
coefficient) on the torus $\TT^d$.  It will cause no confusions since
their meaning will be always clear from the context.

\end{enumerate}

\section{Asymptotic formulae and dimension-free Waring problem}

\label{sec:3'}

In what follows we are only interested in $t>0$ such that
$\la=t^2\in\NN$. Throughout the paper $t$ and $\la$ are always related
by the equation
\begin{equation*}
t=\sqrt{\la}.
\end{equation*}
The following notation will be used
\begin{equation*}
w_{\la}(x):=\ind{S_{\sqrt {\la}}\cap \ZZ^d}(x),\qquad x\in \ZZ^d, \quad \la \in \NN.
\end{equation*}
The Fourier transform of $w_{\la}$ is given by 
\begin{equation}
\label{eq:wlamultdef}
\hat{w}_{\la}(\xi)=\sum_{x\in
S_{\sqrt{\la}}\cap\ZZ^d}e^{2\pi i \xi\cdot x}
\end{equation}
for any $\xi\in\TT^d\equiv[-1/2, 1/2)^d$.  Note that
\begin{align} \label{10.1}
\hat{w}_{\la}(0)=|\{x\in \ZZ^d \colon |x|^2=\la\}|=|S_{\sqrt {\la}}\cap \ZZ^d|.
\end{align}
The spherical averaging operator $\mathcal A^d_{t}$ given by \eqref{eq:27} is a convolution operator with the kernel
\begin{align*}
\mathcal K_{t}(x):=\frac{1}{|S_{t}\cap \ZZ^d|}\ind{S_t\cap \ZZ^d}(x)=\frac{1}{|S_{\sqrt {\la}}\cap \ZZ^d|}w_{\lambda}(x),\qquad x\in \ZZ^d.   
\end{align*}
The  corresponding multipliers are normalized exponential sums given by
\begin{align}
\label{eq:116}
\mathfrak m_t (\xi):=\hat{\mathcal
K}_t(\xi)=\frac{1}{|S_{t}\cap\ZZ^d|}\sum_{x\in
S_{t}\cap\ZZ^d}e^{2\pi i \xi\cdot x}=\frac{1}{|S_{\sqrt {\la}}\cap \ZZ^d|}\hat{w}_{\lambda}(\xi),\qquad \xi\in\TT^d.
\end{align}
Note that the kernel $\mathcal K_{t}$ as well as its multiplier $\mathfrak m_t (\xi)$ are  invariant under permutations, 
namely, for any $\tau \in {\rm Sym}(d)$, one has
\begin{align*}
\mathcal K_{t}(\tau \circ x)=\mathcal K_{t}( x),
\qquad \text{ and } \qquad
\mathfrak m_t (\tau \circ \xi)=\mathfrak m_t (\xi).
\end{align*}
These invariance properties will be crucial  in our further arguments. For further reference we note that
\begin{equation}
\label{eq:mtsym}
\mathfrak m_t (\xi+\vo/2)= (-1)^{\la}\mathfrak m_t (\xi),\qquad \xi \in \TT^d.
\end{equation}
The proof of \eqref{eq:mtsym} is based on the identity
$\sum_{i=1}^d x_i\equiv \sum_{i=1}^d x_i^2\equiv \la \pmod 2$ for
$x\in S_{t}\cap\ZZ^d.$ Consequently, for $x\in S_{t}\cap\ZZ^d$ we have
\begin{align*}
e^{2\pi i \xi\cdot x}=e^{-\pi i |x|^2}e^{2\pi i (\xi+\vo/2)\cdot x}=(-1)^{\la}e^{2\pi i (\xi+\vo/2)\cdot x},
\end{align*}
and \eqref{eq:mtsym} follows.

\subsection{Asymptotic formula for $\hat{w}_{\la}(\xi)$}
We shall employ the circle method of Hardy and Littlewood to derive an
asymptotic formula for $\hat{w}_{\la}(\xi)$. As a consequence we
recover the asymptotic formula in the classical Waring problem for the
squares for all $d\ge 5$, since
$\hat{w}_{\la}(0)=|S_{\sqrt{\la}}\cap \ZZ^d|$.  In fact, in our
asymptotic formula, we obtain certain uniformities with respect to
radii $\lambda\ge C d^3$ (for an absolute constant $C>0$) and
dimensions $d\ge5$. This refinement of the Waring problem, as we shall
see later, is an essential novelty of this paper and will be important
in our further arguments. Before we formulate the main result of this
section we need to set necessary notation and terminology.

Let
\begin{align*}
\delta:=1/\lambda >0.
\end{align*}
Then we write 
\begin{align*}
e^{-2\pi}w_{\la}(x)=e^{-2\pi}\int_0^{1} e^{2\pi i (|x|^2-\lambda)\alpha}\,d\alpha
=\int_0^{1} e^{2\pi i |x|^2(\alpha+i\delta)} e^{-2\pi i \la \alpha}\,d\alpha,
\qquad x \in \mathbb{Z}^d,
\end{align*}
and consequently we obtain
\begin{align*}
\hat{w}_{\lambda}(\xi)=e^{2\pi} \int_0^1 s(\alpha,\delta,\xi)e^{-2\pi i \la \alpha}\,d\alpha,
\end{align*} 
where $s(\alpha,\delta,\xi)$ is an absolutely convergent series given by
\begin{align*}
s(\alpha,\delta,\xi):=\sum_{x \in \ZZ^d}e^{2\pi i x \cdot \xi}e^{2\pi i |x|^2(\alpha+i\delta)}.
\end{align*}

Let
\begin{align} \label{10.2}
N:=\big\lfloor \sqrt{\la}\big\rfloor =\floor{t}
\end{align}
and consider corresponding Farrey sequence
\begin{align*}
H_N:=\left\{\frac{p}{q}\in\QQ\colon 0\le p\le q\le N,\, (p,q)=1\right\}.
\end{align*}
Now we make a Farey dissection at level $N$ of the unit interval
\begin{align*}
[0,1)=\bigcup_{p/q \in H_N} V_N(p/q),
\end{align*}
where the sets $V_N(p/q)$ are pairwise disjoint intervals, and
\begin{align*}
\left\{\alpha\in[0,1)\colon \bigg|\alpha-\frac{p}{q}\bigg|<\frac{1}{2Nq}\right\}
\subseteq V_N\bigg(\frac{p}{q}\bigg)\subseteq
\left\{\alpha\in[0,1)\colon \bigg|\alpha-\frac{p}{q}\bigg|<\frac1{Nq}\right\}.
\end{align*}
Further we define 
\begin{equation}
\label{eq:tVN}
\tilde{V}_N\bigg(\frac{p}{q}\bigg):=V_N\bigg(\frac{p}{q}\bigg)-\frac{p}{q},
\quad \text{ for } \quad
\frac{p}{q}\in H_N.
\end{equation}
Then we see that
\begin{equation*}
\tilde{V}_N\bigg(\frac{p}{q}\bigg)=\left\{\alpha \in \RR\colon \frac{-\beta(N,p/q)}{Nq}\leq \alpha < \frac{\gamma(N,p/q)}{Nq}\right\},
\end{equation*}
where $\beta(N,p/q)$ and $\gamma(N,p/q)$ both belong to the interval $[1/2,1]$; here and later on we identify $0$ with~$1$. 

Using these sets we decompose our multiplier 
\begin{equation}
\label{eq:wcom1}
\hat{w}_{\lambda}(\xi)=e^{2\pi}\sum_{p/q\in H_N}\int_{\tilde{V}_N(p/q)}s(\alpha+p/q,\delta,\xi)e^{-2\pi i \la (\alpha+p/q)}\,d\alpha.
\end{equation}
Note that
\begin{align*}
s(\alpha+p/q,\delta,\xi)=\sum_{x \in \ZZ^d} e^{2\pi i |x|^2p/q+2\pi i x \cdot\xi}\,h_{\alpha,\delta}(x),
\end{align*}
where
\begin{align*}
h_{\alpha,\delta}(x):=e^{2\pi i |x|^2(\alpha+i\delta)},
\qquad x\in\RR^d.
\end{align*}
For further reference note that
\begin{equation}
\label{eq:aux_1}
\mathcal F (h_{\alpha,\delta})(y)=\bigg(\frac{1}{2(\delta-i \alpha)}\bigg)^{d/2}e^{-\frac{\pi}{2(\delta-i\alpha)}|y|^2},\qquad y\in \RR^d.
\end{equation}

Now summing  over the reminders modulo $q$ we may write
\begin{equation}
\label{eq:summodq}
s(\alpha+p/q,\delta,\xi)
=\sum_{n \in \NN_q^d} e^{2\pi i |n|^2p/q}\sum_{m\in \ZZ^d} e^{2\pi i (qm+n)\cdot \xi}\,h_{\alpha,\delta}(qm+n).
\end{equation}
By the Poisson summation formula applied to the inner sum in \eqref{eq:summodq} we obtain
\begin{align*}
\sum_{m\in \ZZ^d} e^{2\pi i (qm+n) \xi}\,h_{\alpha,\delta}(qm+n)=
\sum_{x \in \ZZ^d} e^{2\pi ix \cdot n/q }q^{-d}\mathcal F (h_{\alpha,\delta})(x/q-\xi).
\end{align*}
Therefore, coming back to \eqref{eq:wcom1} we conclude
$$\hat{w}_{\lambda}(\xi)=e^{2\pi}\sum_{p/q\in H_N}\int_{\tilde{V}_N(p/q)}
\sum_{x \in \ZZ^d}q^{-d}\sum_{n \in \NN_q^d} e^{2\pi i (|n|^2p/q+x\cdot n/q )}
\mathcal F (h_{\alpha,\delta})(x/q-\xi)\,e^{-2\pi i \la (\alpha+p/q)}\,d\alpha.$$
For $(p,q)=1$ and $x\in \ZZ^d$ let $G(p/q;x)$ be the $d$-dimensional Gaussian sum 
\begin{equation}
\label{eq:def:gsum}
G(p/q;x):=q^{-d}\sum_{n \in \NN_q^d} e^{2\pi i (|n|^2p/q+x\cdot n/q)}.
\end{equation}
Using \eqref{eq:def:gsum} we now write
\begin{equation}
\label{eq:wcom2}
\hat{w}_{\lambda}(\xi)=
e^{2\pi}\sum_{p/q\in H_N}e^{-2\pi i \la p/q}\int_{\tilde{V}_N(p/q)}
\sum_{x \in \ZZ^d}G(p/q;x)
\mathcal F (h_{\alpha,\delta})(x/q-\xi)\,e^{-2\pi i \la \alpha}\,d\alpha.
\end{equation}  
If $\xi=0$ then \eqref{eq:wcom2} yields the formula for the number of
lattice points in $S_t$, see \eqref{10.1}.

Fix a function $\varphi $ in $C_c^{\infty}((-1/4,1/4))$ which is equal
to $1$ on $[-1/8,1/8],$ and satisfies
$\|\varphi\|_{L^{\infty}(\RR)}\le 1.$ Then denote
\begin{equation}
\label{eq:psidef}
\psi(\xi):=\prod_{j=1}^d\varphi(\xi_j),\qquad \xi\in \RR^d.
\end{equation} 
We have almost prepared the ground to formulate the main result of this section. Let us define
\begin{align}
\label{eq:3}
M_1(\lambda, \xi, n):=&\frac12 \la^{d/2-1}\sum_{\substack{p/q \in H_N\\q<n}} e^{-2\pi i \la p/q}G(p/q;\vfloor{q\xi})\mathcal F \sigma\big(\sqrt{\la}\big(\vfloor{q\xi}/q-\xi\big)\big),\\
\label{eq:4}
M_2(\lambda, \xi, n):=&\frac12 \la^{d/2-1}\sum_{\substack{p/q \in H_N\\q\ge n}}\sum_{x \in \ZZ^d}\ e^{-2\pi i \la p/q}G(p/q;x)\psi(q\xi -x)\mathcal F \sigma\big(\sqrt{\la}(x/q-\xi)\big).
\end{align}
If $n>N=\lfloor \sqrt{\lambda}\rfloor$, then
$M_2(\lambda, \xi, n)=0$ and analogously $M_1(\lambda, \xi, 1)=0$. We now state our main result of this
section, which can be thought of as a dimension-free variant of the solution to the Waring
problem. This phenomenon is exhibited by the asymptotic formula stated
in \eqref{eq:1}, where the multiplicative error term has been obtained in the classical Waring problem for the squares 
as long as  $\la \ge C d^3$ and $d\to\infty$.

\begin{theorem}
\label{thm:asymptotic}
There exists an absolute constant $C>0$ such that for all integers
$d\ge 5$ and $\lambda>0$ satisfying $\la \ge C d^3$, and for all
$n\in\NN_{N+1}$ (with $N=\lfloor \sqrt{\lambda}\rfloor$) and $\xi \in \TT^d$ we have
\begin{align}
\label{eq:wlaxidec'}
\big|\hat{w}_{\la}(\xi)-M_1(\lambda, \xi, n)-M_2(\lambda, \xi, n)\big|\lesssim^d (d\la)^{d/4}.
\end{align}
If $n=N+1$, then \eqref{eq:wlaxidec'} gives the following estimate
\begin{align}
\label{eq:wlaxidec}
\big|\hat{w}_{\la}(\xi)-M_1(\lambda, \xi, N+1)\big|\lesssim^d (d\la)^{d/4}.
\end{align}
Taking $\xi=0$ in \eqref{eq:wlaxidec}, for $d \ge 16$ one has
\begin{align}
\label{eq:1}
|S_{\sqrt{\lambda}}\cap \ZZ^d|= \frac{\pi^{d/2}}{\Gamma(d/2)}\lambda^{d/2-1}\mathfrak{S}_d(\lambda)\big(1+o(1)\big)
\quad \text{ as } \quad \la \ge C d^3 \text{ and } \,  d\to \infty,
\end{align}
where the singular series $\mathfrak{S}_d(\lambda)$ is given by
\begin{align}
\label{eq:40}
\mathfrak{S}_d(\lambda):=\sum_{q=1}^{\infty}\sum_{\substack{1\leq p\le  q \\ (p,q)=1}}e^{-2\pi i \la p/q}G(p/q;0)
\end{align}
and satisfies the estimate $\frac{1}{2}\le \mathfrak{S}_d(\lambda)\le \frac{3}{2}$ independently of the dimension $d\ge 16$ and $\lambda\in\NN$. 
Finally, the asymptotic formula from \eqref{eq:1} also ensures that
\begin{align}
\label{wla0}
|S_{\sqrt{\lambda}}\cap \ZZ^d|\simeq \frac{\pi^{d/2}}{\Gamma(d/2)}\lambda^{d/2-1}\simeq \sigma(S^{d-1})\lambda^{d/2-1},
\qquad \la \ge C d^3, \quad d \ge 16.
\end{align}

\end{theorem}

Before we turn to the proof of Theorem \ref{thm:asymptotic} we need
several auxiliary lemmas.  
\begin{lemma}
\label{lem:aux_1}
For any integers $p, q\in\ZZ$ such that $0\le p\le q$ and  $(p,q)=1$ one has
\begin{align}
\label{eq:2}
|G(p/q;x)|\le (2/q)^{d/2},\qquad x\in \ZZ^d.
\end{align}
Moreover,
\begin{align}
\label{eq:aux_6_1}
\sum_{n\in \NN_q^d}|G(p/q;n)|^2=1.
\end{align}
\end{lemma} 
\begin{proof}

The proof of Lemma \ref{lem:aux_1} is based on simple calculations.
Thanks to the product structure of the Gauss sums we may assume, without loss of generality, that $d=1$.
\end{proof}

\begin{lemma}
\label{lem:aux_2}
For each $T_0>0 $ there is a constant $C(T_0)>0$ such that
\begin{align*}
\sum_{x\in \ZZ^d\setminus\{0\}}e^{-T |x+y|^2}  \leq C(T_0)^d\,e^{-T/4},
\end{align*}
uniformly in $T\ge T_0$ and  $y\in Q$.  
\end{lemma} 
\begin{proof}
Observe that $|x+y|^2\ge |x|^2/4$ for $x\in \ZZ^d$ and
$y\in Q$, which is a consequence of the one-dimensional situation. 
Thus, for $y\in Q,$ we have
\begin{align*}
\sum_{x\in \ZZ^d\setminus\{0\}}e^{-T |x+y|^2}
&\leq e^{-T/4} \sum_{x\in \ZZ^d\setminus\{0\}}\exp\big(-T( |x|^2/4-1/4)\big)\\
&\leq e^{-T/4}\sum_{x \in \ZZ^d}\exp\big(-T_0( |x|^2/4-1/4)\big)\\
&\leq C(T_0)^d\,e^{-T/4},
\end{align*} 
where $C(T_0):=e^{T_0/4}\sum_{x \in \ZZ}e^{-T_0|x|^2/4}$. This completes the proof of Lemma \ref{lem:aux_2}.
\end{proof}

\begin{lemma}
\label{lem:aux_3}
Let $d\ge 5$ and define
\begin{align*}
I_{\la}(\xi):=e^{2\pi}
\int_{\RR}e^{-2\pi i \alpha} \left(\frac{1}{2(1-i\alpha)}\right)^{d/2}\exp\left(-\frac{\pi \la|\xi|^2}{2(1-i\alpha)}\right)\,d\alpha,\qquad \la>0, \quad \xi\in\RR^d.
\end{align*}
Then one has
\begin{align*}
I_{\la}(\xi)=\frac12 \mathcal F (\sigma^d)(\sqrt{\la}\xi),\qquad \la >0,\quad \xi \in \RR^d.
\end{align*}
\end{lemma} 
\begin{proof}
The proof of Lemma \ref{lem:aux_3} can be found in \cite[Lemma
6.1]{MSW}.
\end{proof}

We now prove Theorem \ref{thm:asymptotic}.

\begin{proof}[Proof of Theorem \ref{thm:asymptotic}]
Appealing to \eqref{eq:aux_1}, changing the variable $\alpha=\delta \beta$ and using the fact that $\delta=1/\lambda$ we see that
\begin{equation}
\label{eq:aux_1_cons}
\begin{split}
e^{2\pi}\int_{\RR}\mathcal F (h_{\alpha,\delta})(y)\,e^{-2\pi i \la \alpha}\,d\alpha
&=\delta^{-d/2+1}e^{2\pi}\int_{\RR}\bigg(\frac{1}{2(1-i \beta)}\bigg)^{d/2}e^{-\frac{\pi \la}{2(1-i\beta)}|y|^2}e^{-2\pi i  \beta}\,d\beta\\
&= \delta^{-d/2+1}I_{\la}(y)\\
&=\frac12 \la^{d/2-1}\mathcal F \sigma\big(\sqrt{\la}y\big), \qquad y\in\RR^d, \quad \lambda > 0,
\end{split}
\end{equation}
where in the last equality we have used Lemma \ref{lem:aux_3}.

Recalling the expression \eqref{eq:wcom2} for $\hat w_{\la}(\xi)$ and using 
\eqref{eq:aux_1_cons} with $y=\vfloor{q\xi}/q-\xi$ and $y=x/q-\xi$ we may write
\begin{align}
\label{eq:wlaxidec'_decom}
\hat{w}_{\la}(\xi)=(M_1-E_1+E_2)+(M_2-E_3+E_4),
\end{align}
where $M_1:=M_1(\lambda, \xi, n)$ and $M_2:=M_2(\lambda, \xi, n)$ are defined respectively  in \eqref{eq:3} and \eqref{eq:4} and $E_j:=E_j(\lambda, \xi, n)$, $1 \le j \le 4$, are defined by setting 
\begin{align*}
E_1:=&e^{2\pi}\sum_{\substack{p/q\in H_N\\q<n}}e^{-2\pi i \la p/q}\int_{(\tilde{V}_N(p/q))^c}G(p/q;\vfloor{q\xi})
\mathcal F (h_{\alpha,\delta})(\vfloor{q\xi}/q-\xi)\,e^{-2\pi i \la \alpha}\,d\alpha,\\
E_2:=&e^{2\pi}\sum_{\substack{p/q\in H_N\\q<n}}e^{-2\pi i \la p/q}\int_{\tilde{V}_N(p/q)}\sum_{\substack{x \in \ZZ^d\\x\neq \vfloor{q\xi}}}G(p/q;x)
\mathcal F (h_{\alpha,\delta})(x/q-\xi)\,e^{-2\pi i \la \alpha}\,d\alpha,
\end{align*}
and
\begin{align*}
E_3:=&e^{2\pi}\sum_{\substack{p/q\in H_N\\q \ge n}}e^{-2\pi i \la p/q}\int_{(\tilde{V}_N(p/q))^c}
\sum_{x \in \ZZ^d}G(p/q;x)\psi(q\xi -x)\mathcal F (h_{\alpha,\delta})(x/q-\xi)\,e^{-2\pi i \la \alpha}\,d\alpha,\\
E_4:=&e^{2\pi}\sum_{\substack{p/q\in H_N\\q\ge n}}e^{-2\pi i \la p/q}\int_{\tilde{V}_N(p/q)}
\sum_{x \in \ZZ^d} G(p/q;x)(1-\psi(q\xi -x))\mathcal F (h_{\alpha,\delta})(x/q-\xi)\,e^{-2\pi i \la \alpha}\,d\alpha,
\end{align*}	
with  the sets $\tilde{V}_N(p/q)$ defined in \eqref{eq:tVN}.

Now, in view of \eqref{eq:wlaxidec'_decom}, the proof of Theorem \ref{thm:asymptotic} will be completed 
once we prove that
\begin{equation}
\label{eq:wlaxidec'_Ej}
|E_j|\lesssim^d (d\la )^{d/4},\qquad \text{ for } \quad j=1,2,3,4,
\end{equation}
and establish the asymptotic formulae from \eqref{eq:1} and \eqref{wla0}. The details will be split into five steps.

\paragraph{\bf Step 1}
We now prove \eqref{eq:wlaxidec'_Ej} for $j=1$. By \eqref{eq:2} and
the fact that $(\tilde{V}_N(p/q))^c\subseteq [-1/(2Nq),1/(2Nq))^c$ we
obtain
\begin{align*}
|E_1| \lesssim^d \sum_{1\le q\le N} q^{1-d/2} \int_{ |\alpha|\ge (2Nq)^{-1}}|\mathcal F (h_{\alpha,\delta})(\vfloor{q\xi}/q-\xi)|\,d\alpha.
\end{align*}
By \eqref{eq:aux_1} we have
$|\mathcal F (h_{\alpha,\delta})(\vfloor{q\xi}/q-\xi)|\le 2^{-d/2}(\delta^2+\alpha^2)^{-d/4}$,
and since $N=\lfloor \sqrt{\lambda}\rfloor$, we
conclude
\begin{align*}
|E_1| \lesssim^d \sum_{1\le q\le N} q^{1-d/2} \int_{ |\alpha|\ge (2Nq)^{-1}}|\alpha|^{-d/2}\,d\alpha
\lesssim^d \sum_{1\le q\le N} q^{1-d/2} (Nq)^{d/2-1}\le \la^{d/4}.
\end{align*}

\paragraph{\bf Step 2}
We now prove \eqref{eq:wlaxidec'_Ej} for $j=2$.  Using \eqref{eq:2}
and \eqref{eq:aux_1} we obtain
\begin{align*}
|E_2|&\lesssim^d \sum_{1\le q \le N} q \int_{|\alpha|\le 1/(Nq)}\sum_{\substack{x \in \ZZ^d\\ x\neq \vfloor{q\xi}}}
(q^2(\delta^2+\alpha^2))^{-d/4}\,\exp\left(\frac{-\pi \delta}{2(\delta^2+\alpha^2)}\,|x/q-\xi|^2\right)\,d\alpha\\
&=
\sum_{1\le q \le N} q \int_{|\alpha|\le 1/(Nq)}(q^2(\delta^2+\alpha^2))^{-d/4}\,\sum_{x \in \ZZ^d\setminus\{0\}}
\exp\left(\frac{-\pi \delta}{2(\delta^2+\alpha^2)q^2}\, |x+\vfloor{q\xi}-q\xi|^2\right)\,d\alpha.
\end{align*}
Since $N^2 \simeq \la $ and  $\alpha^2 q^2 N^2\le 1$ we see that
\begin{equation}
\label{eq:aux_4}
\frac{\pi \delta}{2(\delta^2+\alpha^2)q^2}=\frac{\pi}{2/\la +2\la \alpha^2 q^2}\ge T_0,
\end{equation}
where $T_0>0$ is a universal constant. Therefore, by Lemma \ref{lem:aux_2} with
$y=\vfloor{q\xi}-q\xi\in Q$ we have
\begin{align*}
|E_2|\lesssim^d 	\sum_{1\le q \le N} q \int_{|\alpha|\le 1/(Nq)}(q^2(\delta^2+\alpha^2))^{-d/4}\,
\exp\left(\frac{-\pi \delta}{8(\delta^2+\alpha^2)q^2}\right)\,d\alpha.
\end{align*}
Using the inequality $e^{-x}\le (d/4)^{d/4} e^{-d/4}x^{-d/4}$ with
$x=\frac{\pi \delta}{8(\delta^2+\alpha^2)q^2}$, and recalling $\delta=1/\lambda$, we are led to
\begin{align*}
|E_2|\lesssim^d d^{d/4}\,\sum_{1\le q \le N} q \int_{|\alpha|\le 1/(Nq)} \delta^{-d/4}\, d\alpha
\lesssim^d ( d \la)^{d/4}.
\end{align*}

\paragraph{\bf Step 3}
We now prove \eqref{eq:wlaxidec'_Ej} for $j=3$. By definition of
$\psi$ the term $\psi(q\xi -x)$ is non-zero for at most one
$x\in \ZZ^d;$ moreover, $\|\psi\|_{L^{\infty}(\RR^d)}\leq 1$. Hence,
\eqref{eq:2}  and
\eqref{eq:aux_1} imply
\begin{align*}
|E_3&|\lesssim ^d  \sum_{1\le q \le N} q 
\int_{|\alpha|\ge 1/(2Nq)}\sup_{x \in \ZZ^d}\left(q^{-d/2}(\delta^2+\alpha^2)^{-d/4}\,
\exp\left(\frac{-\pi \delta}{2(\delta^2+\alpha^2)}\, |x/q-\xi|^2\right)\right)\,d\alpha\\
&\lesssim^d \sum_{1\le q\le N} q^{1-d/2}\int_{|\alpha|\ge 1/(2Nq)}|\alpha|^{-d/2}\,d\alpha
\lesssim^d \sum_{1\le q\le N} q^{1-d/2}(Nq)^{d/2-1}
\lesssim^d \la^{d/4}.
\end{align*}
\paragraph{\bf Step 4}
It remains to verify \eqref{eq:wlaxidec'_Ej} for $j=4.$ Note that, by
definition of $\psi$ if $1-\psi(q\xi -x)\neq 0,$ then
$|x-q\xi|_{\infty}>1/8.$ Hence, \eqref{eq:2} and \eqref{eq:aux_1} show
that
\begin{align*}
|E_4|&\lesssim ^d  \sum_{1\le q \le N} q  \int_{|\alpha|\le 1/(Nq)}\sum_{x \in \ZZ^d}(q^2(\delta^2+\alpha^2))^{-d/4}\,\ind{\{|x-q\xi|_{\infty}>1/8\}}\,
g(x-q\xi)
\,d\alpha\\
&=\sum_{1\le q \le N} q  \int_{|\alpha|\le 1/(Nq)}\sum_{x \in \ZZ^d}(q^2(\delta^2+\alpha^2))^{-d/4}\,\ind{\{|x+\vfloor{q\xi}-q\xi|_{\infty}>1/8\}}\,
g(x+\vfloor{q\xi}-q\xi)\,d\alpha,
\end{align*}
where
\begin{align*}
g(y):=\exp\left(\frac{-\pi \delta}{2q^2(\delta^2+\alpha^2)}\, |y|^2\right), \qquad y\in\RR^d.
\end{align*}
Since $N^2\simeq \la$ and  $\alpha^2 q^2 N^2\le 1$,
by \eqref{eq:aux_4} we may apply Lemma \ref{lem:aux_2} with
$y=\vfloor{q\xi}-q\xi\in Q$ to conclude
\begin{align*}
|E_4|&\lesssim ^d  \sum_{1\le q \le N} q  \int_{|\alpha|\le 1/(Nq)}(q^2(\delta^2+\alpha^2))^{-d/4}\\
&\qquad \qquad \times \,\left(\ind{\{|\vfloor{q\xi}-q\xi|_{\infty}>1/8\}}\exp\left(\frac{-\pi \delta}{2q^2(\delta^2+\alpha^2)}\, |\vfloor{q\xi}-q\xi|^2\right)+\exp\left(\frac{-\pi \delta}{8q^2(\delta^2+\alpha^2)}\right)\right)\,d\alpha\\
&\lesssim^d  \sum_{1\le q \le N} q  \int_{|\alpha|\le 1/(Nq)}(q^2(\delta^2+\alpha^2))^{-d/4}\exp\left(\frac{-\pi \delta}{128q^2(\delta^2+\alpha^2)}\right)\,d\alpha.
\end{align*}
Now, using the inequality $e^{-x}\le (d/4)^{d/4} e^{-d/4}x^{-d/4}$ with $x=\frac{\pi \delta}{128(\delta^2+\alpha^2)q^2} $  we are thus led to
\begin{align*}
|E_4|\lesssim^d d^{d/4}\, \sum_{1\le q \le N} q \int_{|\alpha|\le 1/(Nq)} \delta^{-d/4}\, d\alpha \lesssim^d (d\la)^{d/4}.
\end{align*}
This also completes the proof of inequality \eqref{eq:wlaxidec'} as well as \eqref{eq:wlaxidec}.

\paragraph{\bf Step 5} We now establish asymptotic formula in \eqref{eq:1}. For this purpose we observe that \eqref{eq:wlaxidec} with $\xi=0$ and \eqref{10.1} yields
\begin{align}
\label{eq:5}
\big||S_{\sqrt{\lambda}}\cap \ZZ^d|-M_1(\lambda, 0, N+1)\big|\lesssim^d (d\la)^{d/4}.
\end{align}
Combining $\mathcal F\sigma(0)=\sigma(S^{d-1})= \frac{2\pi^{d/2}}{\Gamma(d/2)}$ with \eqref{eq:3} we obtain
\begin{align*}
M_1(\lambda, 0, N+1)=\frac{\pi^{d/2}}{\Gamma(d/2)}\la^{d/2-1}\mathfrak S_d(\lambda; N),
\end{align*}
where 
\begin{align*}
\mathfrak S_d(\lambda; P):=\sum_{q=1}^{P}\sum_{\substack{1\leq p\le  q \\ (p,q)=1}}e^{-2\pi i \la p/q}G(p/q;0), \qquad P\in\NN.
\end{align*}
For $d\ge 16$ it is also not difficult to see that 
\begin{align}
\label{eq:10}
\frac{1}{2}\le \mathfrak S_d(\lambda)\le \frac{3}{2},
\end{align}
since
\begin{align*}
\mathfrak S_d(\lambda)=1+\sum_{q=3}^{\infty}\sum_{\substack{1\leq p\le  q \\ (p,q)=1}}e^{-2\pi i \la p/q}G(p/q;0),
\end{align*}
and by \eqref{eq:2} for $d \ge 16$ we have
\begin{align*}
\Big|\sum_{q=3}^{\infty}\sum_{\substack{1\leq p\le  q \\ (p,q)=1}}e^{-2\pi i \la p/q}G(p/q;0)\Big|
&\le  
\sum_{q=3}^{\infty} q^{1-d/2} 2^{d/2}
\le 
2^{d/2} 3^{1-d/2} + 2^{d/2} \int_3^{\infty} y^{-d/2 +1}\,dy \\
& \le 
3 (2/3)^{d/2} + 9 (2/3)^{d/2} \frac{1}{d/2 - 2}
\le
1/2.
\end{align*}
In a similar way, recalling that $N\simeq \lambda^{1/2}$ and using \eqref{eq:2}, we obtain
\begin{align}
\label{eq:6}
|\mathfrak S_d(\lambda)-\mathfrak S_d(\lambda; N)|
\le 2^{d/2} \sum_{q=N+1}^{\infty}q^{-d/2+1}\lesssim^d\lambda^{-d/4+1}.
\end{align}
Combining \eqref{eq:5} and \eqref{eq:6} we conclude
\begin{align}
\label{eq:7}
\Big||S_{\sqrt{\lambda}}\cap \ZZ^d|-\frac{\pi^{d/2}}{\Gamma(d/2)}\la^{d/2-1}\mathfrak S_d(\lambda)\Big|
\lesssim^d (d\la)^{d/4}
+ \frac{1}{\Gamma(d/2)}\la^{d/4} 
\simeq ^d (d\la)^{d/4}.
\end{align}
Moreover, since  $\Gamma(d/2)\simeq^d d^{d/2}$ we see that
\begin{align}
\label{eq:8}
(d\la)^{d/4} \lesssim^d \la^{-d/4+1}d^{3d/4}\cdot \frac{\pi^{d/2}}{\Gamma(d/2)}\la^{d/2-1}=d^3\bigg(\frac{d^3}{\la}\bigg)^{d/4-1}\cdot \frac{\pi^{d/2}}{\Gamma(d/2)}\la^{d/2-1}.
\end{align}
Therefore, \eqref{eq:10}, \eqref{eq:7} and \eqref{eq:8} ensure, for some universal constant $C_0>0$, that
\begin{align}
\label{eq:9}
|S_{\sqrt{\lambda}}\cap \ZZ^d|=\frac{\pi^{d/2}}{\Gamma(d/2)}\la^{d/2-1}\mathfrak S_d(\lambda)
\Big(1+O\Big(C_0^d\big(d^3\la^{-1}\big)^{d/4-1}\Big)\Big),
\end{align}
which for a large absolute constant $C>0$ and for all integers
$d\ge 16$ and $\lambda>0$ obeying $\lambda\ge Cd^3$ implies
\eqref{eq:1}. Asymptotic \eqref{eq:9} and \eqref{eq:10} also imply
\eqref{wla0} and the proof of Theorem \ref{thm:asymptotic} is
completed.
\end{proof}

\subsection{Lattice points for balls and spheres}

We shall also need a comparison between numbers of lattice points in balls and spheres.

\begin{lemma}
\label{lem:lfs}
Let $d\ge 5$. Then for all $t>0$ such that $\la=t^2\in\NN$ we have
\begin{equation}
\label{eq:lfs1}
|B_t^2(d-4)\cap \ZZ^{d-4}|\le |S_t^{d-1}\cap \ZZ^d|, 
\end{equation}
Consequently, for such $t$ we have
\begin{equation}
\label{eq:lfs2}
|S_t^{d-1}\cap \ZZ^d|\le |B_t^2(d)\cap \ZZ^d|\le (2t+1)^{4} |S_t^{d-1}\cap \ZZ^d|.
\end{equation} 
\end{lemma}
\begin{proof}	
By the Lagrange four squares theorem for each $0\le j\le \la$ we find integers
$y_{j, d-3},y_{j, d-2}, y_{j, d-1}, y_{j, d}$ such that 
\[
\sum_{i=d-3}^{d}y_{j, i}^2=\la-j.
\]
Then  for each  $x=(x_1,\ldots,x_{d-4})\in B_t^2(d-4)\cap \ZZ^{d-4}$ we define
$\Phi \colon B_t^2(d-4)\cap \ZZ^{d-4} \to S_{t}^{d-1}\cap \ZZ^d $ by
setting
\[
\Phi(x):=(x_1,\ldots,x_{d-4},y_{j, d-3},y_{j, d-2},y_{j, d-1},y_{j, d}),\qquad \text{if}\qquad (x_1,\ldots,x_{d-4})\in S_{\sqrt{j}}^{d-5},
\]
with $0\le j\le \la.$ Since $B_t^2(d-4)\cap \ZZ^{d-4}$ decomposes as
the disjoint union
$\bigcup_{j=0}^{\la}S_{\sqrt{j}}^{d-5}\cap \ZZ^{d-4}$ the function
$\Phi$ is an injection from $B_t^2(d-4)\cap \ZZ^{d-4}$ to
$S_{t}^{d-1}\cap \ZZ^d$. Thus, we have proved \eqref{eq:lfs1}.
It remains to justify   the second
inequality in \eqref{eq:lfs2}. Observe
that
\begin{align*}
B_t^2(d)\cap \ZZ^d \subseteq (B_t^2(d-4)\cap \ZZ^{d-4})\times ([-t,t]^4\cap \ZZ^4)
\end{align*}
and consequently
$|B_t^2(d)\cap \ZZ^d|\le (2t+1)^{4}| B_t^2(d-4)\cap \ZZ^{d-4}|$, which by  \eqref{eq:lfs1} gives \eqref{eq:lfs2}.
\end{proof}

\section{General estimates for certain Fourier multipliers}
\label{sec:3}
In this section we gather general estimates for Fourier multipliers in
the continuous and discrete setup, which will be used later on in the
paper. We first provide dimension-free estimates of the Fourier
transform corresponding to the spherical measure in $\RR^d$.

\subsection{Fourier transform estimates for the continuous spherical measures}
For $r\ge2$, let $\mu^r$
denote the normalized spherical surface measure on $S^{r-1}$ as in \eqref{eq:42}. We shall abbreviate $\mu^r$ to $\mu$ if $r=d$.

Dimension-free estimates of the Fourier transforms corresponding to
the spherical measures are provided in Lemma \ref{lem:aux_4} and Lemma
\ref{lem:aux_5}.  These two results may be of independent interest.

\begin{lemma}
\label{lem:aux_4}
There exists a constant $c>0$ such that for all $r\ge 2$ and
$\eta\in \RR^r$ we have
\begin{align*}
|\mathcal F\mu^r(\eta)|\lesssim e^{-2\pi |\eta|/\sqrt{r}}+e^{-cr},
\end{align*}
where the implicit constant is independent of $r$ and $\eta$.
\end{lemma}

\begin{proof}
The  lemma is obvious when the dimension $r$ is
small. Thus, from now on we only focus on sufficiently large
$r\in\NN$. Throughout the proof we abbreviate $\mu^r$ to $\mu$, 
$\sigma^{r}$ to $\sigma$, and $S^{r-1}$ to $S$.  It is well known, for instance  by \cite[Appendix B.4]{Gra_C}, that 
\begin{equation}
\label{eq:FsiBes}
\mathcal F \sigma(\eta)=\frac{2\pi}{|\eta|^{r/2-1}}J_{r/2-1}(2\pi |\eta|),\qquad \eta\in \RR^r,\end{equation} 
where $J_{\nu}$ is the Bessel function of the first kind, which can be written, for $\Re \nu > -1/2$ and $u>0$ as
\begin{equation}
\label{eq:Bes}
J_{\nu}(u):=\frac{u^{\nu}}{2^{\nu}\Gamma(\nu+1/2)\sqrt{\pi}}\int_{-1}^{1} e^{isu}(1-s^2)^{\nu-1/2}\,ds.\end{equation}
Therefore, using \eqref{eq:42} together with \eqref{eq:FsiBes} and \eqref{eq:Bes} we have
\begin{align*}
\mathcal F \mu(\eta)=\frac{2\pi^{\frac{r-1}{2}}}{\sigma(S)\Gamma(\frac{r-1}{2})}\int_{-1}^1e^{-2\pi i |\eta|s}(1-s^2)^{\frac{r-3}{2}}\,ds=\frac{\Gamma(r/2)}{\pi^{1/2}\Gamma(\frac{r-1}{2})}\int_{-1}^1e^{-2\pi i |\eta|s}(1-s^2)^{\frac{r-3}{2}}\,ds.
\end{align*}
By  Stirling's formula we deduce that
\begin{equation}
\label{eq:aux_4_0}
|\mathcal F \mu(\eta)|\simeq \left|\sqrt{r} \int_{-1}^1e^{-2\pi i |\eta|s}(1-s^2)^{\frac{r-3}{2}}\,ds\right|=\left| \int_{-\sqrt{r}}^{\sqrt{r}}e^{-2\pi i |\eta|s/\sqrt{r}}(1-s^2/r)^{\frac{r-3}{2}}\,ds\right|.\end{equation}
Setting
\begin{align*}
M(u):= \int_{-\sqrt{r}}^{\sqrt{r}}e^{-2\pi i us}(1-s^2/r)^{\frac{r-3}{2}}\, ds, \qquad u \in \mathbb{R},
\end{align*}
and noting that $M(u) = M(|u|)$, $u \in \mathbb{R}$, we see that
the proof of Lemma \ref{lem:aux_4} reduces to showing that for sufficiently large $r \in \NN$ we have
\begin{equation}
\label{eq:aux_4_1}
|M(u)|\lesssim e^{-2\pi u}+e^{-cr}, \qquad u \ge 0.
\end{equation}
We now establish  \eqref{eq:aux_4_1}. We first note
\begin{align*}
|M(u)|
& \le
\bigg|\int_{r^{1/2}/100\le |s|\le r^{1/2}}e^{-2\pi is u }\bigg(1-\frac{s^2}{r}\bigg)^{\frac{r-3}{2}}\dif s\bigg|
\\
& \quad +
\bigg|\int_{|s|\le r^{1/2}/100}e^{-2\pi i s u }
\bigg(1-\frac{s^2}{r}\bigg)^{\frac{r-3}{2}}\dif s\bigg|.
\end{align*}
Observe that 
\begin{align*}
\bigg|\int_{r^{1/2}/100\le |s|\le r^{1/2}}e^{-2\pi
	is u }\bigg(1-\frac{s^2}{r}\bigg)^{\frac{r-3}{2}}\dif s\bigg|
\le 
2 r^{1/2}(1-10^{-4})^{\frac{r-3}{2}}\lesssim e^{-cr},
\end{align*}
since $1-\frac{s^2}{r}\le 1-\frac{1}{10^4}$ for $|s| \ge r^{1/2}/100$.
The proof of \eqref{eq:aux_4_1} will be completed if we show that
\begin{align}
\label{eq:aux_4_3}
\bigg|\int_{|s|\le r^{1/2}/100}e^{-2\pi i s u }
\bigg(1-\frac{s^2}{r}\bigg)^{\frac{r-3}{2}}\dif s\bigg|\lesssim  e^{-2\pi u }+e^{-cr}.
\end{align}
To prove \eqref{eq:aux_4_3} we will change the contour of
integration. Namely, let
$\gamma=:\gamma_{0}\cup\gamma_{1}\cup\gamma_2\cup\gamma_3$ be the
rectangle with the parametrization
\begin{align*}
\gamma_0(s)&:=s&&\text{for} \quad s\in[-r^{1/2}/100, r^{1/2}/100],\\
\gamma_1(s)&:=-is+\frac{r^{1/2}}{100}           &&\text{for} \quad
s\in[0,1],\\
\gamma_2(s)&:=-s-i&&\text{for} \quad s\in[-r^{1/2}/100,
r^{1/2}/100],\\
\gamma_3(s)&:=-i(1-s)-\frac{r^{1/2}}{100}           &&\text{for} \quad
s\in[0,1].
\end{align*}
The function $ z\mapsto e^{-2\pi i z u }
\big(1-\frac{z^2}{r}\big)^{\frac{r-3}{2}}$ is holomorphic in $\{z \in \mathbb{C} : |z| < \sqrt{r}/10 \}$ so the
Cauchy integral theorem implies 
\begin{align}
\label{eq:aux_4_8}
\begin{split}
\bigg|\int_{|s|\le r^{1/2}/100}e^{-2\pi i s u }
\bigg(1-\frac{s^2}{r}\bigg)^{\frac{r-3}{2}}\dif s\bigg|
&\le\sum_{j\in\{1, 3\}}\bigg|\int_{0}^1e^{-2\pi i \gamma_{j}(s) u }
\bigg(1-\frac{\gamma_j(s)^2}{r}\bigg)^{\frac{r-3}{2}}\gamma_j'(s)\dif
s\bigg|\\
& \quad 
+\bigg|\int_{|s|\le r^{1/2}/100}e^{2\pi i (s+i) u }
\bigg(1-\frac{(s+i)^2}{r}\bigg)^{\frac{r-3}{2}}\dif s\bigg|.
\end{split}
\end{align}
Observe now that
\begin{align}
\label{eq:aux_4_5}
\sum_{j\in\{1, 3\}}\bigg|\int_{0}^1e^{-2\pi i \gamma_{j}(s) u }
\bigg(1-\frac{\gamma_j(s)^2}{r}\bigg)^{\frac{r-3}{2}}\gamma_j'(s)\dif
t\bigg|\le
\sum_{j\in\{1, 3\}}\int_{0}^1
\bigg|1-\frac{\gamma_j(s)^2}{r}\bigg|^{\frac{r-3}{2}}\dif
s\lesssim e^{-cr},
\end{align}
since for $s\in[0,1]$ and sufficiently large $r\in\NN$ we have
\[
\sum_{j\in\{1,
	3\}}\bigg|1-\frac{\gamma_j(s)^2}{r}\bigg|^{\frac{r-3}{2}}\leq 2
\bigg(1-\frac{1}{10^4}+\frac{1}{r}+\frac{1}{50r^{1/2}}\bigg)^{\frac{r-3}{2}}
\le
2 \bigg(1-\frac{1}{10^5}\bigg)^{\frac{r-3}{2}}\lesssim e^{-cr}.
\]
We also have $e^{2\pi i (s+i) u }=e^{-2\pi  u }e^{2\pi i s u }$. 
Thus  it suffices to prove that for sufficiently large $r\in\NN$ we have
\begin{align}
\label{eq:aux_4_4}
\bigg|\int_{|s|\le r^{1/2}/100}e^{2\pi i s u }
\bigg(1-\frac{(s+i)^2}{r}\bigg)^{\frac{r-3}{2}}\dif
s\bigg|\lesssim 1, \qquad u \ge 0.
\end{align}
Then \eqref{eq:aux_4_8} combined with \eqref{eq:aux_4_5} and
\eqref{eq:aux_4_4} yields \eqref{eq:aux_4_3}.
We now observe that
\begin{align*}
\bigg|1-\frac{(s+i)^2}{r}\bigg|\le
\begin{cases}
1+\frac{50}{r}, & \text{ if } |s|\le 5,\\
1-\frac{s^2}{2r}, & \text{ if } 5<|s|\le \frac{r^{1/2}}{100},
\end{cases}
\end{align*}
and consequently
\begin{align*}
\bigg|\int_{|s|\le r^{1/2}/100}e^{2\pi i s u }
\bigg(1-\frac{(s+i)^2}{r}\bigg)^{\frac{r-3}{2}}\dif
s\bigg|&\lesssim 1+\int_{5\le|s|\le r^{1/2}/100}\bigg(1-\frac{s^2}{2r}\bigg)^{\frac{r-3}{2}}\dif s\\
&\lesssim 1+r^{1/2}\int_{-1}^1(1-{s^2})^{\frac{r-3}{2}}\dif s\\
& \lesssim 1+\mathcal F \mu(0)\\
&\simeq 1,
\end{align*}
where in the penultimate inequality we used \eqref{eq:aux_4_0} with $\eta=0.$
This completes the proof of \eqref{eq:aux_4_3}, hence, also the proof of \eqref{eq:aux_4_1}. Therefore, the proof of Lemma \ref{lem:aux_4} is completed.
\end{proof}

\begin{lemma}
\label{lem:aux_5}
For the normalized spherical  measure $\mu:=\mu^d$ on $S^{d-1}$
as in \eqref{eq:42} with $d\ge 2$, one has
\begin{equation}
\label{eq:aux_5_1}
|\mathcal F \mu(\xi)-1|\leq 2\pi^2\big(|\xi|/\sqrt{d}\big)^2, \qquad \xi\in \RR^d,
\end{equation}
and
\begin{equation}
\label{eq:aux_5_2}
|\mathcal F \mu(\xi)| \lesssim \big(|\xi|/\sqrt{d}\big)^{-1/2}, \qquad \xi\in \RR^d,
\end{equation}
where the implicit constant is independent of $d$ and $\xi$.
\end{lemma}

\begin{proof}
Using symmetry we see that
\begin{align*}
\mathcal F \mu(\xi)=\int_{S^{d-1}}e^{-2\pi i x \cdot \xi}\,d\mu(x)
=\int_{S^{d-1}}\prod_{j=1}^d\cos (2\pi x_j \xi_j)\,d\mu(x).
\end{align*}
For any sequence $(a_j: j\in\NN_d)\subseteq\CC$ and $(b_j: j\in\NN_d)\subseteq\CC$, if $\sup_{j\in\NN_d}|a_j|\le1$ and $\sup_{j\in\NN_d}|b_j|\le1$ then 
\begin{align}
\label{eq:104}
\Big|\prod_{j=1}^da_j-\prod_{j=1}^db_j\Big|\le\sum_{j=1}^d|a_j-b_j|.
\end{align}
Since $\cos(2x)=1-2\sin^2x$, 
hence, \eqref{eq:104} and inequality \eqref{eq:103}
imply, for $\xi \in \RR^d,$ that
\begin{align*}
|\mathcal F \mu(\xi)-1|
&\le 2\int_{S^{d-1}}\sum_{j=1}^d \sin^2(\pi x_j\xi_j )\,d\mu(x)\\
&\le 2\pi ^2\, \sum_{j=1}^d \xi_j^2\,\int_{S^{d-1}}x_j^2 \,d\mu(x).
\end{align*}
This proves \eqref{eq:aux_5_1}, since $\int_{S^{d-1}}x_j^2 \,d\mu(x)=d^{-1}\int_{S^{d-1}}|x|^2 \,d\mu(x)\le d^{-1}$ for any $j\in\NN_d$.
	
It remains to justify \eqref{eq:aux_5_2}. Take $\xi \in \RR^d$ and
assume first that $|\xi|\le d\sqrt{d}.$ Then $d\ge|\xi|/\sqrt{d}$ and
thus Lemma \ref{lem:aux_4} with $r=d$ easily gives
\eqref{eq:aux_5_2}. Therefore, from now on we assume that
$|\xi|> d\sqrt{d}$. Recalling
\eqref{eq:FsiBes} (with $r=d$) we write
\begin{align*}
\mathcal F \mu(\xi)=\frac{2\pi}{\sigma(S^{d-1})|\xi|^{d/2-1}}J_{d/2-1}(2\pi |\xi|),\qquad \xi\in \RR^d,
\end{align*}
with $J_{d/2-1}$ being the Bessel function \eqref{eq:Bes}. By \cite[eq. 10.14.1]{NIST} we have
\begin{align}
\label{eq:13}
|J_{d/2-1}(y)|\le 1, \qquad y\in\RR.
\end{align}
Inequality \eqref{eq:13} and the formula $\sigma(S^{d-1})=\frac{2\pi^{d/2}}{\Gamma(d/2)}$ imply the first bound below 
\begin{align*}
|\mathcal F \mu(\xi)|\le \min\bigg\{\frac{\pi \Gamma(d/2)}{\pi^{d/2}|\xi|^{d/2-1}}, \frac{C_d}{|\xi|^{(d-1)/2}}\bigg\},
\end{align*}
whereas the second estimate (with some $C_d>0$) follows from well known asymptotics of the Bessel function $|J_{d/2-1}(2\pi|\xi|)|\lesssim_d |\xi|^{-1/2}$, see e.g.\ \cite[eq.\ 10.17.3]{NIST}. If $2\le d\le 4$ we apply the second bound to deduce \eqref{eq:aux_5_2}.  If $d\ge 5$ we use the first bound and  Stirling's formula
$\Gamma(y+1)\simeq \sqrt{2\pi y}\, y^y e^{-y}$ for $y\ge 1$, which gives
\begin{align*}
|\mathcal F \mu(\xi)|\lesssim \frac{d^{d/2-1/2}}{(\pi e)^{d/2}|\xi|^{d/2-1}}
\lesssim \frac{1}{|\xi|^{1/2}}\cdot \frac{d^{d/2-1/2}}{(\pi e)^{d/2}d^{3(d-3)/4}}
\lesssim \frac{1}{|\xi|^{1/2}}\cdot\frac{ d^{7/4}}{(\pi e)^{d/2}d^{d/4}}\lesssim \frac{d^{1/4}}{|\xi|^{1/2}},
\end{align*}
whenever $|\xi|> d^{3/2}.$ This proves \eqref{eq:aux_5_2} and completes the proof of Lemma \ref{lem:aux_5}.
\end{proof}

\subsection{Fourier transform estimates for the discrete spherical measures}

The proofs of inequalities \eqref{eq:ls}, \eqref{eq:is} and \eqref{eq:ss} in Theorem \ref{thm:00} will use  Proposition \ref{prop:0}, which provides
estimates of the multiplier $\mathfrak m_t(\xi)$ at the origin. On the
other hand, they will appeal respectively  to Proposition \ref{prop:mtinfla}, Proposition \ref{prop:2},
and Proposition \ref{prop:4}, which provide estimates of the
multiplier $\mathfrak m_t(\xi)$ at infinity. All these estimates will
be described in terms of a proportionality constant
\begin{align*}
\kappa(d,\la):=\bigg(\frac{\la}{d}\bigg)^{1/2} = \frac{t}{\sqrt{d}}.
\end{align*}
This quantity is just a reparametrization of the proportionality
constant from the Euclidean ball case \cite{BMSW2}, where the radius $N$ is
replaced with $\sqrt{\la}$.

\begin{proposition}
\label{prop:0}
Let $d, \la\in\NN$ be such that $d \ge 2$ and denote $t=\sqrt{\la}.$ Then for every $\xi\in\TT^d$ we have
\begin{align}
  \label{eq:22}
  |\mathfrak m_t(\xi)-1|\le 2\pi^2\kappa(d,\la)^2\|\xi\|^2.
\end{align}
Additionally, we have
\begin{align}
\label{eq:22'}
|\mathfrak m_t(\xi)-(-1)^{\la}|\le 2\pi^2\kappa(d,\la)^2\|\xi+\vo/2\|^2.
\end{align} 
\end{proposition}
\begin{proof}
We first prove \eqref{eq:22}. Exploiting the symmetries of
$S_{t}\cap\ZZ^d$ we have
\begin{align}
\label{eq:43}
\begin{split}
\mathfrak m_t(\xi)=\frac{1}{|S_{t}\cap\ZZ^d|}\sum_{x\in
  S_{t}\cap\ZZ^d}\prod_{j=1}^d \cos(2\pi x_j \xi_j).
\end{split}
  \end{align}
Therefore, using \eqref{eq:104} and the formula  $\cos(2x) = 1-2\sin^2x$, we obtain
\begin{align*}
  |\mathfrak m_t(\xi)-1|\le\frac{2}{|S_{t}\cap\ZZ^d|}\sum_{x\in
  S_t\cap\ZZ^d}\sum_{j=1}^{d}\sin^2(\pi x_j \xi_j).
\end{align*}
Observe  that $|\sin(\pi x y)|\le |x||\sin(\pi y)|$ for every $x\in\ZZ$ and $y\in\RR$, and also
for every $i,j\in\NN_d$ one has
\[
\sum_{x\in S_{t}\cap\ZZ^d}x_i^2=\sum_{x\in S_{t}\cap\ZZ^d}x_j^2
=
\frac{1}{d}\sum_{x\in S_{t}\cap\ZZ^d}|x|^2=\kappa(d,\la)^2
|S_{t}\cap\ZZ^d|.
\]
Thus, taking  into account these observations and changing the order of summations we obtain
\begin{align*}
|\mathfrak m_t(\xi)-1| &\le\frac{2}{|S_{t}\cap\ZZ^d|}\sum_{j=1}^{d}\sin^2(\pi \xi_j)
\frac{1}{d}\sum_{x\in S_{t}\cap\ZZ^d}|x|^2\\
&\le2\pi^2\kappa(d,\la)^2\|\xi\|^2,
\end{align*}
where in the last line we have used \eqref{eq:103}. Now \eqref{eq:22} is justified. 

It remains to prove \eqref{eq:22'}. Using   \eqref{eq:mtsym} and  \eqref{eq:22} we see that
$$|\mathfrak m_t(\xi)-(-1)^{\la}|=|\mathfrak m_t(\xi+\vo/2)-1|\le  2\pi^2\kappa(d,\la)^2\|\xi+\vo/2\|^2,$$
and the proof of \eqref{eq:22'} is completed.
\end{proof}

Proposition \ref{prop:mtinfla} together with a dimension decrease
trick will be the key ingredient in the proof of intermediate-scale inequality \eqref{eq:is} in Theorem \ref{thm:00}. It also suffices to partially treat the large-scale inequality \eqref{eq:ls} in Theorem \ref{thm:00}, but only when the supremum is restricted
to $\{t\in \mathbb D\colon C_3 d^{3/2}\le t \le d^{n}\},$ with a fixed
but arbitrary large $n\in\NN$. To treat the maximal function over the
infinite set $\mathbb D _{C_3}$ we will have to proceed
differently. This will be illustrated in the next section.
\begin{proposition}
\label{prop:mtinfla}
There exist universal constants $C,c>0$ such that for  $d\ge 16,$ $\la\ge Cd^3,$
and $\xi \in \TT^d$ one
has
\begin{equation*}
|\mathfrak m_t(\xi)|\lesssim e^{-2\pi \kappa(d,\la)\|\xi\|}+ e^{-2\pi \kappa(d,\la)\|\xi+\vo/2\|}+\la^{-2}+e^{-cd}.
\end{equation*}
\end{proposition}
\begin{proof}
Recalling that $N=\lfloor \sqrt{\lambda}\rfloor$, $t=\sqrt{\lambda}$ and invoking \eqref{eq:116}, \eqref{eq:wlaxidec}, and \eqref{wla0} from Theorem \ref{thm:asymptotic} (together with $\Gamma(d/2)\approx^d d^{d/2}$), we obtain
\begin{align*}
\bigg|\mathfrak m_t(\xi)-\frac{M_1(\lambda, \xi, N+1)}{|S_t\cap \ZZ^d|}\bigg|
\lesssim^d\frac{(d\lambda)^{d/4}}{|S_t\cap \ZZ^d|}
\lesssim^d \frac{d^{3d/4}}{\la^{d/4-1}}\lesssim^d\frac{d^9}{\la^2}\bigg(\frac{d^3}{\la}\bigg)^{d/4-3}\lesssim \frac{1}{\la^2},
\end{align*}
if $\la \ge C d^3$ for large enough $C>0$.

By \eqref{wla0} and  definition \eqref{eq:3} of $M_1(\lambda, \xi, N+1)$ it suffices to estimate 
\begin{align*}
R_1=R_1(d,\la,\xi):=
 \sum_{\substack{p/q \in H_N\\ q\in\{1, 2\}}} e^{-2\pi i \la p/q}G(p/q;\vfloor{q\xi})
\mathcal F \mu\big(\sqrt{\la}\big(\vfloor{q\xi}/q-\xi\big)\big),
\end{align*}
and
\begin{align*}
R_2=R_2(d,\la,\xi):=
 \sum_{\substack{p/q \in H_N\\ q\ge3}} e^{-2\pi i \la p/q}G(p/q;\vfloor{q\xi})
\mathcal F \mu\big(\sqrt{\la}\big(\vfloor{q\xi}/q-\xi\big)\big).
\end{align*}

We first handle $R_2$, which is easier. We use
$\sup_{y\in\RR^d}|\mathcal F \mu(y)|\le 1$ and \eqref{eq:2} to obtain
\begin{align*}
|R_2|&\le \sum_{q\ge 3} q\cdot (2/q)^{d/2}=2^{d/2}\sum_{q\ge 3} q^{-d/2+1}=3\cdot (2/3)^{d/2} +2^{d/2}\sum_{q\ge 4} q^{-d/2+1}\\
&\lesssim (2/3)^{d/2}+ 2^{d/2}\int_3^{\infty}y^{-d/2+1}\,dy\lesssim (2/3)^{d/2}\le e^{-d/5}.
\end{align*}
Now it remains to estimate $R_1$. We may assume that $\xi \in Q$.

Clearly, if $p/q \in H_N$ and $q\in\{1,2\}$, then
$p/q \in \{1/1,1/2\}$ since we identify $0/1$ with $1/1$. By definition \eqref{eq:def:gsum} for
$x\in \ZZ^d$ we have $G(1/1;x)=1$ and
\begin{align*}
G(1/2;x)=2^{-d}\prod_{j=1}^d\big(1+e^{\pi i (x_j+1)}\big)=\prod_{j=1}^d \ind{_{2\ZZ+1}}(x_j)
=\ind{(2\ZZ+1)^d}(x).
\end{align*}
Therefore, we have
\begin{equation}
\label{eq:mtinflaform}
\begin{split}
R_1&
=\mathcal F \mu\big(\sqrt{\la}(\vfloor{\xi}-\xi)\big)
+e^{-\pi i \la}\, \ind{(2\ZZ+1)^d}(\vfloor{2\xi})\mathcal F \mu\big(\sqrt{\la}(\vfloor{2\xi}/2-\xi)\big)\\
&=
\mathcal F \mu\big(\sqrt{\la}\xi\big)
+e^{-\pi i \la}\, \ind{(2\ZZ+1)^d}(\vfloor{2\xi})\mathcal F \mu\big(\sqrt{\la}(\vfloor{2\xi}/2-\xi)\big),
\end{split}
\end{equation}
where in the second equality we have used that $\mathcal F\mu$ is even and $\vfloor{\xi}=0$ for $\xi\in [-1/2,1/2)^d$.  
Note that
\begin{itemize}
\item[(i)] if $\xi_j\in[-1/4,1/4)$ then $2\xi_j\in [-1/2,1/2),$ and hence, $\vfloor{2\xi_j}=0$ and $\vfloor{2\xi_j}/2-\xi_j=-\xi_j,$
\item[(ii)] if $\xi_j\in[-1/2,-1/4)$ then $2\xi_j\in [-1,-1/2),$ and hence, $\vfloor{2\xi_j}=-1$ and $\vfloor{2\xi_j}/2-\xi_j=-1/2-\xi_j,$
\item[(iii)] if $\xi_j\in[1/4,1/2)$ then $2\xi_j\in [1/2,1),$ and hence, $\vfloor{2\xi_j}=1,$ and  $\vfloor{2\xi_j}/2-\xi_j=1/2-\xi_j$. 
\end{itemize}
Consequently, we see that
\begin{align*}
\ind{(2\ZZ+1)^d}(\vfloor{2\xi})=\ind{([-1/2,-1/4)\cup [1/4,1/2) )^d}(\xi),
\end{align*}
and
\begin{align*}
|\vfloor{2\xi}/2-\xi|=\|\xi+\vo/2\|\qquad\textrm{for}\qquad \vfloor{2\xi} \in (2\ZZ+1)^d.
\end{align*}
Using \eqref{eq:mtinflaform} and Lemma \ref{lem:aux_4} with $r=d$ we
obtain
\begin{align*}
|R_1|&\lesssim
\exp\Big(-2\pi \sqrt{\la}|\xi|/\sqrt{d}\Big)
+\ind{(2\ZZ+1)^d}(\vfloor{2\xi})\,
\exp\Big(-2\pi \sqrt{\la}|\vfloor{2\xi}/2-\xi|/\sqrt{d}\Big)+e^{-cd}\\
&\lesssim e^{-2 \pi \kappa(d,\la)\|\xi\|}+e^{-2 \pi \kappa(d,\la)\|\xi+\vo/2\|}+e^{-cd},
\end{align*}
where $c>0$ is the universal constant from Lemma \ref{lem:aux_4}. This
completes the proof of Proposition
\ref{prop:mtinfla}.
\end{proof}

\subsection{Symmetric diffusion semigroups}
Large part of our estimates will rely on dimension-free bounds for
symmetric diffusion semigroups. Namely, for every $t>0$ let $P_t$
be the semigroup with the multiplier
\begin{align}
\label{eq:119}
\mathfrak p_t(\xi):=e^{-t\sum_{i=1}^d\sin^2(\pi\xi_i)} \quad\text{for}\quad \xi\in\TT^d.
\end{align}
It is very well known that for every
$p\in(1, \infty)$ there is $C_p>0$ independent of $d\in\NN$ such that
for every $f\in\ell^p(\ZZ^d)$ we have
\begin{align}
\label{eq:47}
\big\|\sup_{t>0}|P_tf|\big\|_{\ell^p(\ZZ^d)}\le C_p \|f\|_{\ell^p(\ZZ^d)}.
\end{align}
We refer to 
\cite{Ste1} and to  \cite[Section 4.1]{BMSW3} for more details. 

We close this section by giving a simple application of inequality \eqref{eq:47} and Lemma \ref{lem:aux_5}. 

\begin{lemma}
\label{lem:PerSm}
Let $d\ge 2$ and for $t>0$ define
\begin{equation*}
a_t(\xi):=\mathcal F \mu(t(\xi-\vfloor{\xi})), \qquad \xi\in\RR^d.
\end{equation*}
Then, for all $f\in \ell^2(\ZZ^d)$ one has
\begin{equation}
\label{eq:PerSm}
\big\|\sup_{t\in \mathbb D} |\mathcal F^{-1}(a_t\hat{f})|\big\|_{\ell^2(\ZZ^d)}\lesssim \|f\|_{\ell^2(\ZZ^d)},
\end{equation}
where the implicit constant is independent of the dimension.
\end{lemma}
\begin{proof}
The multiplier $a_t$ is $1$-periodic in each coordinate $\xi_j$ for
$j\in\NN_d$ thus it is well defined as a function on
$\TT^d$. Moreover, for $\xi \in Q$ we have $\xi-\vfloor{\xi}=\xi$ and
$|\xi-\vfloor{\xi}|=\|\xi\|.$ Therefore, recalling that
$\kappa(d,\la)=td^{-1/2}$ and using Lemma \ref{lem:aux_5} we obtain,
for $\xi\in\TT^d$ and $t>0$ the estimates
\begin{equation}
\label{eq:PerSm_0}
|a_t(\xi)-1|\lesssim (\kappa(d,\la) \|\xi\|)^2,
\qquad \text{ and } \qquad
|a_t(\xi)|\lesssim (\kappa(d,\la) \|\xi\|)^{-1/2}.
\end{equation}
Now, \eqref{eq:PerSm_0} implies that
\begin{equation}
\label{eq:PerSm_1}
|\mathfrak p_{t^2d^{-1}} (\xi) -a_t(\xi)|
\lesssim
\min \big\{(\kappa(d,\la)\|\xi\|)^2 ,(\kappa(d,\la)\|\xi\|)^{-1/2}\big\},\qquad \xi \in \TT^d,
\end{equation}
where $\mathfrak p_{t^2d^{-1}}=\mathfrak p_{\kappa(d,\la)^2}$ is the multiplier from \eqref{eq:119}
corresponding to the semigroup operator $P_{\kappa(d,\la)^2}.$ Hence,
using \eqref{eq:47} and \eqref{eq:PerSm_1} we obtain
\begin{align*}
\big\|\sup_{t\in \mathbb D} |\mathcal F^{-1}(a_t\hat{f})|\big\|_{\ell^2(\ZZ^d)}
\leq
\big\|\sup_{t>0} |P_tf|\big\|_{\ell^2(\ZZ^d)}
+\Big\|\Big(\sum_{t\in\mathbb D}\big|\mathcal F^{-1}\big((\mathfrak p_{t^2d^{-1}}-a_t)\hat{f}\big)\big|^2\Big)^{1/2}\Big\|_{\ell^2(\ZZ^d)}
\lesssim \|f\|_{\ell^2(\ZZ^d)}.
\end{align*}
This completes the proof of Lemma \ref{lem:PerSm}.
\end{proof}

We have also a continuous analogue of Lemma \ref{lem:PerSm}.

\begin{lemma}
\label{lem:PerSm'}
Let $d\ge 2$. Then, for all $f\in L^2(\RR^d)$ one has
\begin{equation*}
\big\|\sup_{t\in \mathbb{D}} |
\mathcal F^{-1}(\mathcal F \mu(t\cdot)\mathcal F f)|\big\|_{L^2(\RR^d)}
\lesssim \|f\|_{L^2(\RR^d)},
\end{equation*} 
where the implicit constant is independent of the dimension.
\end{lemma}
\begin{proof}
The proof of Lemma \ref{lem:PerSm'} goes much the same way as the
proof of the previous lemma; the only change is the use of the heat semigroup on
$\RR^d$ in place of the semigroup $P_t$. We omit the details.
\end{proof}

\section{Proof of Theorem \ref{thm:00}: large-scale estimate \eqref{eq:ls}}
\label{sec:la}
The goal of this section is to prove inequality \eqref{eq:ls} in Theorem \ref{thm:00}. In order to
do this we shall need a number of estimates and expansions for the
multiplier $\hat{w}_{\la}$ given by \eqref{eq:wlamultdef}.
For $q\ge 1,$ $(p,q)=1,$ $t>0,$ and $\xi\in\TT^d$ we denote 
\begin{equation}
\label{eq:atpq}
\mathfrak a_{t,p/q}(\xi):=\frac{\la^{d/2-1}}{2 |S_{t}\cap \ZZ^d|}  e^{-2\pi i \la p/q}
G(p/q;\vfloor{q\xi})\mathcal F \sigma\big(\sqrt{\la}\big(\vfloor{q\xi}/{q}-\xi\big)\big).
\end{equation}
Next, for $1\le n\le N+1$, $t > 0$ and $\xi\in\TT^d$ we let
\begin{equation}
\label{eq:btpq}
\mathfrak b_{t,n}(\xi):=\frac{\la^{d/2-1}}{2 |S_{t}\cap \ZZ^d|}
\sum_{\substack{p/q \in H_N\\q\ge n}}\sum_{x \in \ZZ^d}\ e^{-2\pi i \la p/q}G(p/q;x)\psi(q\xi -x)
\mathcal F \sigma\big(\sqrt{\la}(x/q-\xi)\big),
\end{equation}
recall that $N = \lfloor \sqrt{\la} \rfloor$, see \eqref{10.2}.
Theorem \ref{thm:asymptotic} immediately gives a decomposition of $\mathfrak m_t$ in terms of the above multipliers. 

\begin{proposition}
\label{pro:mtxidec}
There exists a universal constant $C>0$ such that for all integers
$d\ge 16$ and $\lambda>0$ satisfying $\la \ge C d^3$, and all integers
$1\le n\le N+1$ (with $N=\lfloor \sqrt{\lambda}\rfloor$), we have the decomposition
\begin{equation}
\label{eq:mtxidec}
\mathfrak m_t(\xi)=\sum_{\substack{p/q \in H_N\\q<n}}\mathfrak a_{t,p/q}(\xi)
+\mathfrak b_{t,n}(\xi)+E_{t,n}(\xi),
\end{equation} 
where the error term $E_{t,n}(\xi)$ satisfies
\begin{equation}
\label{eq:12}
|E_{t,n}(\xi)|\lesssim^d \frac{d^{3d/4}}{\la^{d/4-1}},
\end{equation}
uniformly in $\xi \in \TT^d$, $1\le n\le N+1$, $\la \ge C d^3$ and $d\ge 16$.
\end{proposition}  
\begin{proof}
The first sum in \eqref{eq:mtxidec} and $\mathfrak b_{t,n}(\xi)$ correspond  to the
normalized multipliers from \eqref{eq:3} and \eqref{eq:4}, respectively. Thus \eqref{eq:wlaxidec'} from 
Theorem \ref{thm:asymptotic} applies giving
\begin{align*}
|E_{t,n}(\xi)|\lesssim^d\frac{(d\la)^{d/4}}{|S_{t}\cap \ZZ^d|},
\end{align*}
which in turn combined with \eqref{wla0} and \eqref{eq:8} yields \eqref{eq:12}.
\end{proof}

\begin{theorem}
\label{thm:1'_s1}
There exists a universal constant $C>0$ such that for all $d\ge 16$ and all integers  $1\le p\le q$ so that $(p,q)=1$, one has
\begin{equation*}
\big\|\sup_{t\in \mathbb D_{C,\infty}} |\mathcal F^{-1}(\mathfrak a_{t,p/q}\,\hat{f})|\big\|_{\ell^2(\ZZ^d)}
\lesssim \|f\|_{\ell^2(\ZZ^d)},
\end{equation*} 
where the implicit constant is independent of $p,q,$ and the dimension $d$. 
\end{theorem}
\begin{proof}
Fix $q\in\NN$ and for $w\in \NN_q^d$ we define the sets $T_w:=\{\xi\in \TT^d \colon \vfloor{q\xi} \equiv w \;(\bmod\; q)\}$.
Then any $f\in \ell^2(\ZZ^d)$ admits the decomposition
\begin{equation}
\label{eq:fTw}
f=\sum_{w\in \NN_q^d} f_w, \qquad \text{ with } \qquad \hat{f}_w:=\hat{f}\cdot\ind{T_w}.\end{equation}
For further use we note that the functions $\hat{f}_w$ have pairwise disjoint supports.
For $w\in \NN_q^d$ one has
\begin{equation}
\label{eq:thm:1'_s1_1a}
\xi-\vfloor{q\xi}/q=\xi-w/q-\vfloor{\xi -w/q},\qquad \xi \in T_w.
\end{equation}
Indeed, writing $\vfloor{q\xi}=w+kq$ for some $k\in \ZZ^d,$ we see
that
\begin{equation}
\label{eq:thm:1'_s1_2}
\xi-w/q=k+\xi-\vfloor{q\xi}/q.
\end{equation}
We obtain $\vfloor{\xi-w/q}=k$, since
\begin{align*}
\xi-\vfloor{q\xi}/q=(q\xi-\vfloor{q\xi})/q\in \bigg[\frac{-1}{2q},\frac{1}{2q}\bigg)^d
\subseteq Q.
\end{align*}
Thus,  \eqref{eq:thm:1'_s1_2}
justifies \eqref{eq:thm:1'_s1_1a}. By \eqref{wla0} from Theorem \ref{thm:asymptotic}
(with $C>0$ large enough) we have 
\begin{align*}
\frac{\la^{d/2-1}}{ |S_{t}\cap \ZZ^d|}\simeq \frac{1}{\sigma(S)},\qquad t\in \mathbb D_{C,\infty}.
\end{align*}
By definition \eqref{eq:atpq} and decomposition \eqref{eq:fTw}  we estimate
\begin{align*}
\big\|\sup_{t\in \mathbb D_{C,\infty}} |\mathcal F^{-1}(\mathfrak a_{t,p/q}\,\hat{f})|\big\|_{\ell^2(\ZZ^d)}
\lesssim \sum_{w\in \NN_q^d}|G(p/q;w)| \big\|\sup_{t\in \mathbb D_{C,\infty}}
\big|\mathcal F^{-1}\big(\mathcal F \mu\big(t\big(\vfloor{q\xi}/{q}-\xi\big)\big)\hat{f}_w\big)\big|\big\|_{\ell^2(\ZZ^d)}.
\end{align*}
Now by \eqref{eq:thm:1'_s1_1a} and  Lemma \ref{lem:PerSm} we obtain
\begin{align*}
\big\|\sup_{t\in \mathbb D_{C,\infty}}\big|\mathcal F^{-1}\big(\mathcal F \mu\big(t\big(\vfloor{q\xi}/{q}&-\xi\big)\big)\hat{f}_w\big)\big|\big\|_{\ell^2(\ZZ^d)}\\
&=
\big\|\sup_{t\in \mathbb D_{C,\infty}}\big|\mathcal F^{-1}\big(\mathcal F \mu\big(t\big(\xi-w/q-\vfloor{\xi -w/q}\big)\big)\hat{f}_w\big)\big|\big\|_{\ell^2(\ZZ^d)}\\
&=\big\|\sup_{t\in \mathbb D_{C,\infty}}\big|\mathcal F^{-1}\big(a_t(\xi) \hat{f}_w(\xi+w/q)\big)\big|\big\|_{\ell^2(\ZZ^d)}\\
&\lesssim \|\hat{f}_w\|_{L^2(\TT^d)}.
\end{align*}
Finally by the Cauchy--Schwarz inequality together with \eqref{eq:aux_6_1},  \eqref{eq:fTw} and Plancherel's theorem we conclude
\begin{align*}
\big\|\sup_{t\in \mathbb D_{C,\infty}} |\mathcal F^{-1}(\mathfrak a_{t,p/q}\,\hat{f})|\big\|_{\ell^2(\ZZ^d)}
\lesssim \Big(\sum_{w\in \NN_q^d}|G(p/q;w)|^2\Big)^{1/2}\Big(\sum_{w\in \NN_q^d}\|\hat{f}_w\|_{L^2(\TT^d)}^2\Big)^{1/2}
= \|f\|_{\ell^2(\ZZ^d)}.
\end{align*}
The proof of Theorem \ref{thm:1'_s1} is thus completed. 
\end{proof}

We now pass to our second main ingredient necessary to complete the proof of inequality \eqref{eq:ls} in Theorem
\ref{thm:00}. 
\begin{theorem}
\label{thm:1'_s2}
There exist universal constants $C>0$ and
$n_0\in \NN$ such that for all $d\ge 16$ one has
\begin{equation*}
\big\|\sup_{t\in \mathbb D_{C,\infty}} |\mathcal F^{-1}(\mathfrak b_{t,n_0}\,\hat{f})|\big\|_{\ell^2(\ZZ^d)}
\lesssim \|f\|_{\ell^2(\ZZ^d)},
\end{equation*} 
where the implicit constant is independent of the dimension. 
\end{theorem}

The proof of Theorem \ref{thm:1'_s2} will take up the bulk of this
section.  It will be based on a reduction to the sampling principle of
Magyar, Stein and Wainger \cite{MSW}. We now have to set the necessary notation. 
Assume that $q\in \NN$ and let $B$ be a finite dimensional
Banach space. Let $m\in L^{\infty}(\RR^d; B)$ be a function supported
in $q^{-1} Q$ and define
\begin{align}
\label{eq:18}
m_{\rm per}^q(\xi):=\sum_{x \in \ZZ^d}m(\xi-x/q),\qquad \xi \in \RR^d.
\end{align}
Then $m_{\rm per}^q$ is $1/q$ periodic in each coordinate; in
particular it may be regarded as a function on $\TT^d.$ Consider the
multiplier operators $T$ and $T_{\rm dis}^q$ given by
\begin{equation}
\label{eq:multT}
\mathcal F(T f) (\xi) = m(\xi) \mathcal F(f)(\xi),\qquad f\in L^2(\RR^d),\quad \xi \in \RR^d\end{equation}
and
\begin{equation}
\label{eq:multTdis}
\widehat{T_{\rm dis}^q f} (\xi) =m_{\rm per}^q(\xi)  \hat{f}(\xi),\qquad f\in \ell^2(\ZZ^d),\quad \xi\in \TT^d.
\end{equation}

We shall need two transference results from \cite{MSW}.
\begin{proposition}[{\cite[Corollary 2.1]{MSW} with $B_1=\CC$ and
$B_2=B$}]
\label{prop:MSW1}
Let $1\le p\le \infty$ and $q\in \NN.$ Then, there exists a
universal constant $C>0$ such that
\begin{equation}
\label{eq:prop:MSW1}
\|T_{\rm dis}^q \|_{\ell^p(\ZZ^d)\to \ell^p(\ZZ^d; B) }\leq C^d \|T \|_{L^p(\RR^d)\to L^p(\RR^d; B)}.
\end{equation}
The constant $C$ is independent of $d, p,q$ and the finite dimensional Banach space $B$.
\end{proposition}
\begin{remark*}
The estimate \eqref{eq:prop:MSW1} is not explicitly stated in
\cite{MSW}. However, it can be easily deduced by a careful inspection
of the proof of \cite[Corollary 2.1]{MSW}.  We now present an argument
which absorbs the exponential growth in \eqref{eq:prop:MSW1} by
choosing a large integer $n_0\in\NN$ in Theorem \ref{thm:1'_s2}.
\end{remark*}
Further we consider the multiplier
\begin{equation*}
m(\xi)=\sum_{x\in \ZZ^d} \gamma_{x} \Phi(\xi-x/q),\qquad \xi \in \RR^d,
\end{equation*}
satisfying 
\begin{enumerate}[label*={\alph*})]
\item $\Phi \in C_c^{\infty}(q^{-1}Q)$ and as a function on $Q$ it has
the Fourier expansion
\begin{align*}
\Phi(\xi)=\sum_{x \in \ZZ^d}\varphi_{x}e^{-2\pi i x \cdot \xi},\qquad \xi\in Q,
\end{align*}
with $\sum_{x \in \ZZ^d}|\varphi_x|\le A.$
\item $(\gamma_{x})_{x\in\ZZ^d}$ is a $q \ZZ^d$ periodic sequence; i.e.,
$\gamma_{x}=\gamma_{x'}$ if $x - x ' \in q\ZZ^d.$
\end{enumerate}
Using a) and b) and Plancherel's theorem we conclude that
\begin{equation}
\label{eq:prop:MSW2}
\big\|\mathcal F^{-1}(m\hat{f})\big\|_{\ell^2(\ZZ^d)}\leq A \sup_{x \in \ZZ^d}|\gamma_{x}|\|f\|_{\ell^2(\ZZ^d)}.
\end{equation}  

We now are  ready to prove Theorem \ref{thm:1'_s2}.
\begin{proof}[Proof of Theorem \ref{thm:1'_s2}]
By the monotone convergence theorem it suffices to show that for $C>0$
and $n_0\in \NN$ large enough the uniform  estimate
\begin{equation}
\label{eq:thm:1'_s2_1'}
\sup_{J\in\NN}\big\|\sup_{t\in \mathbb D_{C,\infty}\cap [0,J]} |\mathcal F^{-1}(\mathfrak b_{t,n_0}\,\hat{f})|\big\|_{\ell^2(\ZZ^d)}
\lesssim \|f\|_{\ell^2(\ZZ^d)},
\end{equation} 
holds true. From now on  $J\in \NN$ is fixed. We also fix a function $\varphi' $ in
$C_c^{\infty}((-1/2,1/2))$ which is equal to $1$ on $[-1/4,1/4],$ and
satisfies $\|\varphi'\|_{L^{\infty}(\RR)}\le 1;$ and
denote
\begin{align*}
\psi'(\xi):=\prod_{j=1}^d\varphi'(\xi_j),\qquad \xi\in \RR^d.
\end{align*}
Then clearly, $\psi \psi'=\psi,$ where $\psi$ is the function defined
by \eqref{eq:psidef}. We  set
\begin{align*}
H_{\infty}:=\bigcup_{N\in\NN} H_N=\{p/q\in\QQ\colon 1\le p\le q,\, (p,q)=1\},
\end{align*}
recall that we identify $0/1$ with $1/1$.
For every $p/q\in H_{\infty}$ we also define two operators 
$U^q\colon \ell^2(\ZZ^d)\to \ell^2(\ZZ^d; B)$, where $B:=\ell^{\infty}(\NN_{J^2} \cap \mathbb D^2)$,   and
$V^{p/q}\colon \ell^2(\ZZ^d)\to \ell^2(\ZZ^d)$ by setting
\begin{equation}
\label{eq:16}
U^q f:=\Big(\mathcal F^{-1}\Big(\sum_{x \in \ZZ^d}\psi_q(\xi -x/q)
\mathcal F \mu\big(\sqrt{\la}(\xi-x/q)\big)\hat{f}\Big)\Big)_{\la \in \NN_{J^2 } \cap \mathbb D^2 },\qquad f\in \ell^2(\ZZ^d),
\end{equation}
where 
\[
\psi_q(\xi)=\psi(q\xi), \qquad 
\psi_q'(\xi)=\psi'(q\xi), \qquad \xi\in\RR^d,
\]
and
\begin{align*}
V^{p/q} f:=	\mathcal F^{-1}\Big(\sum_{x \in \ZZ^d}\,G(p/q;x)\psi'_q(\xi -x/q)
\hat{f}\Big),\qquad f\in \ell^2(\ZZ^d).
\end{align*}
Since   $t\in \mathbb D_{C,\infty}$ we have $t\ge Cd^{3/2}$
so that $N=\lfloor \sqrt{\lambda}\rfloor\ge Cd^{3/2}-1$. Recalling \eqref{eq:btpq}, using
\eqref{wla0} and the fact that $\psi_q \psi_q' =\psi_q$  we see that 
\begin{equation}
\label{eq:17}
\big\|\sup_{t\in \mathbb D_{C,\infty}\cap [0,J]} |\mathcal F^{-1}(\mathfrak b_{t,n_0}\,\hat{f})|\big\|_{\ell^2(\ZZ^d)}
\lesssim \sum_{\substack{p/q \in H_{\infty}\\q\ge n_0}} \|U^q V^{p/q} f\|_{\ell^2(\ZZ^d; B)},
\end{equation}
uniformly in $1 \le n_0\le   Cd^{3/2}$. We specify $n_0$ later.
	
Now, observe that the proof of \eqref{eq:thm:1'_s2_1'} will be
completed once we show that there are universal constants $D_1, D_2>0$ such that
\begin{align}
\label{eq:14}\|U^q f\|_{\ell^2(\ZZ^d; B)}&\leq D_1^d\, \|f\|_{\ell^2(\ZZ^d)},\\
\label{eq:15}\|V^{p/q}f\|_{\ell^2(\ZZ^d)}&\leq D_2^d\, q^{-d/2}\|f\|_{\ell^2(\ZZ^d)},
\end{align}
uniformly in $p/q \in H_{\infty}$ and $d \ge 16$. 
Indeed, using \eqref{eq:17} and 
assuming momentarily that  \eqref{eq:14} and \eqref{eq:15} hold and taking $n_0:= \lfloor(D_1 D_2)^{10}\rfloor+1$ we obtain 
\begin{align*}
\big\|\sup_{t\in \mathbb D_{C,\infty}\cap [0,J]} |\mathcal F^{-1}(\mathfrak b_{t,n_0}\,\hat{f})|\big\|_{\ell^2(\ZZ^d)}
& \lesssim 
\|f\|_{\ell^2(\ZZ^d)} \sum_{q\ge n_0}(D_1D_2)^dq^{-d/2+1} \\
& \lesssim
\|f\|_{\ell^2(\ZZ^d)} \sum_{q\ge n_0} q^{-2d/5+1}\bigg(\frac{(D_1 D_2)^{10}}{q}\bigg)^{d/10} 
\lesssim 
\|f\|_{\ell^2(\ZZ^d)},
\end{align*}
as long as $d\ge 16$. 
Therefore we reduced  the proof of Theorem
\ref{thm:1'_s2} to showing \eqref{eq:14} and \eqref{eq:15}.

\paragraph{\bf Proof of estimate \eqref{eq:14}}
Considering $(m_{\rm per}^q(\xi))_{\la \in \NN_{J^2} \cap \mathbb D^2}\in B=\ell^{\infty}(\NN_{J^2 } \cap \mathbb D^2 )$ from \eqref{eq:18} with
\begin{align*}
m(\xi):=\big(\psi_q(\xi)\mathcal F \mu\big(\sqrt{\la}\xi\big)\big)_{\la \in \NN_{J^2 } \cap \mathbb D^2 },\qquad \xi \in \RR^d
\end{align*}
it is easy to see, 
from definition \eqref{eq:16}, that $U^q$ is equal to
$T_{\rm dis}^q$ defined in 
\eqref{eq:multTdis}. Since $\supp m  \subseteq q^{-1}Q$ and $m$ belongs to
$L^{\infty}(\RR^d; B)$ and satisfies
$\|m\|_{L^{\infty}_{B}(\RR^d)}\leq 1$, hence,
Proposition \ref{prop:MSW1} applies and  we obtain
\begin{align*}
\|U^q \|_{\ell^2(\ZZ^d)\to \ell^2(\ZZ^d; B)} \lesssim^d
\|T \|_{L^2(\RR^d)\to L^2(\RR^d; B)},
\end{align*}
where $T$ is defined by \eqref{eq:multT}. The latter operator norm, 
using Lemma \ref{lem:PerSm'} (see also \cite{SteinMax,StStr}), is bounded and for $f\in L^2(\RR^d)$ we have
\begin{align*}
\|T f\|_{L^2(\RR^d; B)}
\lesssim  
\| \mathcal F^{-1} ( \psi_q
\mathcal F f ) \|_{L^2(\RR^d)}
\le 
\| \psi_q \|_{L^\infty(\RR^d)} \|f\|_{L^2(\RR^d)}
\leq \|f\|_{L^2(\RR^d)}. 
\end{align*}
Thus \eqref{eq:14} follows with a universal constant $D_1>0$.

\paragraph{\bf Proof of estimate \eqref{eq:15}}
Here we apply inequality \eqref{eq:prop:MSW2} with $\gamma_x:=G(p/q;x)$
and $\Phi:=\psi'_q$. Note that $\Phi$ satisfies condition a); here we
shall use the fact that $\psi'\in C_c^{\infty}(Q).$ Indeed,
$\Phi \in C_c^{\infty}(q^{-1}Q)$ and by Fourier inversion theorem, we
have
$\varphi_x=\int_Q \psi'(q\xi) e^{2\pi i x \cdot \xi}\,d\xi=q^{-d}\mathcal F (\psi')(x/q)$.
Since $\mathcal F(\psi')$ is a tensor product of Schwartz functions,
we thus obtain $\sum_{x \in \ZZ^d}|\varphi_x|\lesssim^d1$.  Therefore,
$\Phi$ satisfies condition a) of inequality \eqref{eq:prop:MSW2} with a
constant $A^d$ in place of $A$, where $A$ is a universal constant depending
only on $\psi'$. Moreover, $G(p/q;x)$ is $q\ZZ^d$ periodic, and
$\sup_{x \in \ZZ^d}|G(p/q;x)|\le (2/q)^{d/2},$ by \eqref{eq:2}. Therefore, an application of inequality \eqref{eq:prop:MSW2} is justified and leads to \eqref{eq:15}
with $D_2=\sqrt{2}A.$
\end{proof}

\subsection{All together: proof of inequality \eqref{eq:ls} in Theorem \ref{thm:00}} Let $C_3>0$
be a universal constant so large that the conclusions of Proposition \ref{pro:mtxidec},  Theorem
\ref{thm:1'_s1}, and Theorem \ref{thm:1'_s2} are
satisfied. Take $n_0\in\NN$ large enough so that the
conclusion of Theorem \ref{thm:1'_s2} holds. Then, using Proposition
\ref{pro:mtxidec}, together with Theorem \ref{thm:1'_s1}, and Theorem
\ref{thm:1'_s2}, and Plancherel's theorem we see that

\begin{equation}
\label{eq:alltogether:thm:1'}
\begin{split}
\big\|\sup_{t\in \mathbb D_{C_3,\infty}}|\mathcal A^d_{t}f|\big\|_{\ell^2(\ZZ^d)}
&\lesssim  \|f\|_{\ell^2(\ZZ^d)}+\sum_{t\in \mathbb D_{C_3,\infty}}\big\|\mathcal F^{-1}\big(E_{t,n_0}\,\hat{f}\big)\big\|_{\ell^2(\ZZ^d)}\\
&\lesssim \bigg(1+\sum_{\la\ge C_3^2 d^3}A^d\frac{d^{3d/4}}{\la^{d/4-1}}\bigg)\|f\|_{\ell^2(\ZZ^d)}\\
&\lesssim \bigg(1+\sum_{\la\ge C_3^2 d^3}\frac{d^9}{\la^2}\bigg(\frac{d^3 A^4}{\la}\bigg)^{d/4-3}\bigg)\|f\|_{\ell^2(\ZZ^d)},
\end{split}\end{equation}
where $A>1$ is a universal constant and $E_{t,n_0}$ is the error term
satisfying \eqref{eq:12}. Since $d\ge16$ taking in \eqref{eq:alltogether:thm:1'}
a constant $C_3>0$ such that $C_3^2 \ge 2A^4$ we obtain
\begin{align*}
\big\|\sup_{t\in \mathbb D_{C_3,\infty}}|\mathcal A^d_{t}f|\big\|_{\ell^2(\ZZ^d)}
\lesssim \Big(1+\sum_{\la\ge C_3^2 d^3}\la^{-2}\Big)\|f\|_{\ell^2(\ZZ^d)} \lesssim \|f\|_{\ell^2(\ZZ^d)}.
\end{align*}
This completes the proof of \eqref{eq:ls}.\qed

\begin{remark}\label{rem:100}
We note that it is possible to show a slightly stronger result than
\eqref{eq:ls}. Namely, there are universal constants $C, C_3 >0$ such
that
\begin{align}
\label{25.1}
\sup_{d\ge5}\mathcal C(2,\sqrt{\NN} \cap (C_3 d^{3/2}, \infty), S^{d-1})\le C.
\end{align}
To prove \eqref{25.1} we need a refinement of \eqref{eq:PerSm}, where
$\sup_{t\in \mathbb D}$ is replaced by $\sup_{t > 0}$. For this purpose we have to prove the following estimate
\begin{align} \label{D20.22}
| \langle \eta, \nabla (\mathcal F\mu^r) (\eta) \rangle |
\lesssim
1, \qquad r \ge 2, \quad \eta \in \RR^r,
\end{align}
where the implicit constant is independent of $r$ and $\eta \in \RR^r$.
Proceeding as in the proof of Lemma~\ref{lem:aux_4} we may deduce \eqref{D20.22}.
We omit the details.

\end{remark}

\section{Proof of Theorem \ref{thm:00}: intermediate-scale estimate \eqref{eq:is}}

\label{sec:me}
This section is devoted to the proof of inequality \eqref{eq:is} in Theorem \ref{thm:00}. The main
estimate for Fourier transforms in this regime is the following
proposition.
\begin{proposition}
\label{prop:2}
Let $100d\le \la \le d^3 .$  Then   for all $\xi\in\TT^d$ we have
\begin{align}
  \label{eq:23}
  |\mathfrak m_t(\xi)|
	\lesssim (\kappa(d, \la)\|\xi\|)^{-1}+(\kappa(d, \la)\|\xi+\vo/2\|)^{-1}+\kappa(d, \la)^{-4}. 
\end{align}
\end{proposition}

Assuming for a moment Proposition \ref{prop:2} we may proceed as
follows. We let \begin{align}
\label{eq:107}
V_{\xi}:=\{i\in\NN_d\colon \cos(2\pi\xi_i)<0\}=\{i\in\NN_d\colon 1/4<\| \xi_i \| \le1/2\} \quad \text{for} \quad \xi\in\TT^d.
\end{align}
We will study the maximal functions associated with the following multipliers
\begin{align}
\label{eq:96}
p^1_{\lambda}(\xi)&:=  e^{-\kappa(d,\la)^2\sum_{i=1}^d\sin^2(\pi\xi_i)} && \text{ if
} |V_{\xi}|\le d/2,\\
\label{eq:97}p^2_{\lambda}(\xi)&:=(-1)^{\la}
e^{-\kappa(d,\la)^2\sum_{i=1}^d\cos^2(\pi\xi_i)} && \text{ if
} |V_{\xi}|>d/2.
\end{align}
The sets $V_{\xi}$ and the multipliers $p^1_{\lambda}$ and  $p^2_{\lambda}$ will be also used in the small scales case.   

Propositions \ref{prop:0} and \ref{prop:2} imply the following estimates.

\begin{proposition}
\label{prop:4p}
Let $d,\la\in\NN$ be such that $100d\le \la\le d^3$. Then for every
$\xi\in\TT^d$ we have the following
bounds: \begin{enumerate}[label*={\arabic*}.]
\item if $|V_{\xi}|\le d/2$, then
\begin{align}
\label{eq:94p}
|\mathfrak m_t(\xi)-p^1_{\lambda}(\xi)|\lesssim \min\big\{(\kappa(d, \la)\|\xi\|)^{-1},\kappa(d, \la)\|\xi\|\big\}+\kappa(d,\la)^{-1};
\end{align}
\item if $|V_{\xi}|>d/2$, then
\begin{align}
\label{eq:95p}
|\mathfrak m_t(\xi)-p^2_{\lambda}(\xi)|\lesssim \min\big\{(\kappa(d, \la)\|\xi+\vo/2\|)^{-1},\kappa(d, \la)\|\xi+\vo/2\|\big\}+\kappa(d,\la)^{-1}.
\end{align}
\end{enumerate}
\end{proposition}
\begin{proof}
Note that $\sum_{i=1}^d \sin^2(\pi \xi_i)\simeq \|\xi\|^2$ and
$\sum_{i=1}^d \cos^2(\pi \xi_i)\simeq \|\xi+\vo/2\|^2.$ We consider
two cases.
	
If $|V_{\xi}|\le d/2,$ then $\|\xi+\vo/2\|\gtrsim \sqrt{d} \ge 1$. Therefore
using Proposition \ref{prop:2} we obtain
\begin{align*}
|\mathfrak m_t(\xi)-p^1_{\lambda}(\xi)|
&\lesssim (\kappa(d, \la)\|\xi\|)^{-1}+(\kappa(d, \la)\|\xi+\vo/2\|)^{-1}+\kappa(d,\la)^{-4}
\lesssim (\kappa(d, \la)\|\xi\|)^{-1} + \kappa(d,\la)^{-1}.
\end{align*}
Moreover, combining
$|\mathfrak m_t(\xi)-p^1_{\lambda}(\xi)|\le |\mathfrak m_t(\xi)-1|+|1-p^1_{\lambda}(\xi)|$
with \eqref{eq:22} we see that
\begin{align*}
|\mathfrak m_t(\xi)-p^1_{\lambda}(\xi)|&\lesssim \kappa(d,\la)\|\xi\|.
\end{align*}
  
Consider now $|V_{\xi}|\ge d/2.$ In this case $\|\xi\|\gtrsim \sqrt{d} \ge 1$. Thus, using Proposition \ref{prop:2} we
have
\begin{align*}
|\mathfrak m_t(\xi)-p^2_{\lambda}(\xi)|
&\lesssim (\kappa(d, \la)\|\xi\|)^{-1}+(\kappa(d, \la)\|\xi+\vo/2\|)^{-1}+\kappa(d,\la)^{-4}
\lesssim (\kappa(d, \la)\|\xi+\vo/2\|)^{-1}+\kappa(d,\la)^{-1}.
\end{align*}
Next,
$|\mathfrak m_t(\xi)-p^2_{\lambda}(\xi)|\le|\mathfrak m_t(\xi)-(-1)^{\la}|+|(-1)^{\la}-p^2_{\lambda}(\xi)|$
together with \eqref{eq:22'}  gives
\begin{align*}|\mathfrak m_t(\xi)-p^2_{\lambda}(\xi)|&\lesssim \kappa(d,\la)\|\xi+\vo/2\|.
\end{align*}
Hence, \eqref{eq:94p} and \eqref{eq:95p} are proved as claimed.
\end{proof}

Having proved Proposition \ref{prop:4p} we can now establish estimate
\eqref{eq:is} in Theorem \ref{thm:00}.  Here we shall need the
symmetric diffusion semigroup given by \eqref{eq:119}.

\begin{proof}[Proof of inequality \eqref{eq:is} in Theorem \ref{thm:00}]
\label{p:pf_T_1}
Let $f\in\ell^2(\ZZ^d)$ and we write $f=f_1+f_2$, where
$\hat{f}_1(\xi):=\hat{f}(\xi)\ind{\{\eta\in\TT^d\colon |V_{\eta}|\le d/2\}}(\xi)$. Then
\begin{align*}
\big\|\sup_{t\in\mathbb D_{C_1,C_2}}|\mathcal F^{-1}(\mathfrak m_t\hat{f})|\big\|_{\ell^2(\ZZ^d)}
\le& \sum_{i=1}^2\big\|\sup_{t\in\mathbb D_{C_1,C_2}}|\mathcal F^{-1}(p_{t^2}^i\hat{f}_i)|\big\|_{\ell^2(\ZZ^d)}\\
&+\sum_{i=1}^2\Big\|\Big(\sum_{t\in\mathbb D_{C_1,C_2}}\big|\mathcal F^{-1}\big((\mathfrak m_t - p_{t^2}^i)\hat{f}_i\big)\big|^2\Big)^{1/2} \Big\|_{\ell^2(\ZZ^d)}.
\end{align*}
The usual square function argument permits therefore to reduce the problem to controlling the maximal functions associated with the
multipliers $p^1_{\lambda}$ and  $p^2_{\lambda}$. Indeed, taking
$C_1=10$ and $C_2=1,$ Plancherel's theorem and Proposition \ref{prop:4p} imply 
\[
\sum_{i=1}^2\Big\|\Big(\sum_{t\in\mathbb D_{C_1,C_2}}\big|\mathcal F^{-1}\big((\mathfrak m_t - p_{t^2}^i)\hat{f}_i\big)\big|^2\Big)^{1/2} \Big\|_{\ell^2(\ZZ^d)}
\lesssim  \|f_1\|_{\ell^2(\ZZ^d)}+\|f_2\|_{\ell^2(\ZZ^d)}\le 2\|f\|_{\ell^2(\ZZ^d)}.
\]

Thus we only have to bound the maximal functions corresponding to the
multipliers \eqref{eq:96} and \eqref{eq:97}. Indeed, since
$p^1_{t^2}(\xi)=\mathfrak p_{\kappa(d,\la)^2}(\xi)$, then by
\eqref{eq:47} we obtain
\begin{align}
\label{eq:99p}
\big\|\sup_{t>0}|\mathcal F^{-1}(p^1_{t^2}\hat{f}_1)|\big\|_{\ell^2(\ZZ^d)}\lesssim \|f\|_{\ell^2(\ZZ^d)}.
\end{align}
It is also  not difficult to see that
\begin{align}
\label{eq:100p}
\big\|\sup_{t>0}|\mathcal F^{-1}(p^2_{t^2}\hat{f}_2)|\big\|_{\ell^2(\ZZ^d)}\lesssim \|f\|_{\ell^2(\ZZ^d)}.
\end{align}
Indeed, since $p^2_{t^2} (\xi) = (-1)^{\lambda} \mathfrak p_{\kappa(d,\la)^2}(\xi - \vo/2)$, 
letting $\hat{F}_2(\xi):=\hat{f}_2(\xi+\vo/2)$ we obtain 
\begin{align*}
\sup_{t>0} \big| \mathcal F^{-1}(p^2_{t^2}\hat{f}_2) \big|
=
\sup_{t>0} \big| \mathcal F^{-1}(\mathfrak p_{t}\hat{F}_2) \big|.
\end{align*}
Therefore using \eqref{eq:47} and the fact that 
$\|F_2\|_{\ell^2(\ZZ^d)}=\|f_2\|_{\ell^2(\ZZ^d)}\le \|f\|_{\ell^2(\ZZ^d)}$ we get \eqref{eq:100p}. 
This
finishes the proof of inequality \eqref{eq:is} in Theorem \ref{thm:00}.
\end{proof}

\subsection{Some preparatory estimates}

What is left is to prove Proposition \ref{prop:2}. The key idea will be to
use the dimension reduction trick from \cite{BMSW2}.  The proof of
Proposition \ref{prop:2} will require several lemmas, which are
similar to those proved in \cite[Section 2]{BMSW2}. Their proofs will
mostly rely on Lemma
\ref{lem:lfs} and the methods developed in \cite[Section 2]{BMSW2}.

\begin{lemma}
\label{lem:5}
Let $d\ge 10.$ Given $\varepsilon_1, \varepsilon_2\in(0, 1]$ we define for every $\la\in\NN$ the set
\[
E=\big\{x\in S_t\cap\ZZ^d\colon |\{i\in\NN_d\colon |x_i|\ge
\varepsilon_2\kappa(d,\la)\}|\le\varepsilon_1 d\big\}.
\]
If $\varepsilon_1, \varepsilon_2\in(0, 1/10]$ and $d\ge \kappa(d,\la)\ge10$, then we have
\begin{align*}
|E| \lesssim  e^{-\frac{d}{11}}|S_t \cap\ZZ^d|.
\end{align*}
\end{lemma}
\begin{proof}
Considering
\begin{align*}
\tilde{E}=\big\{x\in B_t^2(d)\cap\ZZ^d\colon |\{i\in\NN_d\colon |x_i|\ge
\varepsilon_2\kappa(d,\la)\}|\le\varepsilon_1 d\big\}
\end{align*}
and using \cite[Lemma 2.4]{BMSW2} (notice that this result is still valid if we allow $N \in \sqrt{\NN}$ there) we see that
\begin{equation}
\label{eq:BMSW2_1}
|\tilde{E}| \lesssim  e^{-\frac{d}{10}}|B_t^2(d)\cap\ZZ^d|.
\end{equation}
Clearly $E\subseteq \tilde{E}$, thus inequality \eqref{eq:BMSW2_1} and Lemma \ref{lem:lfs} show that
\begin{align*}
|E|\le |\tilde{E}|\lesssim e^{-\frac{d}{10}}|B_t^2(d)\cap\ZZ^d| \leq  e^{-\frac{d}{10}} (2t+1)^{4} |S_t\cap \ZZ^d|\lesssim e^{-\frac{d}{11}} |S_t\cap \ZZ^d|,
\end{align*}
where in the last inequality used that
$t=\sqrt{d} \kappa(d,\la)\le d^{3/2}.$ This completes the proof of the
lemma.
\end{proof}

Having proved Lemma \ref{lem:5} we can now perform the
dimension-reduction trick as in \cite[Section 2.3]{BMSW2}. First we
need to justify an analogue of \cite[Lemma 2.7]{BMSW2}. Here a
concentration inequality for the hypergeometric distribution (see
\cite[Lemma 2.5]{BMSW2}) will be important.

\begin{lemma}
\label{lem:8}
Let $d\ge 10$ and $\la \in \NN$. For $\varepsilon\in (0,1/50]$ and $r$
an integer between $1$ and $d$ we define
\begin{equation}
\label{eq:lem:8_1}
E:=\{x\in S_{t}\cap \ZZ^d\colon
\sum_{i=1}^rx_i^2<\varepsilon^3\kappa(d,\la)^2r\}
\end{equation}
If $d\ge \kappa(d,\la)\ge 10$ then we have
\begin{align*}
|E|\lesssim e^{-\frac{\varepsilon r}{11}}|S_{t}\cap \ZZ^d|.
\end{align*}	
\end{lemma}

\begin{proof}
Let $\delta_1\in(0, 1/10]$ be such that $\delta_1 = 5\varepsilon$ and
define $I_x:=\{i\in\NN_d: |x_i|\ge\varepsilon\kappa(d,\la)\}.$ We have
$E\subseteq E_1\cup E_2,$ where
\begin{align*}
E_1&:=\{x\in S_{t}\cap \ZZ^d\colon
\sum_{i\in I_x\cap\NN_r}x_i^2<\varepsilon^3\kappa(d,\la)^2r\quad \textrm{and}\quad |I_{x}|\ge \delta_1d\},\\
E_2&:=\{x\in S_{t}\cap \ZZ^d\colon
|I_{x}|<\delta_1d\}.
	\end{align*} 
By Lemma \ref{lem:5} (with $\varepsilon_1=\delta_1$ and $\varepsilon_2=\varepsilon$) we have $|E_2|\lesssim e^{-\frac{d}{11}}|S_{t}\cap \ZZ^d|.$ 
	
To estimate $E_1$ we note that 
$I_{\tau^{-1}\circ x}=\tau(I_x)$ for any $\tau \in {\rm Sym}(d)$
and $x\in S_t\cap\ZZ^d$. Therefore we have
\begin{align*}
|E_1|
& = 
\sum_{x\in S_t\cap\ZZ^d} \frac{1}{d!} \sum_{\tau\in{\rm Sym}(d)} \ind{E_1}(\tau^{-1} \circ x)
\\
&=\sum_{x\in S_t\cap\ZZ^d}\mathbb P[\{\tau\in{\rm Sym}(d)\colon
\sum_{i\in \tau(I_x)\cap\NN_r}x_{\tau^{-1}(i)}^2<\varepsilon^3\kappa(d, \la)^2r\ \text{ and }\ |I_x|\ge\delta_1d\}]\\
&\le \sum_{x\in S_t\cap\ZZ^d}	\mathbb P[\{\tau\in{\rm Sym}(d)
\colon |\tau(I_x)\cap\NN_r|\le \varepsilon r\}] \ind{|I_x|\ge\delta_1 d}\\
&\le 2 e^{-\frac{\delta_1 r}{10}}|S_{t}\cap \ZZ^d|, 
\end{align*}
where the last bound is a consequence of \cite[Lemma 2.5]{BMSW2} (with
$I=I_x$, $J=\NN_r$, $\delta_2= \varepsilon = \frac{\delta_1}{5}$).
\end{proof}

Lemma \ref{lem:9} below will play an essential role in the proof of
Proposition \ref{prop:2}.
\begin{lemma}
\label{lem:9}
For $d\in\NN$, $\la=t^2\in \NN$ and $\varepsilon\in(0, 1/50],$ if
$10\le \kappa(d, \la)\le d$, then for every $4\le r\le d$ and $\xi\in\TT^d$
we have
\begin{align}
\label{eq:39}
|\mathfrak m_t(\xi)|\lesssim\sup_{\varepsilon^3\kappa(d, \la)^2r\le l \le \la}|\mathfrak
m_{\sqrt{l}}^{(r)}(\xi_1,\ldots, \xi_r)|+e^{-\frac{\varepsilon r}{11}},
\end{align}
where
\begin{align}
\label{eq:21}
\mathfrak m_{t}^{(r)}(\eta)=\frac{1}{|S_{t}^{r-1}\cap\ZZ^r|}\sum_{x\in
S_{t}^{r-1}\cap\ZZ^r}e^{2\pi i \eta\cdot x},\qquad \eta\in \TT^r, \quad r \ge 4,
\end{align}
is the lower dimensional multiplier with $r\in \NN$ and $t>0$.
\end{lemma}

\begin{proof}
The case $r = d$ is trivial so we assume that $r < d$.
We identify $\RR^d\equiv\RR^r\times\RR^{d-r}$ and
$\TT^d\equiv\TT^r\times\TT^{d-r}$ and we will write
$\RR^d\ni x=(x^1, x^2)\in \RR^r\times\RR^{d-r}$ and
$\TT^d\ni \xi=(\xi^1, \xi^2)\in \TT^r\times\TT^{d-r}$, respectively. Note that Lemma
\ref{lem:8} gives the disjoint decomposition
\begin{align}
\label{eq:103'}
\begin{split}
S_{t}\cap \ZZ^d&=E\cup\bigcup_{\substack{l\in \NN \\ \la \ge l \ge \varepsilon^3\kappa(d, \lambda)^2r}}\big(S_{\sqrt{l}}^{r-1}\cap\ZZ^r\big)
\times\big(S^{d-r-1}_{\sqrt{\la-l}}\times\ZZ^{d-r}\big),
\end{split}
\end{align}
with $E$ being defined in \eqref{eq:lem:8_1}. Invoking \eqref{eq:103'} we immediately obtain \eqref{eq:39}.
\end{proof}

The utility of the above lemma comes from the fact that, once $r$ is
appropriately chosen, Lemma \ref{lem:9} allows us to apply Proposition
\ref{prop:mtinfla} for the lower-dimensional multipliers
$\mathfrak m_{\sqrt{l}}^{(r)}(\xi_1,\ldots, \xi_r).$
\begin{lemma} \label{lem:13}
Let $C>0$ be a universal constant for which the conclusion of
Proposition \ref{prop:mtinfla} holds and take $\delta\in(0, 1/2)$ and
$\varepsilon\in(0, 1/50].$ Let $d,\la\in\NN$ be such that $\kappa(d,\lambda) \le d$ and take $r$ an integer
such that $1\le r\le d$ and
\begin{equation}
\label{eq:lem:13_0}
\max\{21, \varepsilon^{\frac{3\delta}{2}}\kappa(d, \la)^{\delta} C^{-\delta/2}/4\}
\le r\le
\max\{21,\varepsilon^{\frac{3\delta}{2}}\kappa(d,\la)^{\delta} C^{-\delta/2}\}.\end{equation}
Then, for every $\xi=(\xi_1,\ldots,\xi_d)\in\TT^d$ we have
\begin{equation}
\label{eq:lem:13_1}
|\mathfrak m_t(\xi)|\lesssim_{\delta,\varepsilon}\kappa(d, \la)^{-4}+e^{-2\pi \varepsilon^{3/2}\kappa(d,\la)\|\eta\|}+ e^{-2\pi \varepsilon^{3/2}\kappa(d,\la)\|\eta+\vo/2\|},
\end{equation}
where $\eta=(\xi_1,\ldots, \xi_r)$.
\end{lemma}

\begin{proof}
If $\kappa(d, \la)\le \varepsilon^{-\frac{3}{2}}$, then there is
nothing to do, since the implied constant in question is allowed to
depend on $\delta$ and $\varepsilon$. We will assume that
$\kappa(d, \la)\ge \varepsilon^{-\frac{3}{2}}$, which ensures that
$\kappa(d, \la)\ge10$. For such $\kappa(d,\la)$ if
$\varepsilon^{\frac{3\delta}{2}}\kappa(d, \la)^{\delta} C^{-\delta/2}/4\le 21$,
then $\kappa(d,\la)\simeq_{\varepsilon,\delta} 1,$ and
\eqref{eq:lem:13_1} is again obvious. Therefore, without loss of
generality we assume that $\kappa(d,\la)\ge 10$ and that
\begin{equation}
\label{eq:lem:13_2}
21 \le 
\varepsilon^{\frac{3\delta}{2}}\kappa(d, \la)^{\delta} C^{-\delta/2}/4
\le r\le
\varepsilon^{\frac{3\delta}{2}}\kappa(d,\la)^{\delta} C^{-\delta/2}.
\end{equation}
In view of Lemma \ref{lem:9} and \eqref{eq:lem:13_2} we have
\begin{align}
\label{eq:35}
|\mathfrak m_t(\xi)|\lesssim_{\delta,\varepsilon} \sup_{\varepsilon^3\kappa(d, \la)^2r\le l \le \la}|\mathfrak
m_{\sqrt{l}}^{(r)}(\eta)|+\kappa(d, \la)^{-4},
\end{align}
where $\eta=(\xi_1,\ldots,\xi_r)$. By \eqref{eq:lem:13_2} we have $\kappa(d,\la)\ge \sqrt{C} r^{1/\delta} \varepsilon^{-3/2}$,
hence $\varepsilon^3 \kappa(d,\la)^2 r\ge C r^{1+2/\delta}\ge C r^3$.
Thus, for $l \ge\varepsilon^3\kappa(d, \la)^2r$ we are allowed to
apply Proposition \ref{prop:mtinfla} in dimension $r$ to each of the
multipliers $ m_{\sqrt{l}}^{(r)}(\eta).$ In view of
\eqref{eq:35} we thus have
\[
|\mathfrak m_t(\xi)|\lesssim_{\delta,\varepsilon} 
\sup_{\varepsilon^3\kappa(d, \la)^2r\le l \le \la} \big( e^{-2\pi \kappa(r,l)\|\eta\|}+ e^{-2\pi \kappa(r,l)\|\eta +\vo/2\|}+ l^{-2}+e^{-cr}+\kappa(d, \lambda)^{-4} \big).
\]
Recalling \eqref{eq:lem:13_2} and noting that 
for $l \ge \varepsilon^3\kappa(d, \la)^2r$ we have $\kappa(r, l) \ge \varepsilon^{3/2} \kappa(d, \la)$ and 
$l \gtrsim_{\varepsilon} \kappa(d, \la)^2$, we see that
the above inequality leads to \eqref{eq:lem:13_1} as desired.
\end{proof}

\subsection{All together: proof of Proposition \ref{prop:2}} We have
prepared all necessary tools to prove Proposition \ref{prop:2}. We
shall be working under the assumptions of Lemma \ref{lem:13} with
$\delta=2/7$ and $\varepsilon=1/50$. While proving \eqref{eq:23} we
may assume that
$\kappa(d,\la)\ge \max \{ 10, 84^{1/\delta}50^{3/2}\sqrt{C} \}$, where $C$ is a constant from Lemma~\ref{lem:13}, as in the
other case the inequality is obvious. Then the assumption
\eqref{eq:lem:13_0} on the integer $r$ becomes
\begin{equation}
\label{eq:prop:2:asumr}
21 \le 50^{-\frac{3\delta}{2}}\kappa(d, \la)^{\delta} C^{-\delta/2}/4\le r\le 50^{-\frac{3\delta}{2}}\kappa(d,\la)^{\delta} C^{-\delta/2} \le d;
\end{equation}
the last inequality is a consequence of the assumption $\la \le d^3$ which forces $\kappa(d,\la) \le d$.

We shall also require a variant of the convexity inequality from \cite[Lemma 2.6]{BMSW2}. The difference between Lemma \ref{lem:7} and \cite[Lemma 2.6]{BMSW2} lies in the fact that we do not require $u_1,\ldots,u_d$ to be monotone. 

\begin{lemma}[{cf. \cite[Lemma 2.6]{BMSW2}}]
\label{lem:7}
Assume that we have a sequence $(u_j: j\in\NN_d)$ with 
$0\le u_j\le(1-\delta_0)/2$ for some
$\delta_0\in(0, 1)$. Suppose that $I\subseteq\NN_d$ satisfies
$\delta_1d\le |I|\le d$ for some $\delta_1\in(0, 1]$.  Then for every
$J=(d_0, d]\cap\ZZ$ with $0\le d_0\le d$ we have
\begin{align*}
\mathbb E\Big[\exp\Big({-\sum_{j\in\tau(I)\cap J}u_j}\Big)\Big]
\le 3\exp\Big({-\frac{\delta_0\delta_1}{20}\sum_{j\in J}u_j}\Big).
\end{align*}
\end{lemma}
\begin{proof}
Denote $\tilde{u}=(\tilde{u}_1,\ldots,\tilde{u}_d)$ with $\tilde{u}_j=u_j$ for  $j\in J$ and $\tilde{u}_j=0$ for $j\in \NN_d\setminus J,$ Let $v=(v_1,\ldots,v_d)$ be a non-increasing rearangment of $\tilde{u}.$  Since $\tilde{u}=\tilde{\tau}\circ v$ for some $\tilde{\tau}\in {\rm Sym}(d)$ we see that
$$\mathbb E\Big[\exp\Big({-\sum_{j\in\tau(I)\cap J}u_j}\Big)\Big]=\mathbb E\Big[\exp\Big({-\sum_{j\in\tau(I)}\tilde{u}_j}\Big)\Big]=\mathbb E\Big[\exp\Big({-\sum_{j\in\tau(I)}v_j}\Big)\Big].$$
Clearly $0\le v_d\le\ldots\le v_1\le (1-\delta_0)/2,$ hence, applying  \cite[Lemma 2.6]{BMSW2} with $J=\NN_d$ we obtain
$$\mathbb E\Big[\exp\Big({-\sum_{j\in\tau(I)\cap J}u_j}\Big)\Big]\le  3\exp\Big({-\frac{\delta_0\delta_1}{20}\sum_{j\in \NN_d}v_j}\Big)=3\exp\Big({-\frac{\delta_0\delta_1}{20}\sum_{j\in J}u_j}\Big).$$
and the proof is completed.
\end{proof}

\begin{proof}[Proof of Proposition \ref{prop:2}]
	Letting
	\begin{align*}
		r_0:=\lfloor50^{-\frac{3\delta}{2}}\kappa(d, \la)^{\delta} C^{-\delta/2}/4\rfloor+1,
	\end{align*}
	we see that any $r\in[r_0,2r_0]$ still satisfies \eqref{eq:prop:2:asumr}. Fix $\xi \in \TT^d.$ Clearly, there exist $\tau, \theta \in {\rm Sym}(d)$ such that
	\begin{align*}
		\|\xi_{\tau(1)}\|\ge \|\xi_{\tau(2)}\|\ge \ldots\ge \|\xi_{\tau(d)}\|
		\quad{\rm and}\quad
		\|\xi_{\theta(1)}+1/2\|\ge \|\xi_{\theta(2)}+1/2\|\ge \ldots\ge \|\xi_{\theta(d)}+1/2\|.
	\end{align*} 	
	Let $$I:=\tau(\NN_{r_0})\cup \theta (\NN_{r_0}),\qquad r:=|I|$$ so that $r_0\le r\le 2r_0.$ Since both sides of \eqref{eq:23} are invariant under the permutation group ${\rm Sym}(d)$ without loss of generality we may assume that $I=\NN_{r}.$ Thus Lemma \ref{lem:13} gives
	\begin{equation}
		\label{eq:prop:2:lem:13_1}
		|\mathfrak m_t(\xi)|\lesssim_{\delta,\varepsilon}\kappa(d, \la)^{-4}+e^{-2\pi \varepsilon^{3/2}\kappa(d,\la)\|\eta\|}+ e^{-2\pi \varepsilon^{3/2}\kappa(d,\la)\|\eta+\vo/2\|},
	\end{equation}
	where $\eta=(\xi_1,\ldots,\xi_r).$ Note that $\NN_r$ contains at least $r_0$ largest numbers among the elements of both $\{\|\xi_j\|: j\in \NN_d\}$ and $\{\|\xi_j+1/2\|:j\in \NN_d\}$, which   means that
	\begin{equation}
		\label{eq:INr}
		\begin{gathered}
			\big|\big\{j\in \NN_r\colon \|\xi_j\|\ge \max_{i\in\NN_d\setminus\NN_r}\|\xi_i\|\big \}\big|\ge r_0,\\
			\big|\big\{j\in \NN_r\colon \|\xi_j+1/2\|\ge \max_{i\in\NN_d\setminus\NN_r}\|\xi_i+1/2\|\big \}\big|\ge r_0.
		\end{gathered}
	\end{equation}
	
	In the proof we shall need a  modification of $V_{\xi}$ from \eqref{eq:107} given by
	\begin{equation*}
		V_{\eta}:=\{i\in \NN_r\colon \cos(2\pi\eta_i)<0\}=\{i\in \NN_r\colon 1/4<\|\eta_i\|\le1/2\} \quad \text{for} \quad \eta\in\TT^{r}.
	\end{equation*} 
	We will consider two cases: either $|V_{\eta}|\le r/2$ or $|V_{\eta}|>r/2.$ 
	
	Firstly assume that $|V_{\eta}|\le r/2.$ Here our goal is to prove that
	\begin{equation}
		\label{eq:prop:2:V_1}
		|\mathfrak m_t(\xi)|\lesssim \kappa(d,\la)^{-4}+(\kappa(d,\la)\|\xi\|)^{-1},
	\end{equation}
	which clearly implies \eqref{eq:23}.  If $|V_{\eta}|\le r/2$ then at
	least $r/2$ numbers among the elements $\{\|\eta_j\|: j\in \NN_r\}$  are
	smaller than $1/4,$ so that $\|\eta+\vo/2\|\geq 2^{-5/2} r^{1/2}.$ In
	this case \eqref{eq:prop:2:lem:13_1} implies
	\begin{equation}
		\label{eq:13_1'}
		|\mathfrak m_t(\xi)|
		\lesssim \kappa(d, \la)^{-4}+\exp(-2\pi 50^{-3/2}\kappa(d,\la)\|\eta\|).
	\end{equation}
	Suppose for a moment that $\|\eta\| ^2\ge\frac{1}{4}\|\xi\|^2.$
	Then \eqref{eq:13_1'} implies \eqref{eq:prop:2:V_1}. Thus we can assume that
	\begin{align}
		\label{eq:54}
		\|\xi_1\|^2+\ldots+\|\xi_r\|^2\le\frac{1}{4}\|\xi\|^2.  
	\end{align}
	
	Let $\varepsilon_1=1/10$
	and assume first that
	\begin{align}
		\label{eq:55}
		\|\xi_j\|\le\frac{\varepsilon_1^{1/2}}{10\kappa(d,\la)}\quad\text{ for all } \quad r <
		j\le d.
	\end{align}
	Using \eqref{eq:43} and the Cauchy--Schwarz inequality we have
	\begin{align*}
		|\mathfrak m_t(\xi)|^2\le \frac{1}{|S_t\cap\ZZ^d|}\sum_{x\in S_t\cap\ZZ^d}\exp\Big(-\sum_{j=r+1}^d\sin^2(2\pi x_j \xi_j)\Big).
	\end{align*}
	For $x\in S_t\cap\ZZ^d$ we define
	\begin{align*}
		I_x&:=\{i\in\NN_d\colon \varepsilon \kappa(d,\la) \le |x_i|\le 2\varepsilon_1^{-1/2}\kappa(d,\la) \},\\
		E&:=\big\{x\in S_t\cap\ZZ^d\colon |I_x|\ge\varepsilon_1 d/2\big\}.
	\end{align*}
	Then by Lemma \ref{lem:5} (with $\varepsilon_2=\varepsilon$)
	we obtain $  |E^{\bf c}|\lesssim e^{-\frac{d}{11}}|S_t\cap\ZZ^d|.$
	Therefore,
	\begin{align}
		\label{eq:57}
		\begin{split}
			|\mathfrak m_t(\xi)|^2
			&\lesssim \frac{1}{|S_t\cap\ZZ^d|}\sum_{x\in
				E}\exp\Big(-\sum_{j\in I_x\cap J_r}\sin^2(2\pi x_j
			\xi_j)\Big) +e^{-\frac{d}{11}},
		\end{split}
	\end{align}
	where $J_r=\NN_d\setminus\NN_r$. From \eqref{eq:55} and the definition
	of $I_x$ we see that
	$\sin^2(2\pi x_j \xi_j)\ge 4 \|2 x_j \xi_j \|^2 \ge 16\varepsilon^2\kappa(d,\la)^2\|\xi_j\|^2$ for $j \in I_x$.
	Consequently, we obtain
	\begin{align}
		\label{eq:58}
		\begin{split}
			\frac{1}{|S_t\cap\ZZ^d|}&\sum_{x\in E}
			\exp\Big(-\sum_{j\in I_x\cap J_r}\sin^2(2\pi x_j\xi_j)\Big) \\
			&\le \frac{1}{|S_t\cap\ZZ^d|}\sum_{x\in E}
			\exp\Big(-16\varepsilon^2\kappa(d,\la)^2\sum_{j\in I_x\cap J_r}\|\xi_j\|^2\Big)
			\lesssim e^{-c\kappa(d,\la)^2\|\xi\|^2},
		\end{split}
	\end{align}
	where $c>0$ is a universal constant. In order to obtain the last
	inequality in \eqref{eq:58} we use the fact that
	$\gamma(E)=E$ and $I_{\gamma^{-1} \circ x} = \gamma(I_x)$ for every
	$\gamma\in{\rm Sym}(d)$ and $x \in S_t$, and we
	apply Lemma \ref{lem:7} with
	$\delta_1=\varepsilon_1/2$, $d_0=r$, $I=I_x$ and $\delta_0=3/5$, to
	conclude that
	\begin{align*}
		\mathbb E\bigg[\exp\Big(-16\varepsilon^2\kappa(d,\la)^2\sum_{j\in \gamma(I_x)\cap
			J_r}\|\xi_j\|^2\Big)\bigg]\lesssim \exp\Big(-c'\kappa(d,\la)^2\sum_{j=r+1}^d\|\xi_j\|^2\Big),
	\end{align*}
	for a universal constant $c'>0$ and for all $x\in E$.
	This proves \eqref{eq:58} since by \eqref{eq:54} we obtain
	\begin{align*}
		\exp\Big(-c'\kappa(d,\la)^2\sum_{j=r+1}^d\|\xi_j\|^2\Big)\le \exp\Big(-\frac{3c'\kappa(d,\la)^2}{4}\sum_{j=1}^d\|\xi_j\|^2\Big).
	\end{align*}
	Coming back to \eqref{eq:57} we have thus proved that
	$$|\mathfrak m_t(\xi)|\lesssim (\kappa(d,\la)\|\xi\|)^{-1}+e^{-d/22}\lesssim \kappa(d,\la)^{-4}+ (\kappa(d,\la)\|\xi\|)^{-1}.$$
	Therefore, \eqref{eq:prop:2:V_1} is true under the assumption \eqref{eq:55}.
	
	Assume now that \eqref{eq:55} does not hold, i.e.\ that for some
	$j\in \NN_d\setminus\NN_r$ we have
	$ \|\xi_j\|\ge\frac{\varepsilon_1^{1/2}}{10\kappa(d,\la)}.$ Then,
	using \eqref{eq:INr} we see that
	\begin{align*}
		\|\eta\|^2\ge\frac{\varepsilon_1r_0}{100\kappa(d,\la)^2}\ge \frac{\varepsilon_1r}{200\kappa(d,\la)^2} .
	\end{align*}
	Thus, invoking \eqref{eq:13_1'} and recalling that $\varepsilon_1=1/10$ we obtain 
	\begin{align*}
		|\mathfrak m_t(\xi)|&\lesssim\kappa(d,\la)^{-4}+\exp(-2\pi 50^{-3/2}\kappa(d,\la)\|\eta\|)\leq \kappa(d,\la)^{-4}+\exp(-\pi 5^{-3/2} 10^{-9/2} \sqrt{r})\lesssim \kappa(d,\la)^{-4},
	\end{align*}
	since $r\simeq\kappa(d,\la)^{\frac{2}{7}}$. This completes the proof of \eqref{eq:prop:2:V_1}.

	It remains to consider the case $|V_{\eta}|>r/2$. We prove that
	\begin{equation}
		\label{eq:prop:2:V_2}
		|\mathfrak m_t(\xi)|\lesssim \kappa(d,\la)^{-4}+(\kappa(d,\la)\|\xi+\vo/2\|)^{-1}.
	\end{equation}
	The above bound clearly implies \eqref{eq:23}. Actually
	\eqref{eq:prop:2:V_2} may be easily deduced from the previous case
	\eqref{eq:prop:2:V_1}. Indeed, by \eqref{eq:mtsym} we have
	$|\mathfrak m_t(\xi)|=|\mathfrak m_t(\xi+\vo/2)|.$ Now, for the vector
	$\tilde{\xi}:=\xi+\vo/2$ we have $\tilde{\eta}=\eta+\vo/2$ and
	$|V_{\tilde{\eta}}|\le r/2.$ Moreover, \eqref{eq:INr} is still
	satisfied with $\tilde{\xi}$ in place of $\xi.$ Therefore, we can
	repeat the argument used in the proof of \eqref{eq:prop:2:V_1} with
	$\tilde{\xi}$ replacing $\xi.$ This leads to \eqref{eq:prop:2:V_2} as desired.
	
	The proof of Proposition \ref{prop:2} is thus completed. 
\end{proof}

\section{Proof of Theorem \ref{thm:00}: small-scale estimate \eqref{eq:ss}}
\label{sec:sm}
This section is devoted to the proof of inequality \eqref{eq:ss} in Theorem \ref{thm:00}.  We shall
proceed in much the same way as in the corresponding case for the
discrete Euclidean balls from \cite{BMSW2}. Therefore, we only briefly
point out the main differences. The strategy of the proof of inequality \eqref{eq:ss} is similar
to the proof of inequality \eqref{eq:is}. The approximating multiplies
$p^1_{\lambda}(\xi)$ and $p^2_{\lambda}(\xi)$ (see \eqref{eq:96},
\eqref{eq:97}) depend on the size of the set $V_{\xi}$ defined in
\eqref{eq:107}.  Proposition \ref{prop:4}, which is the main results
of this section, is a variant of Proposition \ref{prop:4p} adjusted to
the small scales.

\begin{proposition}
\label{prop:4}
Let $d\ge 5$ and assume that $\kappa(d,\la)\le 1/5.$ Then for every
$\la\in\NN$ and $\xi\in\TT^d$ we have the following bounds with the
constant $c\in(0, 1)$ as in \eqref{eq:70}. Namely,
\begin{enumerate}[label*={\arabic*}.]
\item if $|V_{\xi}|\le d/2$, then
\begin{align}
\label{eq:94}
|\mathfrak m_t(\xi)-p^1_{\lambda}(\xi)|
\lesssim
\min\Big\{e^{-\frac{c\kappa(d,\la)^2}{400}\sum_{i=1}^d\sin^2(\pi\xi_i)}, \kappa(d,\la)^2\sum_{i=1}^d\sin^2(\pi\xi_i)\Big\},
\end{align}
\item if $|V_{\xi}|\ge d/2$, then
\begin{align}
\label{eq:95}
|\mathfrak m_t(\xi)-p^2_{\lambda}(\xi)|
\lesssim
\min\Big\{e^{-\frac{c\kappa(d,\la)^2}{400}\sum_{i=1}^d\cos^2(\pi\xi_i)}, \kappa(d,\la)^2\sum_{i=1}^d\cos^2(\pi\xi_i)\Big\}.
\end{align}
\end{enumerate}
\end{proposition}

Assuming momentarily Proposition \ref{prop:4} we can now deduce
inequality \eqref{eq:ss} in Theorem \ref{thm:00}. The argument is essentially the same as in the
proof of inequality \eqref{eq:is}. The only
difference is the use of Proposition \ref{prop:4} in place of
Proposition \ref{prop:4p}, however, we give it for the sake of
completeness.

\begin{proof}[Proof of inequality \eqref{eq:ss} in Theorem \ref{thm:00}]
 Let $f\in\ell^2(\ZZ^d)$ and we write $f=f_1+f_2$, where
$\hat{f}_1(\xi):=\hat{f}(\xi)\ind{\{\eta\in\TT^d\colon |V_{\eta}|\le d/2\}}(\xi)$. Then
\begin{align*}
\big\|\sup_{t\in\mathbb D_{C_0}}|\mathcal F^{-1}(\mathfrak m_t\hat{f})|\big\|_{\ell^2(\ZZ^d)}
\le \sum_{i=1}^2\big\|\sup_{t\in\mathbb D_{C_0}}|\mathcal F^{-1}(p_{t^2}^i\hat{f}_i)|\big\|_{\ell^2(\ZZ^d)}
+\sum_{i=1}^2\Big\|\Big(\sum_{t\in\mathbb D_{C_0}}\big|\mathcal F^{-1}\big((\mathfrak m_t - p_{t^2}^i)\hat{f}_i\big)\big|^2\Big)^{1/2}\Big\|_{\ell^2(\ZZ^d)}.
\end{align*}
The usual square function argument permits  to reduce the problem to bounding the maximal functions associated with the
multipliers $p^1_{\lambda}$ and  $p^2_{\lambda}$. Taking $C_0=1/5,$ Plancherel's theorem and Proposition \ref{prop:4} we obtain 
\[
\sum_{i=1}^2\Big\|\Big(\sum_{t\in\mathbb D_{C_0}}\big|\mathcal F^{-1}\big((\mathfrak m_t - p_{t^2}^i)\hat{f}_i\big)\big|^2\Big)^{1/2}\Big\|_{\ell^2(\ZZ^d)}
\lesssim  \|f_1\|_{\ell^2(\ZZ^d)}+\|f_2\|_{\ell^2(\ZZ^d)}\le 2\|f\|_{\ell^2(\ZZ^d)}.
\]
It remains to bound the maximal functions corresponding to the
multipliers $p_{t^2}^1$ and $p_{t^2}^2$. This was already done in
\eqref{eq:99p} and \eqref{eq:100p}, and is a simple consequence of \eqref{eq:47}. Hence, the proof of inequality \eqref{eq:ss} in Theorem \ref{thm:00} is completed.
\end{proof}

The rest of Section \ref{sec:sm} is devoted to the proof of Proposition \ref{prop:4}.

\subsection{Some preparatory estimates}
We shall need a version of \cite[Lemma 3.2]{BMSW2} with discrete spheres in place of  discrete balls.

\begin{lemma}
\label{lem:15}
For every $d, \la\in\NN$, $d\ge 5,$ if $\kappa(d,\la)\le 1/5$ and  $\la\ge k\ge2^9\max\{1, \kappa(d,\la)^6 \la\}$, then
\begin{align}
\label{eq:83}
|\{x\in S_{t}\cap\ZZ^d\colon |\{i\in\NN_d\colon x_i=\pm1\}|\le \la-k\}|\le (2t+1)^4 2^{-k+1}|S_{t}\cap\ZZ^d|.
\end{align}
\end{lemma}
\begin{proof}	
We define $A:=\{ x\in \ZZ^d\colon|\{i\in\NN_d\colon x_i=\pm1\}|\le \la-k\}.$ Then, by \cite[Lemma 3.2]{BMSW2} (second inequality below) and Lemma \ref{lem:lfs} (last inequality below) we obtain
\begin{align*}
|A\cap S_{t}\cap\ZZ^d|\le |A\cap B_t^2(d)\cap\ZZ^d|\le 2^{-k+1}|B_t^2(d)\cap\ZZ^d|\le 2^{-k+1}(2t+1)^4|S_{t}\cap\ZZ^{d}|,
\end{align*}
and the proof of \eqref{eq:83} is completed.
\end{proof}

Lemma \ref{lem:15} will be essential in our next result. 

\begin{proposition}
\label{prop:5}
Let $d\ge 5$ and assume that $\kappa(d,\la)\le 1/5.$ Then, for every
$\la\in\NN$ and $\xi \in \TT^d$ we have the following estimate
\begin{align}
\label{eq:98}
|\mathfrak m_t(\xi)|\lesssim e^{-\frac{c\kappa(d,\la)^2}{100}\sum_{i=1}^d\sin^2(\pi\xi_i)}
+e^{-\frac{c\kappa(d,\la)^2}{100}\sum_{i=1}^d\cos^2(\pi\xi_i)},
\end{align}
where $c\in(0, 1)$ is the absolute constant from \eqref{eq:70}.
\end{proposition}

\begin{proof}
Note that when $\la \le 2^{10}$ then
$\kappa(d,\la)^2=\la/d\lesssim d^{-1}$ and \eqref{eq:98} is obvious.
Thus, in what follows we assume that $\la\ge2^{10}.$ For any $x\in\ZZ^d$ we
define the
sets
\begin{align*}
I_x:=\{i\in\NN_d\colon x_i=\pm1\}
\qquad \text{ and }\qquad
E:=\{x\in S_{t}\cap\ZZ^d\colon |I_x|>\la/2\}.
\end{align*}
Since $\la =\kappa(d,\la)^2 d,$ using Lemma \ref{lem:15}, with
$k= \lambda - \lfloor \la/2\rfloor$, we see that
\begin{align*}
|E^{\bf c}|&\le  (2t+1)^4 2^{-\la/2+2}|S_{t}\cap\ZZ^d|\lesssim 2^{-3\la/8}|S_{t}\cap\ZZ^d|   \leq e^{-\frac{\kappa(d,\la)^2d}{4}}|S_{t}\cap\ZZ^d|.
\end{align*}
In view of these estimates it now suffices to show that
\begin{align}
\label{eq:59}
\frac{1}{|S_{t}\cap\ZZ^d|}\Big|\sum_{x\in S_{t}\cap \ZZ^d\cap E}
\prod_{j=1}^d \cos(2\pi x_j \xi_j)\Big|\lesssim
e^{-\frac{c\kappa(d,\la)^2}{100}\sum_{i=1}^d\sin^2(\pi\xi_i)}
+e^{-\frac{c\kappa(d,\la)^2}{100}\sum_{i=1}^d\cos^2(\pi\xi_i)}.
\end{align}

As in \cite[Section 3]{BMSW2} the proof of \eqref{eq:59} will rely on
the properties of the Krawtchouk polynomials. For the convenience of
the reader we recall their definitions and basic properties.  For every $n\in\NN_0$ and integers
$x, k\in[0, n]$ we define the $k$-th Krawtchouk polynomial
\begin{align}
\label{eq:38}
\Bbbk_k^{(n)}(x):=\frac{1}{\binom{n}{k}}\sum_{j=0}^k(-1)^j\binom{x}{j}\binom{n-x}{k-j}.
\end{align}
We gather properties of Krawtchouk polynomials required to establish \eqref{eq:59}.
\begin{theorem}
\label{thm:100}
For every $n\in\NN_0$ and integers  $x, k\in[0, n]$ we have
\begin{enumerate}[label*={\arabic*}.]

\item \label{item:1}{\it Symmetry:} $\Bbbk_k^{(n)}(x)=\Bbbk_x^{(n)}(k)$.

\item \label{item:2}{{\it Reflection symmetry:}}
$\Bbbk_k^{(n)}(n-x)=(-1)^k\Bbbk_k^{(n)}(x)$.
\item \label{item:5}{\it A uniform bound:} there exists a constant $c\in(0, 1)$ such that
for all $n\in\NN_0$ the following inequality
\begin{align}
  \label{eq:70}
  \big|\Bbbk_k^{(n)}(x)\big|\le e^{-\frac{ckx}{n}}
\end{align}
holds for all integers $0\le x, k\le n/2$.

\end{enumerate}
\end{theorem}
The proof of Theorem \ref{thm:100} can be found in \cite{HKS}, see
also the references therein. It turns out that the left-hand side of
\eqref{eq:59} is essentially the Krawtchouk polynomial  (with appropriate
parameters $k, n$). In order to see this reduction one has to repeat
the proof of inequality \cite[Section 3, (3.20)]{BMSW2}
with $S_t$ in place of $B_N$. Once this is done we can easily deduce
\eqref{eq:59} with $c\in(0, 1)$ as in \eqref{eq:70}.
\end{proof}

\subsection{All together}
We are now ready to prove Proposition \ref{prop:4}.
\begin{proof}[Proof of Proposition \ref{prop:4}]
Firstly, we assume that $|V_{\xi}|\le d/2$, then \eqref{eq:98} implies that
\begin{align*}
	|\mathfrak m_t(\xi)|\lesssim e^{-\frac{c\kappa(d,\la)^2}{400}\sum_{i=1}^d\sin^2(\pi\xi_i)},
\end{align*}
since $\sum_{i=1}^d\cos^2(\pi\xi_i)\ge d/4\ge 1/4\sum_{i=1}^d\sin^2(\pi\xi_i)$. Thus
\begin{align}
\label{eq:110}
\Big|\mathfrak m_t(\xi)-e^{-\kappa(d,\la)^2\sum_{i=1}^d\sin^2(\pi\xi_i)}\Big|\lesssim e^{-\frac{c\kappa(d,\la)^2}{400}\sum_{i=1}^d\sin^2(\pi\xi_i)}.
\end{align}
On the other hand, using \eqref{eq:22} from Proposition \ref{prop:0} we obtain 
\begin{align}
\label{eq:111}
\Big|\mathfrak m_t(\xi)-e^{-\kappa(d,\la)^2\sum_{i=1}^d\sin^2(\pi\xi_i)}\Big|\lesssim \kappa(d,\la)^2\sum_{i=1}^d\sin^2(\pi\xi_i).
\end{align}
We now see that \eqref{eq:110} and \eqref{eq:111} imply \eqref{eq:94}.

Secondly, we assume that $|V_{\xi}|\ge d/2$, then \eqref{eq:98} implies that
\begin{align*}
	|\mathfrak m_t(\xi)|\lesssim e^{-\frac{c\kappa(d,\la)^2}{400}\sum_{i=1}^d\cos^2(\pi\xi_i)},
\end{align*}
since $\sum_{i=1}^d\sin^2(\pi\xi_i)\ge d/4\ge 1/4\sum_{i=1}^d\cos^2(\pi\xi_i)$. Thus
\begin{align}
\label{eq:112}
\Big|\mathfrak m_t(\xi)-(-1)^\la
e^{-\kappa(d,\la)^2\sum_{i=1}^d\cos^2(\pi\xi_i)}\Big| \lesssim e^{-\frac{c\kappa(d,\la)^2}{400}\sum_{i=1}^d\cos^2(\pi\xi_i)}.
\end{align}
On the other hand, \eqref{eq:22'} Proposition \ref{prop:0} implies
\begin{align}
\label{eq:113}
  \Big|\mathfrak m_t(\xi)-(-1)^\la
  e^{-\kappa(d,\la)^2\sum_{i=1}^d\cos^2(\pi\xi_i)}\Big|
  \lesssim \kappa(d,\la)^{2}\sum_{i=1}^d\cos^2(\pi\xi_i).
\end{align}
We now see that \eqref{eq:112} and \eqref{eq:113} imply \eqref{eq:95}. This completes the proof of the proposition.
\end{proof}


\begin{thebibliography}{99}
\bibitem{Ald1} J.M. Aldaz, \textit{The weak type $(1, 1)$ bounds for the
  maximal function associated to cubes grow to infinity with the
  dimension}, 
	Ann.  Math.  {\bf 173}, (2011), 1013--1023.
	
\bibitem{AM} 
T. Anderson, J. Madrid, 
\textit{New bounds for discrete lacunary spherical averages}.
Preprint 2020, arXiv:2001.11557.
  

\bibitem{Bou0} 
J. Bourgain,
\textit{Averages in the plane over convex curves and maximal operators}, 
J{.} Anal{.} Math{.} {\bf 47}, (1986),  69--85.

  
\bibitem{B1} 
J. Bourgain, 
\textit{On high dimensional maximal
  functions associated to convex bodies}, 
	Amer. J.  Math. {\bf 108}, (1986), 1467--1476.
        
\bibitem{B2} 
J. Bourgain, 
\textit{On $L^p$ bounds for maximal
  functions associated to convex bodies in $\RR^n$}, 
	Israel J. Math. {\bf 54}, (1986), 257--265.

\bibitem{B3} 
J. Bourgain, 
\textit{On the Hardy-Littlewood maximal
  function for the cube}, 
	Israel J. Math.  {\bf 203}, (2014), 275--293.

\bibitem{BMSW1} 
J. Bourgain, M. Mirek, E.M. Stein, B. Wr\'obel,
\textit{Dimension-free variational estimates on $L^p(\RR^d)$ for
  symmetric convex bodies}, 
	Geom. Funct. Anal.
  {\bf 28}, (2018),  58--99.

\bibitem{BMSW2} 
J. Bourgain, M. Mirek, E.M. Stein, B. Wr\'obel,
\textit{On discrete Hardy--Littlewood maximal functions 
over the balls in
$\ZZ^d$: dimension-free estimates}, 
Geometric Aspects of Functional Analysis -- Israel Seminar (GAFA) 2017-2019, Lecture Notes in Mathematics 2256.

  \bibitem{BMSW3} J. Bourgain, M. Mirek, E.M. Stein, B. Wr\'obel,
  \textit{Dimension-free estimates for discrete Hardy--Littlewood
  	averaging operators over the cubes in
  	$\mathbb Z^d$}, 
		Amer. J. Math. {\bf 141}, (2019), 857--905.


  \bibitem{BMSW4} 
	J. Bourgain, M. Mirek, E.M. Stein, B. Wr\'obel,
  \textit{On the Hardy--Littlewood maximal functions in high
  dimensions: Continuous and discrete perspective}, Preprint 2018.


\bibitem{Cal} 
C.P. Calder{\'o}n,
\textit{Lacunary spherical means},
Illinois J. Math. {\bf 23}, (1979), 476--484.
  
\bibitem{Car1} 
A. Carbery, 
\textit{An almost-orthogonality principle
  with applications to maximal functions associated to convex bodies},
Bull. Amer. Math. Soc.  {\bf 14}, (1986), 269--274.

\bibitem{CW} 
R.R. Coifman, G. Weiss,
\textit{Review: R.E. Edwards and G.I. Gaudry, Littlewood--Paley and multiplier theory},
Bull. Amer. Math. Soc. {\bf 84}, (1978), 242--250.

\bibitem{C1} 
B. Cook, 
\textit{Maximal function inequalities and a theorem of Birch}, 
Israel J{.} Math{.} {\bf 231}, (2019), 211--241.

\bibitem{C3} 
B. Cook,
\textit{A note on discrete spherical averages over sparse sequences}. 
Preprint 2018, arXiv:1808.03822.

\bibitem{CH} 
B. Cook, K. Hughes,
\textit{Bounds for lacunary maximal functions given by Birch-Magyar averages}. 
Preprint 2019, arXiv:1905.09189, to appear in Trans. Amer. Math. Soc.


\bibitem{DGM1} 
L. Deleaval, O. Gu\'edon, B. Maurey,
\textit{Dimension-free bounds for the Hardy--Littlewood maximal
  operator associated to convex sets},
	 Ann. Fac. Sci. Toulouse Math. {\bf 27}, (2018), 1--198.

\bibitem{Gra_C} 
L. Grafakos, 
\textit{Classical Fourier Analysis}, third edition, Graduate texts in Mathematics, 2014.

\bibitem{HKS}
A.W. Harrow, A. Kolla, L.J. Schulman.
\textit{Dimension--free $L_2$ maximal inequality for
spherical means in the hypercube},
Theory of Computing {\bf 10}, (2014), 55--75.


\bibitem{H1} 
K. Hughes,
\textit{The discrete spherical averages over a family of sparse sequences}, 
J{.} Anal{.} Math{.} {\bf 138}, (2019), 1--21.


\bibitem{I1} 
A.D. Ionescu,
\textit{An endpoint estimate for the discrete spherical maximal function}, 
Proc{.} Amer{.} Math{.} Soc{.} {\bf 132}, (2004), 1411--1417.


\bibitem{IK} 
H. Iwaniec, E. Kowalski,
\textit{Analytic Number Theory}.
Vol. 53, Amer. Math. Soc. Colloquium
Publications, Providence RI, (2004).

\bibitem{KLM} 
R. Kesler, M. Lacey, D. Mena,
\textit{Lacunary discrete spherical maximal functions},
New York J{.} Math{.} {\bf 25}, (2019), 541--557.

\bibitem{KLM2} 
R. Kesler, M. Lacey, D. Mena,
\textit{Sparse bounds for the discrete spherical maximal functions}, 
Pure Appl{.} Anal{.} {\bf 2}, (2020), 75--92.

\bibitem{KMPW} 
D. Kosz, M. Mirek, P. Plewa, B. Wr\'obel, 
\textit{Dimension--free estimates for the Hardy-Littlewood maximal function over the $q$-balls on $\ZZ^d$}. 
Preprint 2020, arxiv:2010.07379. 



\bibitem{Kov}
O. Kovrizhkin,
\textit{On the norms of discrete analogues of convolution operators.}
Proc. Amer. Math. Soc. {\bf 140} (2012), no. 4, 1349--1352.

\bibitem{Ma} 
\'A. Magyar,
\textit{$L^p$-bounds for spherical maximal operators on $\mathbb{Z}^n$}, 
Rev{.} Mat{.} Iberoamericana {\bf 13}, (1997), 307--317.

\bibitem{MSW}  
\'A. Magyar, E.M. Stein, S. Wainger,
\textit{Discrete Analogues in Harmonic Analysis: Spherical Averages}, 
Ann. of Math. {\bf 155}, (2002), 189--208.


\bibitem{Mul1} 
D. M\"uller, 
\textit{A geometric bound for maximal
  functions associated to convex bodies}, 
	Pacific J. Math. {\bf 142}, (1990), 297--312.

\bibitem{Nat} 
M.B. Nathanson,
\textit{Additive Number Theory. The Classical Bases}.
Springer--Verlag, 1996.



\bibitem{NIST} 
F.W.J. Olver,  D.W. Lozier, R.F. Boisvert, C.W. Clark (editors), 
\textit{NIST Handbook of Mathematical Functions}, U.S. Department of Commerce, National Institute of Standards and Technology, Washington, DC; Cambridge University Press, Cambridge, 2010. xvi+951 pp.

\bibitem{Ste0}
E{.}M{.} Stein, 
\textit{Maximal functions. I. Spherical means}, 
Proc{.} Nat{.} Acad{.} Sci{.} U{.}S{.}A{.} {\bf 73}, (1976), 2174--2175.


\bibitem{Ste1} 
E.M. Stein, 
\textit{Topics in harmonic analysis
  related to the Littlewood--Paley theory}, 
	Annals of Mathematics
Studies, Princeton University Press 1970, 1--157.


\bibitem{SteinMax} 
E.M. Stein, 
\textit{The development of square
  functions in the work of A. Zygmund}, 
	Bull.  Amer. Math. Soc. {\bf 7}, (1982), 359--376.

\bibitem{StStr} 
E.M. Stein, J.O. Str\"omberg, 
\textit{Behavior of maximal functions in $\mathbb{R}^n$ for large $n$}, 
Ark. Mat. {\bf 21}, (1983), 259--269.

\bibitem{Tao} T. Tao,
\textit{The Ionescu--Wainger multiplier theorem and the adeles}. 
Preprint 2020, arXiv:2008.05066.
 
\end{thebibliography}
\end{document}